\documentclass[sn-mathphys]{sn-jnl}% Math and Physical Sciences Reference Style
\usepackage{graphicx}
\usepackage{hyperref} 
\usepackage{amsfonts, amsthm}
\usepackage{amsmath}
\usepackage{xcolor}
\usepackage{hyperref}
\usepackage{epsfig}
\usepackage{epstopdf}
\usepackage{comment}
\usepackage{graphicx}
\usepackage{color}
\usepackage{url}
\usepackage{subfig}
\usepackage{lineno}
\numberwithin{equation}{section}

\jyear{2021}%

\theoremstyle{thmstyleone}%
\newtheorem{theorem}{Theorem}%  meant for continuous numbers
\newtheorem{proposition}{Proposition}% to get separate numbers for theorem and proposition etc.

\begin{document}

%\linenumbers

\title[Optimal anti-hormonal treatment for breast cancer]{Modeling of mouse experiments suggests that optimal anti-hormonal treatment for breast cancer is diet-dependent}
%\title[Article Title]{Article Title}


\author*[1,2]{\fnm{Tu\u{g}ba} \sur{Akman}} \email{takman@thk.edu.tr} 

\author[3]{\fnm{Lisa M.} \sur{Arendt}} \email{lmarendt@wisc.edu}

\author[4,5]{\fnm{J\"{u}rgen Geisler}} \email{jurgen.geisler@medisin.uio.no}

\author[6]{\fnm{Vessela N.} \sur{Kristensen}} \email{v.n.kristensen@medisin.uio.no}

\author[1,7]{\fnm{Arnoldo} \sur{Frigessi}} \email{arnoldo.frigessi@medisin.uio.no}

\author*[1,7]{\fnm{Alvaro} \sur{K\"ohn-Luque }} \email{alvaro.kohn-luque@medisin.uio.no} 

\affil[1]{\orgdiv{Oslo Centre for Biostatistics and Epidemiology}, \orgname{Faculty of Medicine, University of Oslo}, \city{Oslo}, \postcode{0317}, \country{Norway}}

\affil[2]{\orgdiv{Department of Computer Engineering}, \orgname{University of Turkish Aeronautical Association}, \orgaddress{\street{Etimesgut}, \city{Ankara}, \postcode{06790}, \country{Turkey}}}

\affil[3]{\orgdiv{Department of Comparative Biosciences}, \orgname{University of Wisconsin-Madison}, \country{U.S.A}}

\affil[4]{\orgdiv{Department of Oncology}, \orgname{Akershus University Hospital}, \city{Lørenskog}, \country{Norway}}

\affil[5]{\orgdiv{Institute of Clinical Medicine}, \orgname{Faculty of Medicine, University of Oslo}, \city{Campus AHUS}, \country{Norway}}

\affil[6]{\orgdiv{Department of Medical Genetics, Institute of Clinical Medicine}, \orgname{Oslo University Hospital and University of Oslo}, \orgaddress{\city{Oslo}, \country{Norway}}}
%Department of Medical Genetics, Institute of Clinical Medicine, Faculty of Medicine, University of Oslo, Oslo, Norway.

\affil[7]{\orgdiv{Oslo Centre for Biostatistics and Epidemiology}, \orgname{Oslo University Hospital}, \city{Oslo}, \country{Norway}}



%\affil[3]{\orgdiv{Department}, \orgname{Organization}, \orgaddress{\street{Street}, \city{City}, \postcode{610101}, \state{State}, \country{Country}}}


\abstract{
Estrogen receptor positive breast cancer is frequently treated with anti-hormonal treatment such as aromatase inhibitors (AI). Interestingly, a high body mass index has been shown to have a negative impact on AI efficacy, most likely due to disturbances in steroid metabolism and adipokine production. Here, we propose a mathematical model based on a system of ordinary differential equations to investigate the effect of high-fat diet on tumor growth. We inform the model with data from mouse experiments, where the animals are fed with high-fat or control (normal) diet. By incorporating AI treatment with drug resistance into the model and by solving optimal control problems we found differential responses for control and high-fat diet. To the best of our knowledge, this is the first attempt to model optimal anti-hormonal treatment for breast cancer in the presence of drug resistance. Our results underline the importance of considering high-fat diet and obesity as factors influencing clinical outcomes during anti-hormonal therapies in breast cancer patients.}




\maketitle

\section{Introduction}\label{intro}

Lifestyle factors such as age at menarche and menopause, body mass index, child birth and breast feeding, as well as genetic disposition, among others, are well-established breast cancer risk factors \cite{wu2016substantial, neuhouser2015overweight}.  However, much less is known about the role lifestyle factors play on breast cancer treatment response. Anti-hormonal treatment for estrogen receptor (ER) positive breast cancer constitutes a puzzling case in obese patients that requires more quantitative investigation. Approximately 75\% of all breast tumors express ER, and most women with these tumors will receive anti-hormonal therapy \cite{clark1984correlations}. ER in breast cancer cells is activated by estrogen and it promotes cell proliferation and tumor growth \cite{johnston2003aromatase}. Anti-hormonal treatment with Aromatase Inhibitors (AI) decreases estrogen levels while anti-estrogen's block directly the action of steroids at the estrogen receptor \cite{pearson1982antiestrogen}. Interestingly, high Body Mass Index (BMI) and adiposity have a negative impact on AI efficacy \cite{folkerd2012suppression, ioannides2014effect, jiralerspong2016obesity, bahrami2021lack, wang2015aromatase, gelsomino2020leptin, goodwin2010obesity, lonning2014relationship, sendur2012efficacy}. While the puzzle of the optimal anti-hormonal therapy in postmenopausal obese women is still unfinished, good monitoring of the suppression of estrogen levels with valid methods may guide treatment decisions during treatment with aromatase inhibitors \cite{bordeleau2010multicenter, ligibel2012risk}.
%[Folkerd et al 2012 J Clin Oncol, Ioannides et al 2014 Breast Cancer Res Treat, Wolters et al 2012 Breast Cancer Res Treat, Jiralerspong and Goodwing 2016 J Clin Oncol, Brown 2021 Nat Rev Endocrine]
%[Goodwing 2010 J Clin Oncol, Ligibel and Winer 2012 J Clin Oncol]

An additional layer of complexity arises from the fact that ER-positive breast cancer cells may be resistant to anti-hormonal treatments. Resistance can arise due to multiple mechanisms that are not completely understood \cite{daldorff2017cotargeting, ma2015mechanisms}. Tumor cells can adapt to AI therapy after exposure for certain time (adaptive resistance), for instance due to the upregulation of ER expression or activation of alternative pathways conferring the cells survival and proliferative capacity. Instead, de novo or pre-existing resistance refers to the presence of estrogen independent cells before therapy. For instance, cells carrying specific mutations of the ER that confer constitutive ligand-idependent activity \cite{jeselsohn2015esr1}, which might lead to clonal selection under anti-hormonal treatment. The current paradigm consist of administering high AI doses to both obese and non obese patients, but this may not be the best strategy to avoid or delay drug resistance.

The aforementioned issues are difficult to quantify in preclinical and clinical settings and can benefit from more formal approaches. Here, we propose a new mathematical model, based on a system of ordinary differential equations (ODEs), to model the concentration of estrogen in the cancer tissue, which takes into account the local interplay between the tumor and fat tissues. We inform the model with data from mouse experiments that investigate the effect of obesity in breast cancer using two groups of mice, fed with control diet (CD) or with high-fat diet (HFD). Then, we incorporate AI therapy into the calibrated models, including de novo and adaptive resistance. To determine optimal therapeutic interventions in the CD and HFD cases, we formulate an optimal control problem (OCP) with the goal of minimizing the total tumor volume and the total amount of treatment that is used. We also compare the obtained optimal schedules with constant and alternating treatments.

%\textcolor{magenta}{Mathematical modelling of biological phenomenon enables us to express real-world problems as a set of equations and their simulations in silico save time and resources compared to biological experiments and clinical trials} \cite{fischer2008mathematical, brady2019mathematical}.
Mathematical modeling of breast cancer dynamics under treatment have gained interest for long time \cite{norton1977tumor,enderling2006mathematical,enderling2007mathematical,frieboes2009prediction,roe2011mathematical, yankeelov2013clinically, lai2018modeling, jarrett2019mathematical, lai2019toward, lai2022scalable}. However, modeling of AI treatment in ER-positive breast cancer has received less attention so far. For example, an ODE model was proposed to understand pathway dynamics of ER-positive MCF-7 breast cancer cells under combination of Cdk4/6 inhibition and anti-hormonal therapies, including AI treatment~\cite{he2020mathematical}. Similarly, Chen et al. proposed a mathematical model based on a system of ODEs to understand resistance to AI treatment driven by a shift from estrogen to growth factor receptors \cite{chen2013modeling}. In another study that uses stochastic differential equations and statistical physics techniques, the transitions under AI treatment between three different estrogen sensitive phenotypes were considered \cite{chen2014mathematical}. To explain the dual effect of estrogen inducing both growth and regression of hormone-dependent breast cancer (referred as estrogen paradox), Ouifki and Oke proposed an ODE model and determined conditions to eliminate cancer recurrence for long-term treatment based on stability analysis \cite{ouifki2022mathematical}.
%Mirzaei et al. proposed a mathematical model including breast cancer cells, cytokines, elements of immune system, estrogen and cancer associated adipocytes via TCGA and METABRIC clinical data to understand the dynamics in tumor microenvironment. \cite{mohammad2021mathematical}.  
%(Enderling et al JTB 2006, 2007, RoeDale et al Bull Math Biol. 2011, Yankeelov et al STM 2013, Jarret et al Math Med and Bio 2019, Lai et al Cancer Research 2019, add some more references)
%\textcolor{red}{Mention here paper by Chen et al FEBS Lett 2013} 
Cancer treatment scheduling optimization by means of OCPs has received considerable attention \cite{schattler2015optimal, jarrett2020optimal, yildiz2018new, akman2018optimal}. For instance, OCPs were proposed to optimise treatment schedules of chemotherapies \cite{de2001mathematical, panetta2003optimal, ledzewicz2022structure}, angiogenic inhibitors \cite{ledzewicz2007antiangiogenic}, cytotoxic and antiangiogenic therapies \cite{colli2021optimal}, immunotherapy via a dendritic cell vaccine \cite{castiglione2007cancer} and combination therapies \cite{ledzewicz2012multi, sharp2020designing}. In addition, resistance to chemotherapy~\cite{costa1992optimal, carrere2017optimization} or combination of chemotherapy with ketogenic diet~\cite{oke2018optimal} were also studied using OCPs. Another recent study investigated the optimal combination of doxorubicin and HER2 targeting drug trastuzumab, for a murine model of human HER2 positive breast cancer~\cite{lima2022optimizing}. To the best of our knowledge, the present work is the first modeling study to account for anti-hormonal treatment using AIs in the presence of drug resistance in an optimal control framework.

The paper is organized as follows: In the following section~\ref{model}, we formulate the dynamical model, prove some basic  properties and we proceed with model calibration. Sec.~\ref{extension} is dedicated to the model extension for AI treatment and resistance. In Sec.~\ref{ocpp}, we formulate the OCP and derive the optimality system. Then, we proceed with results in Sec.~\ref{sim} to compare various relevant scenarios for anti-hormonal treatment. We conclude by discussing the main conclusions, limitations and future directions.

%\clearpage
\section{Mathematical model development and calibration }\label{model}
In this section, we develop a basic ODE model for the interaction of ER-positive breast tumor cells, estrogen and fat for the postmenopausal situation, and demonstrate some useful mathematical properties of its solution. Our model can describe the contribution of fat intake differences to estrogen and tumor growth over time. %It has been constructed in such a way that experimental data could be represented and it could be extended to interpret cancer recurrence on a longer time interval with additional tumor subpopulations. We restrict ourselves to tumor cells, estrogen concentration and adipocytes as model variables, since we have some available data for tumor volume and adipocyte size (see~\cite{hillers2018obesity}), and estrogen is a vital element of the tumor microenvironment for ER+ subtype.
We have been inspired by the mice experiment conducted by Hillers et al.~\cite{hillers2018obesity} comparing tumor growth in CD and HFD mice, and we use the data obtained in that study to bring our model closer to reality. %By this way, we can put the model to a realistic ground and determine biologically reasonable parameter values in model calibration step. In this current work, we interpret the breast cancer mechanism via basic models of growth, production and decay terms. Of course, it could be possible to propose different models for the biological problem at hand, but more complicated models require more parameters despite of limited number of available data. In the rest of this section, 
We proceed by describing the experiments and available data that inspired our model. We then present model equations and the assumptions they are based on. Then, we discuss mathematical properties of the model. Lastly, we explain the details of the model calibration.

\subsection{Experimental data}\label{sec:exp_data}
Hillers et al. investigated the influence of obesity on breast tumor size and stromal cells within the mammary adipose tissue~\cite{hillers2018obesity}. We use data from that study that utilized breast cancer cell line EO771 derived from a spontaneous mammary adenocarcinoma from a C57Bl/6 mouse. EO771 cells are considered to be a model of luminal B breast cancer subtype and are known to respond to anti-estrogens \cite{le2020eo771}. Specifically, mice were fed with CD (10\% kcal from fat, Test Diet 58Y1) or HFD diet (60\% kcal from fat, Test Diet 58Y2). A total of $1 \times 10^6$ EO771 tumor cells were mixed with $2.5 \times 10^5$ adipocytes taken from CD or HFD mice. After pelleting this mixture of cancer and fat cells, it was injected bilaterally into the inguinal mammary glands of 8-week-old female mice fed with CD. In total, we have the data of eleven mice, five where the fat cells come from mice fed with CD and six where fat cells come from mice fed with HFD. For each tumor independently, tumor volumes were measured at days 10, 13 and 15, as depicted in the study~\cite[Fig.~2B]{hillers2018obesity}. In addition, the number of adipocytes at day 15 was quantified, see \cite[Fig.~S2F]{hillers2018obesity}, and we use it to estimate the amount of fat in the tumor tissue. Details are provided in the Supplementary Material~\ref{supp}. Mice were euthanized when tumor reached the humane endpoint of 15 mm in diameter. In summary, we obtained the initial conditions and six independent measures of tumor volume at three time points (days 10, 13 and 15) for each condition (CD and HFD), and two data points for fat volume at day 15, one for each condition (CD and HFD).

%Hiller's et al. investigated the influence of obesity on breast tumor size and stromal cells within the mammary adipose tissue~\cite{hillers2018obesity}. In particular, we use data courtesy of Lisa M. Arendt using breast cancer cell line EO771 derived from a spontaneous mammary adenocarcinoma from a C57Bl/6 mouse. In that experiment, \textcolor{blue}{5} female mice were fed CD \textcolor{red}{and 6 female mice with} HFD for 16 weeks. Then, a total of $ \textcolor{blue}{1 \times} 10^6$ EO771 tumor cells were mixed with $2.5 \times 10^5$ adipocytes taken from mice \textcolor{red}{ fed HFD or CD}, respectively). After pelleting this mixture of cancer and fat cells, it was injected bilaterally into the inguinal mammary glands of 8-week-old female mice fed with CD.  Tumor volumes were then measured at days 10, 13 and 15, as depicted in the study \cite[Fig.~2B]{hillers2018obesity}. In addition, the number of adipocytes at day 15 was quantified (see \cite[Fig.~S2F]{hillers2018obesity}) and we use it to estimate the amount of fat in the tumor tissue. Calculation details are provided in the Supplementary Material. Mice were euthanized when tumor reached the humane endpoint of 15 mm in diameter. In summary, we obtained the initial conditions and six \textcolor{red}{independent} measures of tumor volume at three time points (days 10, 13 and 15) for each condition (CD vs HFD), and two data points for fat volume at day 15, one for each condition (CD vs HFD).

\subsection{Model development }
In order to quantify the effect that the fat-induced production of estrogen has on tumor growth, we model the temporal dynamics of tumor volume $T:=T(t)$ (mm$^3$), estrogen concentration $E:=E(t)$ (pg/g) and fat volume $F:=F(t)$ (mm$^3$) in the tumor tissue at time $t$ (days). The model is based on the following six assumptions:
\begin{enumerate}
    \item Tumor volume follows logistic growth \cite{benzekry2014classical} . 
    \item Tumor growth rate depends on the estrogen level \cite{le2020eo771}.
    \item Fatty tissue is the major source of estrogen in the tumor \cite{simpson2003sources}. Circulating estrogen concentrations are proportional to adipose mass in postmenopausal women \cite{marchand2018increased}, so we assume that estrogen is produced by fat at a constant rate.
    \item Estrogen is washed out from the tumor tissue at a constant rate \cite{deshpande1967}.
     \item Tumor cells use fat as an energy resource \cite{hoy2017adipocyte, wang2017mammary}.
     \item While mice are fed CD for 15 days, there is no growth source for fat volume. Therefore, diet-based difference in fat volume (CD vs HFD) are accounted for simply due to the amount of fat volume at day 0. 
     \item In the experiments that we model, most implanted adipocytes survive. Cancer cells are known to produce growth factors and cytokines that support the survival of adipocytes. Therefore, we did not include fat decrease due to adipocyte death.
\end{enumerate}

A flow diagram depicting the interactions between model variables $T$, $E$ and $F$ is presented in Fig.~\ref{Fig_Diagram}(a).

\begin{figure}[ht!]
\centering
\subfloat[][]{\includegraphics[height=0.175\textheight]{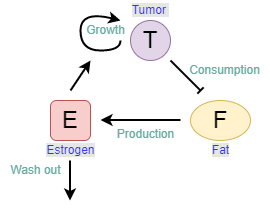}\label{Fig_DiagramA}} \qquad
\subfloat[][]{\includegraphics[height=0.325\textheight]{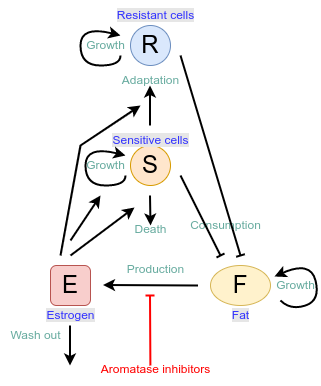}\label{Fig_DiagramB}}
\caption{ \protect\subref{Fig_DiagramA} Modelled interactions between the volume of tumor cells, $T$, fat volume, $F$, and estrogen concentration, $E$, in Eq.~\eqref{Model1}. Tumor volume grows as a consequence of cancer cell proliferation, which is triggered by estrogen. Estrogen is produced by fat and it is washed out. Tumor cells consume fat as energy resource. \protect\subref{Fig_DiagramB} Modelled interactions between the volume of resistant cells, $R$, volume of sensitive cells, $S$, fat volume, $F$, and estrogen concentration, $E$, in Eq.~\eqref{Model2}. The volume of both sensitive and resistant cells grows as a consequence of cell proliferation. While the growth of sensitive cells is triggered by estrogen, that of resistance cells is estrogen independent. Sensitive cells can die under the influence of estrogen or can adapt to low estrogen levels and become resistant. Both sensitive and resistant cells consume fat. Fat volume can change size as a consequence of diet. Estrogen is produced by fat but this production is inhibited by AIs. Estrogen is also naturally washed out. In both diagrams, the lines ending with an arrow represent positive feedback whereas the lines ending with a bar denotes negative feedback.}\label{Fig_Diagram}
\end{figure}

Consequently, we propose the following system of ODEs:
\begin{subequations}\label{Model1}
	\begin{align}
	\underbrace{\dfrac{dT}{dt}}_{change~ in~ tumor~ volume} &= \underbrace{\overbrace{\frac{k_1 E}{a_1 + E}}^{proliferation~rate~triggered~by~estrogen} T \Big(1- m_1 T \Big)}_{logistic~growth~term} , \label{1a}\\ 
	\underbrace{\dfrac{dE}{dt}}_{change~ in~ estrogen~ concentration} &=  \underbrace{rF}_{estrogen~production} - \underbrace{\mu E}_{wash~out}, \label{1b}\\
	\underbrace{\dfrac{dF}{dt}}_{change~ in~ fat~ volume} &= \underbrace{- \alpha TF}_{energy~consumption}, \label{1c}\\
	T(0) &=T_0, \, E(0) = E_{0},  \quad F(0)= F_{0}. \nonumber
	\end{align}%
\end{subequations}
The parameters $k_1, a_1, m_1, r, \mu,\alpha$ and initial conditions $T_0, E_{0}$ and $F_{0}$ are all non-negative real numbers. Eq.~\eqref{1a} represents the tumor logistic growth, where the growth rate is assumed to follow Michaelis-Menten kinetics, $g(E)=\frac{k_1 E}{a_1 + E}$. Parameter $k_1$ is the maximum growth rate for high estrogen levels and $a_1$ is the estrogen concentration at which the growth rate is half-maximum. Parameter $m_1$ is the inverse carrying capacity of the tumor. Eq.~\eqref{1b} models the change in estrogen concentration when it is produced by fat at a rate $r$ and washed out from the tumor tissue at a rate $\mu$. The last equation \eqref{1c} accounts for fat consumption by tumor cells at a rate $\alpha$. The values of these parameters are not known and will be estimated from data.

\subsection{Model properties}
Next we prove that the solution to model~\eqref{Model1} exists, it is unique, non-negative and bounded. These properties will be used later. %\textcolor{red}{In practice this means that $T$, $E$ and $F$ can be uniquely determined if one assumes that they follow Eq.~\eqref{Model1}, which in turn is based on the six assumptions above. The proposition also says that these quantities will be positive (as needed) and that they cannot grow infinitely as time passes.}

\begin{proposition}\label{Thm1}
Eq.~\eqref{Model1} with non-negative initial conditions has a unique solution that is non-negative and bounded from above for all $t\geq 0$.
\end{proposition}
\begin{proof}
As the right-hand side of the model~\eqref{Model1} and their partial derivatives are continuous on $\mathbb{R} \times \mathbb{R}^3$, it follows from the Cauchy–Lipschitz theorem that the existence and uniqueness of the solution are guaranteed \cite[Ch.15]{schatzman2002numerical}.
%As the terms on the right-hand side of the model~\eqref{Model1} and their partial derivative are continuous on $\mathbb{R} \times \mathbb{R}^3$, it follows from the existence and uniqueness theorem that the system has a unique solution in the interval $[t_0 - \epsilon, t_0 + \epsilon]$ for some $\epsilon > 0$~\cite[Sec.2.2]{schroers2011ordinary}.

To prove that the solution to Eq.~\eqref{Model1} is non-negative for all $t\geq 0$, we use the method of separation of variables. Firstly, Eq.~\eqref{1c} leads to 
\begin{align}
    F(t) = F(0) \exp \Big( - \int_{0}^{t} \alpha T(s) \, ds \Big) \geq 0.
\end{align}
Since $F$ and $r$ are non-negative, we can rewrite Eq.~\eqref{1b} as
\begin{align}
    \dfrac{dE}{dt} \geq - \mu E.
\end{align}
Thus, $E \geq 0$. Finally, Eq.~\eqref{1a} leads to
\begin{align}
    T(t) = T(0) \exp \Big(  \int_{0}^{t}  \frac{k_1 E(s)}{a_1 + E(s)} \Big(1 - m_1 T(s) \Big) \, ds \Big) \geq 0.
\end{align}
Therefore, $T \geq 0$, $E \geq 0$ and $F \geq 0 $ for all $t \geq 0$. 

To prove that the solution to Eq.~\eqref{Model1} is bounded from above, we observe from Eq.~\eqref{1a} that
\begin{align}
    \frac{dT}{dt} \leq k_1 T \Big(1- m_1 T \Big).
\end{align}
%By Kamke comparison theorem \textcolor{red}{I cannot find a correct source to cite},
Then,
\begin{align}
   \lim_{t \rightarrow \infty} \sup m_1 T(t) \leq \frac{1}{m_1}.
\end{align}
Since $\frac{dF}{dt} \leq 0$, $F$ stays constant at $F_0$ or decreases. Then, $F(t) \leq F_0$. Finally, Eq.~\eqref{1b} leads to
\begin{align}
    \frac{dE}{dt} \leq  r F_0 - \mu E,
\end{align}
and 
\begin{align}
    \lim_{t \rightarrow \infty} \sup E(t) \leq \frac{r F_0}{\mu}.
\end{align}
Thus, the solution is bounded from above.
\end{proof}

\subsection{Model calibration}

%In this section, we make use of the experimental data described in Sec.~\ref{sec:exp_data} to inform our basic model. It is important to note that direct measures of model parameters are not available and it is very difficult to find reliable estimates of these parameters in the literature. Since we do not have enough data to do formal statistical inference for all parameters, we decided to fix two of the parameters and do the calibration of the rest in two steps. We firstly assume that fat volume is at equilibrium and obtain a steady-state concentration for estrogen. Under those assumptions we estimate the initial estrogen concentration, at steady-state, and parameter values in the tumor equation. Then, we proceed with the original model to determine the rest of the parameter values and the initial condition for fat.

In this section, we make use of the experimental data described in Sec.~\ref{sec:exp_data} to inform our basic model.  Direct measures of model parameters are not available in this experimental setup, and we do not have enough data to do formal statistical inference for all parameters. The main argument we used to fix some parameters was identifiability of the remaining free parameters. To that goal, we decided to fix three of them to reasonable values and made extra assumptions to fix the initial estrogen concentration and fat volume. We then use the available data to calibrate the rest of the parameters for which we lack any information and show that the problem is practically identifiable by using profile likelihood \cite{kreutz2012likelihood}. 

%%%%%%%%%%%%%%%%%%%%%%%%%%%%%%%%%%%
%\textcolor{blue}{The main argument we used to fix some parameters was identifiability of the remaining free parameters. We study identifiability based on the method of profile likelihood as it was mentioned in Sec. 2.4. The problem with all parameters free is nonidentifiable. To solve it, we had to fix some parameters and 
%the ones that we could find reasonable values where $\mu$ (from the literature) and $m_1$ (based on the available tumor data). Fixing those parameters still did not solve the problem, but we discovered that fixing $r$ in addition solved this issue. We have now improved the explanation of this procedure on page 8. We think that parameter identifiability is a stronger argument than lower AICs, because inference problems with low AIC can still lead to non-identifiable parameters. }

%%%%%%%%%%%%%%%%%%%%%%%%%%%%%%5
%We firstly assume that fat volume is at equilibrium and obtain a steady-state concentration for estrogen. Under those assumptions we estimate the initial estrogen concentration, at steady-state, and parameter values in the tumor equation. Then, we proceed with the original model to determine the rest of the parameter values and the initial condition for fat.

The initial amount of fat in the tumor tissue was measured only at day 15. For simplicity, we assume that the level of fat under CD stays constant and it has not changed since the beginning of the experiment (see \cite[Fig.~S2F]{hillers2018obesity}). We acknowledge this is a limitation and an initial fat measurement would have made our results more solid. As estrogen is mainly produced by fat, we also assume that estrogen concentration is proportional to fat volume at baseline. Indeed, estrogen concentration in mice under different, but comparable, conditions was measured between 150-1500 pg/g \cite[Fig.2]{yue1999aromatase}. For estrogen concentration in our model to lie within those measures, we assume that the ratio of estrogen concentration to fat volume is around 3.4.

We then find reasonable values for parameters $m_1$ and $\mu$. We obtain the half-life of estrogen in breast tumor tissue from~\cite{deshpande1967}, $t_{1/2}$=2.8 h. Therefore, $\mu$, that represents the washout rate of estrogen from tumor tissue, can be computed as $\mu=\ln(2)/t_{1/2}=0.25$~h$^{-1}=5.94$~day$^{-1}$. In the experiments, mice were euthanized after the tumor reached 15 mm in diameter, corresponding to a volume of approximately 1767mm$^3$ assuming a spherical tumor. We simply set $m_1=1/2000$~mm$^{-3}$  equally for CD and HFD as a larger value than the highest tumor volume in the data set. Fixing $m_1$ and $\mu$ still did not solve the non-identifiability problem, but we discovered that fixing $r$ in addition solved this issue. Parameter $r$ is fixed as 20 pg/g mm$^{-3}$ day$^{-1}$ based on the assumption that estrogen is at the steady-state in the beginning of the experiment which leads to $E = \frac{rF}{\mu}$. Estrogen concentration roughly satisfies $150 \leq E= \frac{rF}{\mu} \leq 1500 $ \cite[Fig.2]{yue1999aromatase}. By multiplying both sides of this inequality by $\mu=5.94$, we reach $891 \leq r F \leq 8910$. We divide both sides by $F_{CD}(0)=50$ and $F_{HFD}(0)=360$ separately, that results in two inequalities $17.82 \leq r \leq 178.2$ and $2.475 \leq r \leq 24.75$. The intersection of these inequalities gives a range for the parameter $r$ which is $17.82 \leq  r \leq 24.75 $. Therefore, we simply choose $r=20$. 

We perform model calibration and profile likelihood calculations in Data2Dynamics \cite{raue2015data2dynamics, raue2013lessons}. We fixed the lower and upper bounds for the parameters as $10^{-7}$ and $10^4$ in the optimization problem, respectively. Based on the method of profile likelihood \cite{kreutz2012likelihood}, our model is practically identifiable (See.~\ref{prof_like}). We list the obtained parameter values in Table~\ref{Tab1}. We perform sensitivity analysis for all the parameters in Sec.~\ref{supp}.

\begin{table}[h!]
\begin{center}
%\begin{minipage}{10cm}
\caption{Values of the parameters in the Eq.~\eqref{Model1}.}\label{Tab1}
%\begin{tiny}
\begin{tabular}{llll}
\toprule
Parameter & Description & Units & Value \\
\midrule
$m_1$ & Inverse carrying capacity of tumor  & mm$^{-3}$ & 1/2000   (assumed)\\\hline
$\mu$ & Estrogen washout rate & day$^{-1}$ & 5.94   \cite{deshpande1967}\\\hline
$k_1$  & Tumor growth rate  & day$^{-1}$ & 0.55 (calibrated) \\\hline
$a_1$ & Half maximum estrogen threshold  & pg/g    & 43 (calibrated)\\ \hline
$r$ & Estrogen production rate   &pg/g mm$^{-3}$ day$^{-1}$ &  20 (assumed)  \\ \hline
$\alpha$ & Fat consumption rate & day$^{-1}$ mm$^{-3}$ & 1.7e-06  (calibrated) \\ \hline
$T_{0}$ & Initial tumor volume  & mm$^{3}$  & 1   \cite{hillers2018obesity}\\ \hline
$E_{0}$ & Initial estrogen concentration & pg/g  & 170, CD (estimated)  \\ & & & 1200, HFD (estimated) \\ \hline
$F_{0}$ & Initial fat volume  &   mm$^{3}$  & 50, CD (assumed)\\ & & & 360, HFD (assumed) \\ 
\botrule
\end{tabular}
%\end{tiny}
%\end{minipage}
\end{center}
\end{table}

%\begin{remark}
%\\

%Typical estradiol concentration in tumors from postmenopausal women range from 5 to 60 pg/g (mean 15 pg/g) [Tagaki et al 2013 British Journal of Cancer]. 

%If we use E=225 pg/g as a typical value, then $r_{1}=t_{1/2}*E=0.25*225=56.25$~pg*g$^{-1}$*h$^{-1}$.

%Observation: In this case we would be modeling total estradiol concentration in tissue instead of free estradiol. Becasue both the half-life value and the typical concentration are for total estrdiol levels in tissue. 
%\end{remark}

Fig.~\ref{Fig1} shows the simulation results for CD (left panel) and HFD (right panel) obtained with the parameters in Table~\ref{Tab1}. The $y$ left-axes correspond to tumor or fat volume, whereas the $y$ right-axes denote the estrogen level. Data points with error bars for tumor and fat volume are marked with red circles and black crosses, respectively. We observe that simulation results for tumor and fat volume agree well with the data, whereas estrogen level stays within a biologically meaningful interval. We observe that the initial estrogen level and fat volume are higher for HFD than CD (see, Table~\ref{Tab1}.). Temporal evolution of estrogen level and fat volume is similar for both diet types, since fat is assumed as the source of estrogen production. Indeed, tumor associated with HFD increases faster than for the other, due to more estrogen release from HFD fat volume. 
\begin{figure}[h!]
	\centering
	$
	\begin{array}{cc}
	\includegraphics[height=0.4\textwidth]{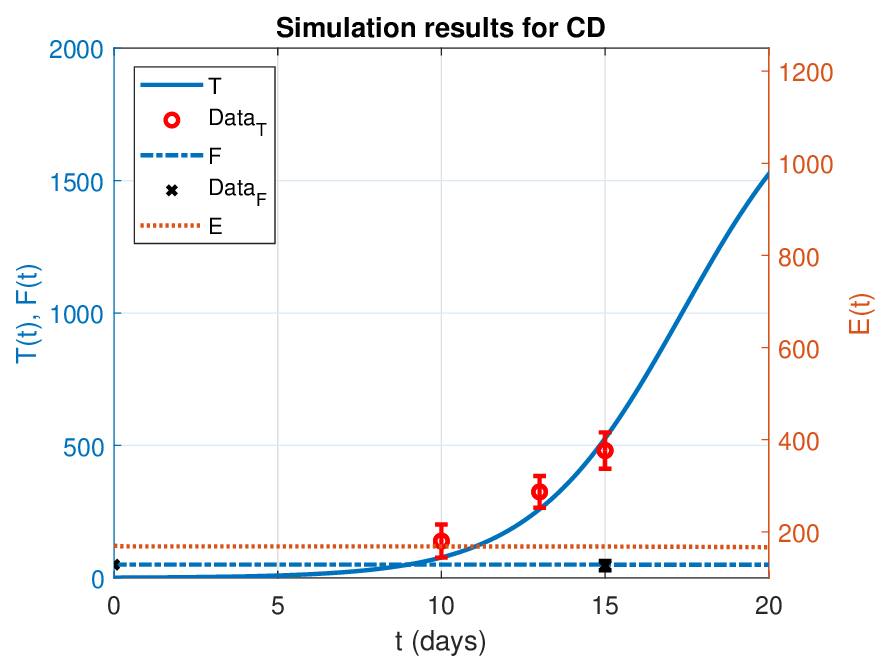}&
	\includegraphics[height=0.4\textwidth]{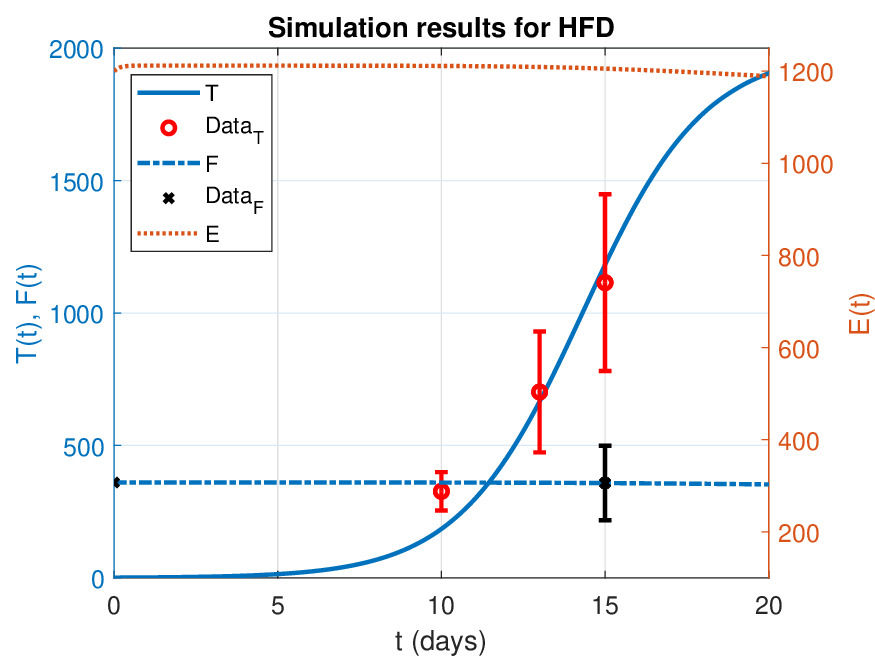}
	\end{array}
	$
	\caption{Simulation results of the Eq.~\eqref{Model1} for CD (\textit{left}) and HFD (\textit{right}) with the data points. Left axis corresponds to tumor size $T(t)$ and fat volume $F(t)$, right axis denotes estrogen concentration $E(t)$ over time $t$.}\label{Fig1}
	%\caption{Model38 - v4 ((without the equality constraint)): Data-fitting for (tumor-fat data) CD data and HFD data, together.}
\end{figure}

Original model \eqref{Model1} expresses the changing dynamics of tumor volume, estrogen level and amount of fat in case of no treatment and simulation results agree well with the available data. The next step is to extend this model to account for AI treatment by considering sensitive and resistant tumor subpopulations. In this way we will be able to study drug resistance to endocrine therapy for ER-positive breast cancer. 

\section{Model extension for resistance to aromatase inhibitor treatment}\label{extension}
Aromatase inhibitors, despite of being an effective treatment choice for ER-positive breast cancer, may suffer from drug resistance \cite{chumsri2011aromatase, ma2015mechanisms}. %\textcolor{red}{Add here also Ma et al Nature Reviews Cancer 2015}. 
To investigate different treatment schedules, including constant, alternating and optimal anti-hormonal treatment, we consider tumor heterogeneity in terms of sensitive and resistant subpopulations under the following assumptions: 
%\paragraph{Assumptions: }
\begin{enumerate}
    \item Breast cancer cells are either sensitive or resistant to estrogen deprivation with AIs. In reality, there could be more than two tumor subpopulations, since development of resistance is considered as a progressive mechanism and cells may shift form one stage to another over time \cite{normanno2005mechanisms}. For simplicity, we assume that there are only two tumor subpopulations, called sensitive and resistant.
    \item Both sensitive and resistant populations follow logistic growth \cite{benzekry2014classical}. 
    \item Growth of sensitive cells is triggered by estrogen~\cite{doisneau2003estrogen}.
    \item Sensitive cells die under low estrogen concentrations~\cite{doisneau2003estrogen}.
    \item Sensitive cells adapt to low estrogen levels and become resistant cells~\cite{chen2013modeling,chen2014mathematical}. 
    \item Resistant cells do not die under low estrogen concentrations~\cite{chen2013modeling,chen2014mathematical}.
    %\item Resistant cells proliferate slower than sensitive cells \textcolor{red}{\cite{duan2017strategy}}.
    \item Fat volume follows logistic growth \cite{ku2016mathematical}.
     \item Both sensitive and resistant cells consume fat as energy resource \cite{ku2016mathematical}.
\end{enumerate}

Consequently, Eq.~\eqref{Model1} is extended with the sensitive cell population $S:=S(t)$ (mm$^3$) and the resistant cell population $R:=R(t)$ (mm$^3$) in the tumor tissue at time $t$ (days):
\begin{subequations}\label{Model2}
    \begin{align}
	\underbrace{\dfrac{dS}{dt}}_{change~in~sensitive~cell~population} &= \underbrace{ \frac{k_1 E}{a_1 + E} S\Big(1- m_1 ( S + \eta R ) \Big)}_{logistic \, growth \, term} -  \underbrace{\frac{c a_2^l}{a_2^l + E^l} S}_{death \, term} - \underbrace{\frac{c a_3^l}{a_3^l + E^l} S}_{adaptation \, term} ,  \label{2a}\\ 
	\underbrace{\dfrac{dR}{dt}}_{change~in~resistant~cell~population} &= \underbrace{k_3 R \Big(1- m_1 ( S + \eta R ) \Big) }_{logistic \, growth \, term  } + \underbrace{\frac{c a_3^l}{a_3^l + E^l} S}_{adaptation \, term} , \label{2b}\\ 
	\underbrace{\dfrac{dE}{dt}}_{change~in~estrogen~concentration} &= \underbrace{ pr F}_{estrogen \, production} - \underbrace{\mu E}_{wash \, out}, \label{2c}\\
	\underbrace{\dfrac{dF}{dt}}_{change~in~fat~volume} &= \underbrace{k_2 F (1-m_2 F)}_{logistic \, growth \, term} - \underbrace{ \alpha (S+R)F}_{energy \, consumption} \label{2d},\\
	S(0) &=S_0, \, R(0)=R_0, \, E(0) = E_{0}, \,  F(0)= F_{0}, \nonumber
	\end{align}%
\end{subequations}
with the non-negative initial conditions $S_0$, $R_0$, $E_0$ and $F_0$. Eq.~\eqref{2a} expresses the logistic growth of sensitive cells over time together with death and adaptation terms. Parameter $m_1$ is the inverse of maximum tumor size and $\eta$ is the competition parameter scaling inhibition of sensitive cells' growth by resistant cells. Sensitive cells die if estrogen level is smaller than $a_2$ while they adapt to estrogen level below $a_3$ and become resistant. Parameter $c$ is the maximum death rate and $l$ denotes Hill's coefficient. Eq.~\eqref{2b} models evolution of resistant cells with the growth rate $k_3$. Eq.~\eqref{2c} stands for dynamics of estrogen level where the parameter $p$, $0<p\leq 1$, reduces the effect $r$ to $p \cdot r$ due to aromatase inhibitors. %Here, $p=1$ corresponds to no estrogen deprivation treatment and smaller values of $p$ models inhibition of estrogen production through AI treatment. 
Eq.~\eqref{2d} models the change in fat volume with logistic growth so that effect of fat growth to anti-hormonal treatment could be investigated. Parameters $k_2$ and $m_2$ are the growth rate of fat and inverse carrying capacity of fat, respectively. The carrying capacity could model how the body is prone to accumulate fat depending on the life style or other metabolic conditions. In addition, both sensitive and resistant cells consume fat as energy resource at the rate $\alpha$. A diagram depicting the interactions between the extended model variables $S$, $R$, $E$ and $F$ is presented in Fig.~\ref{Fig_Diagram}(b).

We assume that the parameters that we calibrated in the basic model are not affected by the treatment and we use them in the extended model. As we do not have data under treatment, we explore the effect that new parameters have by testing different values.
%The new mechanisms that are included due to treatment modeling do not play a role in the mice that we use for calibration. Therefore, we fix the values of new parameters in the extended model to biologically reasonable values, since we do not have data for treated mice to perform model calibration.

\subsection{Model properties}

\begin{proposition}\label{Thm4}
Eq.~\eqref{Model2} with non-negative initial conditions has a unique solution that is non-negative and bounded from above for all $t\geq 0$.
\end{proposition}

\begin{proof}

Existence and uniqueness of the solution is standard and analogous to Proposition~\ref{Thm1}. Thus, we prove here that the solution to Eq.~\eqref{Model2} is non-negative and bounded from above for $t \geq 0$.
Similar to Theorem~\ref{Thm1}, we can prove that $E, F \geq 0$ for $t\geq 0$ by the variation of constants formula. For Eq.~\eqref{2a},
we have
\begin{align}
    S(t) = S(0) \exp \Big\{ \int_{0}^t  \Big(\frac{k_1 E(s)}{a_1 + E(s)} \Big(1- m_1 ( S(s)+ \eta R(s) ) \Big) - \frac{ca_2^l}{a_2^l + E^l(s)} - \frac{ca_3^l}{a_3^l + E^l(s)}\Big)  \Big \} \geq 0.
\end{align}
Since $S \geq 0$ for $t\geq 0$, Eq.~\eqref{2b} can be written as
\begin{align}
    \dfrac{dR}{dt} \geq k_3 R \Big(1- m_1 ( S(s)+ \eta R(s) ) \Big).
\end{align}
Then, we get
\begin{align}
    R(t) \geq R(0)  \exp \Big\{ \int_0^t k_3 \Big(1- m_1 ( S(s)+ \eta R(s) ) \Big) \; ds \Big\} \geq 0.
\end{align}

We can prove that $F$ and $E$ are bounded from above similar to Theorem~\ref{Thm1}. On the other hand, using Eq.~\eqref{2a}-\eqref{2b}, we obtain the sum
\begin{align}
    \frac{dS}{dt} + \eta  \frac{dR}{dt} 
    &\leq \underbrace{ \frac{k_1 E}{a_1 + E}}_{\leq k_1} S\Big(1- m_1 ( S(s)+ \eta R(s) ) \Big) + \eta  k_3 R \Big(1- m_1 ( S(s)+ \eta R(s) ) \Big) \nonumber\\
    &\leq (k_1 S + \eta k_3 R) \Big(1- m_1 ( S(s)+ \eta R(s) ) \Big) \nonumber\\
    &\leq \max\{ k_1, k_3 \} (S + \eta R) \Big(1- m_1 ( S(s)+ \eta R(s) ) \Big).
\end{align}
Thus, 
\begin{align}
    \lim_{t \rightarrow \infty} \sup  (S(t) + \eta R(t)) \leq  1/ m_1.
\end{align}
Since $S$ and $R$ are non-negative, it means that $S$ and $R$ are bounded above. Then, we complete the proof. 
\end{proof}

\subsection{Treatment modelling}
We will investigate differences between constant, intermittent and optimal anti-hormonal treatment. Constant treatment is implemented through the parameter $0 \leq p \leq 1$ in Eq.~\ref{Model2}. The value $p=1$ corresponds to no estrogen deprivation treatment and smaller values of $p$ models AI treatment with inhibition of estrogen production.

Alternating treatment refers to a pre-scheduled treatment scenario with $u_{I}:=u_{I}(t)$ and it is implemented by modifying Eq.~\eqref{2c} to
\begin{subequations}
\begin{align}
    \dfrac{dE}{dt} &=  (1-u_{I}) r F - \mu E,
\end{align}
where 
\begin{align}
    u_{I}  =
    \begin{cases}
    u_b, \hbox{ where } 0 \leq u_b < 1, \quad \hbox{if treatment is applied}, \\
    0, \quad \hbox{ else},
    \end{cases}
\end{align}
\end{subequations}

In the next section, an OCP is constructed to investigate the optimal value of $p$ as a time-dependent function, and results obtained with the optimal endocrine therapy are compared with the constant and alternating treatment.

\begin{comment}
\subsection{Equilibrium points}\label{model_subsec2}
The equilibrium points of the model~\eqref{Model2} are found as
\begin{enumerate}
    %\item $P_1=(0,0,0)$,
    \item $\hat{P}_1=(0,0,0,0)$, 
    \item $\hat{P}_2=(0,\frac{K}{m},0,0)$,
    \item $\hat{P}_3=\Big( 0, 0, \frac{pr}{\mu m_2}, \frac{1}{m_2} \Big)$,
    \item $\hat{P}_4=\Big( 0, \frac{K}{m}, \frac{pr}{\mu m_2} \Big( 1- \frac{\alpha K}{k_2 m} \Big), \frac{1}{m_2 } \Big( 1- \frac{\alpha K}{k_2 m} \Big) \Big)$.
\end{enumerate}
Then, we deduce that for $\hat{P}_4$ to be feasible, the inequality $1- \frac{ \alpha K}{k_2 m} > 0$ must hold. Thus, we must choose the new parameters in such a way that this inequality is not violated.

\textcolor{red}{Here, $P_3$ (no cancer cells) is unstable due to $k_3 > 0$!!! We might be criticized because of it. }
\end{comment}
%\clearpage
\section{Optimal control problem for anti-hormonal treatment}\label{ocpp}
We aim to investigate optimal AI treatment schedules that minimize the total number of cancer cells together with the pharmaceutical intervention over a prespecified time interval $[t_{tr},t_{f}]$. We do not include an equation representing the drug as often done for optimal chemotherapy scheduling in the literature (see, for example, \cite{de2008optimal,sharma2016analysis}). Instead, we model the effect of AI treatment through a continuous control function $u:=u(t)$. AIs act by lowering the estrogen production, so we replace the parameter $p$ in Eq.~\eqref{2c} by the function $1-u$. 

%To the best of our knowledge, this is the first attempt to model optimal hormonal treatment for breast cancer in the presence of drug resistance. 

We formulate the OCP as follows: minimize the cost functional
\begin{align}\label{ocp}
 \mathcal{J}(u) = \int_{t_{tr}}^{t_{f}} ( \omega_S S + \omega_R R +  \frac{\omega_u}{2} u^2 ) \; dt,
\end{align}
subject to
\begin{subequations}\label{Model3}
    \begin{align}
	\dfrac{dS}{dt} &= \frac{k_1 E}{a_1 + E} S\Big(1- m_1 ( S+ \eta R) \Big) -  \frac{ca_2^l}{a_2^l + E^l} S - \frac{ca_3^l}{a_3^l + E^l} S ,  \label{3a}\\ 
	\dfrac{dR}{dt} &= k_3 R \Big(1- m_1 ( S+ \eta R ) \Big) + \frac{ca_3^l}{a_3^l + E^l} S, \label{3b}\\ 
	\dfrac{dE}{dt} &=  (1-u) r F - \mu E, \label{23c}\\
	\dfrac{dF}{dt} &= k_2 F (1-m_2 F) - \alpha (S+R)F \label{3d},\\
	S(0) &=S_0, \, R(0)=R_0, \, E(0) = E_{0}, \,  F(0)= F_{0}, \nonumber
	\end{align}
\end{subequations}	
where,
\begin{align}
    \mathcal{U} = \{ u \, \mid u \hbox{ is \, measurable,} \, u_a \leq u \leq u_b, \hbox{ for all } t \in [t_{tr}, t_{f}], t_{tr} \geq 0, t_{f} >0 \}.
\end{align}

Our aim is to find an optimal control $u^{*}$ such that $\mathcal{J}(u^{*}) = \min_{u \in \mathcal{U}} \mathcal{J}(u)$.\\

We note that constructions of linear or quadratic cost functional in the control function $u$ results in not only biologically but also mathematically different interpretations. While quadratic OCPs have a single extremum and result in continuous controls, linear OCPs result in bang-bang controls and mathematical analysis becomes more complicated due to singular or bang-bang controls that result in non-differentiable solutions curves. We refer readers to the following studies for a detailed comparison \cite{ledzewicz2004comparison, sharp2019optimal, ledzewicz2020role}. In addition, the parameters $\omega_S$, $\omega_R$ and $\omega_u$ in Eq.~\eqref{ocp} can be set to balance the size of the different terms.

In the present case, inclusion of the term $u^2$ in the Eq.~\eqref{ocp} is justified by the treatment side effects. Side effects of AI include from hot flushes to cardiovascular events, vaginal bleeding and bone loss \cite{osborne2005aromatase, cuzick2005aromatase, hadji2010guidelines}. Our quadratic choice reflects the fact that the increase in side effects is negligible for small amounts of therapy and that side effects increase as function of u, rather than increasing at a constant rate as in the linear control.

\begin{theorem}\label{thm6}
There exists an optimal control $u^{*}$ with a corresponding solution $(S^{*}, R^{*}, E^{*}, F^{*})$ to the model~\eqref{Model3} with non-negative initial conditions that minimizes \eqref{ocp} over $\mathcal{U}$.
\end{theorem}
\begin{proof}
The proof is based on several steps according to the study of Fleming and Rishel \cite[Corollary~4.1]{fleming2012deterministic}. Firstly, we observe that the coefficients in Eq.~\eqref{Model3} and its solution are bounded on a finite time interval, so the admissible control set $\mathcal{U}$ and the corresponding solution with initial conditions are non-empty \cite[Thm~9.2.1.]{lukes1982differential}. Secondly, the admissible control set $\mathcal{U}$ is closed and convex. In addition, the right-hand side of the system~\eqref{Model3}, namely $\vec{f}(t,\vec{X},u)$ with $\vec{X}=(S,R,E,F)^{T}$, is continuous, since the system has positive parameters and the non-negative solution by Proposition~\ref{Thm4}. Indeed, it is bounded above by a linear combination of the bounded control and the state as 
\begin{align}
    \mid \vec{f}(t,\vec{X},u) \mid
    &= 
  \Big \lvert \begin{pmatrix}
\frac{k_1 E}{a_1 + E} S\Big(1- m_1 ( S+ \eta R ) \Big) -  \frac{ca_2^l}{a_2^l + E^l} S - \frac{ca_3^l}{a_3^l + E^l} S \\
k_3 R \Big(1- m_1 ( S+ \eta R ) \Big) + \frac{ca_3^l}{a_3^l + E^l} S \\
(1-u) r F - \mu E \\
k_2 F (1-m_2 F) - \alpha (S+R)F
    \end{pmatrix} \Big \rvert \nonumber \\
    &\leq 
\lvert \begin{pmatrix}
k_1 & 0 & 0 &0 \\
0 & k_3 & 0 & 0 \\
0 & 0 & -\mu & r\\
0 & 0 & 0 & k_2 
\end{pmatrix}  
\begin{pmatrix}
S\\
R\\
E \\
F
\end{pmatrix} 
\rvert
+
\lvert \begin{pmatrix}
0\\
0\\
\frac{r}{m_2} u \\
0
\end{pmatrix}  
\rvert \nonumber\\
&\leq C( \mid \vec{X} \mid + \mid u \mid),
\end{align}
due to bounded solution (by Proposition~\ref{Thm4}) and positive parameters in the model for some positive constant $C$. On the third line, we use the relation
\begin{align}
    (1-u)rF - \mu E \leq (1+u)rF - \mu E \leq rF - \mu E + \frac{r}{m_2}u.
\end{align}
The integrand of the objective functional is convex on $\mathcal{U}$ due to the quadratic term. Indeed, it is bounded as
\begin{align}
    \omega_S S + \omega_R R + \frac{\omega_u}{2} u^2
    \geq  \frac{\omega_u}{2} u^2
    \geq - \hat{C} + \frac{\omega_u}{2} u^2,
\end{align}
with some positive constant $\hat{C}$. Thus, we can conclude that an optimal control $u^{*}$ exists.
\end{proof}

\begin{theorem}\label{thm7}
Given an optimal control $u^{*}$ and solution to the system~\eqref{Model3} for \eqref{ocp}, there exist adjoint variables $\lambda_i := \lambda_i(t)$ for $1 \leq i \leq 4$ such that
\begin{subequations}\label{adjoint1}
\begin{align}
\frac{d \lambda_1}{dt} 
&= - w_S - \lambda_1 \left\{ \frac{k_1 E}{a_1 + E} \Big(1 - m_1 ( 2S + \eta R ) \Big) - \frac{ca_2^l}{a_2^l + E^l} - \frac{ca_3^l}{a_3^l + E^l} \right\} \nonumber\\
&+ \lambda_2 \left\{ m_1 k_3 R + \frac{ca_3^l}{a_3^l + E^l}  \right\} + \lambda_4 \alpha F, \\
\frac{d \lambda_2}{dt} 
&=  - w_R + \lambda_1 \left\{ \frac{k_1 m_1 \eta E S }{(a_1 + E)}  \right\} 
- \lambda_2 \left\{  k_3 \Big(1- m_1 (S + 2 \eta R)\Big)  \right\} + \lambda_4 \alpha F, \\
\frac{d \lambda_3}{dt} 
&= - \lambda_1 \left\{ S \Big( 1- m_1 ( S(s)+ \eta R(s) )\Big)  \frac{ k_1 a_1 }{(a_1 + E)^2}  
+ \Big( \frac{ca_2^l}{(a_2^l + E^l)^2} + \frac{ca_3^l}{(a_3^l + E^l)^2} \Big) l E^{l-1}S  \right\} \nonumber\\
&+ \lambda_2 \left\{ \frac{ca_3^l}{(a_3^l + E^l)^2} l E^{l-1} S \right\} + \lambda_3 \mu, \\
\frac{d \lambda_4}{dt} 
&= - \lambda_3 (1-u^{*})r - \lambda_4 (k_2 - 2k_2 m_2 F - \alpha(S+R)),
\end{align}
with
\begin{align}\label{adjoint2}
    \lambda_i(t_f) = 0, \quad 1 \leq i \leq 4.
\end{align}
\end{subequations}
Furthermore, $u^{*}$ can be represented by
\begin{align}\label{optimalitycond}
    u^{*} = \min \Big(u_b, \max \Big( u_a, \frac{r F \lambda_3}{\omega_u} \Big) \Big).
\end{align}
\end{theorem}

\begin{proof}
Following references \cite{de2008optimal,burden2004optimal}, the Lagrangian is constructed as 
\begin{align}
    \mathcal{L} = \mathcal{H} + \xi_1(t) (u - u_a) - \xi_2(t) (u_b-u),
\end{align}
where the Hamiltonian $\mathcal{H}$ is defined as
\begin{align}
\mathcal{H}(S, R, E, F, & \lambda_1, \lambda_2, \lambda_3, \lambda_4, u)  \nonumber\\
    & :=  ( w_S S +  w_R R + \frac{\omega_u}{2} u^2 ) \nonumber \\
    &+ \lambda_1 \Big( \frac{k_1 E}{a_1 + E} S\Big(1- m_1 ( S+ \eta R ) \Big) -  \frac{ca_2^l}{a_2^l + E^l} S - \frac{ca_3^l}{a_3^l + E^l} S  \Big) \nonumber\\
    &+ \lambda_2 \Big( k_3 R \Big(1- m_1 ( S+ \eta R ) \Big) + \frac{ca_3^l}{a_3^l + E^l} S \Big) \nonumber\\
    &+ \lambda_3 \Big( (1-u) r F - \mu E \Big)
    + \lambda_4 \Big(k_2 F (1-m_2 F) - \alpha (S+R)F \Big),
\end{align}    
and $\xi_i(t) \geq 0$ are penalty multipliers such that
\begin{align}
    \xi_1(t) (u - u_a) =0, \qquad \xi_2(t)(u_b-u) = 0 \hbox{ at } u^{*}.
\end{align}

From the Pontryagin’s Maximum Principle, we can derive the adjoint equations by obtaining partial derivative of the model~\eqref{Model2} with respect to $S, R, E$ and $F$, respectively. Indeed, we get
\begin{align}
    \frac{d \lambda_1}{dt} &= - \frac{\partial \mathcal{L}}{\partial S}, \qquad
    \frac{d \lambda_2}{dt} = - \frac{\partial \mathcal{L}}{\partial R}, \nonumber\\
    \frac{d \lambda_3}{dt} &= - \frac{\partial \mathcal{L}}{\partial E}, \qquad
    \frac{d \lambda_4}{dt} = - \frac{\partial \mathcal{L}}{\partial F},
\end{align}
with $\lambda_i(t_{f}) = 0, i=1, \cdots, 4$.

To obtain an expression of the control, we differentiate the Hamiltonian with respect to $u$ as
\begin{align}
    \frac{\partial \mathcal{H}}{\partial u} = \omega_u u - r F \lambda_3,
\end{align}
and project it onto the admissible set of controls.
\end{proof}

%\begin{remark}
%%Uniqueness of the optimal control could be discussed numerically due to the complicated form of the estrogen-dependent tumor growth rate. 
%We observe that larger time interval leads to convergence issues and it is an indication of the uniqueness of the solution on a smaller time interval.
%\end{remark}

\subsection{Implementation of the optimal control problem}
Third-generation AIs (anastrozole, letrozole and exemestane) reduce whole-body aromatisation by $>$90\% (summarised in ref. \cite{geisler2005aromatase}). However, limited local estrogen production in tissue compartments cannot be totally ruled out. Also,  it is possible that some cells could locally produce some estrogen under treatment \cite{sasano2009situ, geisler2003breast}. Therefore, we assume that the maximum drug dose does not eliminate the total estrogen in the vicinity of tumor. This could be done simply by setting a threshold value on the control function. We use $u_a = 0$ and $u_b = 0.99$, where $u_a$ corresponds to the case of no treatment and $u_b$ refers to the strongest possible treatment.

The optimality system consisting of the state equation \eqref{Model3}, the adjoint equation \eqref{adjoint1} and the optimality condition \eqref{optimalitycond} form a nonlinear system of equations, so we obtain the numerical solution via forward-backward sweep (FBS) method \cite{lenhart2007optimal}. As explained by Lenhart and Workman \cite{lenhart2007optimal}, the FBS method requires initiation of a feasible control function to solve the state equation forward in time. Then, the adjoint equation is solved backward in time and the optimality condition is updated at each iteration until the stopping criterion is satisfied. Here, the idea is to find a feasible optimal control iteratively. The update strategy of the control could be done in different ways such as taking average of the current ($u_{cur}$) and previous control ($u_{pre}$) or their convex combination \cite{lenhart2007optimal}. Here, we apply the approach "greedy" convex combination studied by Vatcheva et al. \cite[Sec.~3]{vatcheva2021social} to cover a large range of control combinations during optimization and avoid stagnation. "greedy" convex combination refers to expressing the control as $u_s = (1-s) u_{pre} + s u_{cur}$ where $s \in (0,1)$ is selected in such a way that the smallest value of $\mathcal{J}(u_s)$ is achieved in that iteration. The parameter $s$ is not fixed as opposed to the averaging or convex combination, it may vary in each iteration instead. The stopping criterion in this paper is based on the relative error of the current and previous state, adjoint and control functions. The program is terminated when a relative error less than $10^{-5}$ is achieved.  

Simulations in this study were performed using MATLAB\textsuperscript{\textregistered}~R2022 \cite{MATLAB:2022}.  We used ode15s solver to obtain the numerical solution of the differential equations and fmincon function in the model calibration step. All data and code are available (see data and code availability part for the details).
%\clearpage
\section{Simulation results}\label{sim}
We focus in simulations of the extended model~\eqref{Model2} that explore the effect of threshold values $a_2$ and $a_3$. These two values correspond to the estrogen concentrations below which cancer cells die or become resistant, respectively.  Thus, simulation scenarios using different threshold values represent treatment in hypothetical tumors with differential sensitivities and rates of resistance to the local estrogen availability. For each case, we simulate three different treatment types:  constant treatment, alternating treatment and optimal anti-hormonal treatment. We use the parameter values which are common in both the first and the extended model. For the others, we either fix their values or explore their impact in simulations. We list all parameter values in Table~\ref{Tab2}.
%We use Table~\ref{Tab1} from the calibration of the simple model to set the parameter values $k_3$ and $m_1$ and we list all parameter values in Table~\ref{Tab2}.

\begin{comment}
\begin{table}[h!]
	\centering
	%\begin{scriptsize}
	\caption{Values of the parameters in the Model~\eqref{ocp}-\eqref{Model3}.}\label{Tab2} % title of Table
	\begin{tabular}{llll}
		Parameter & Description & Units & Value   \\ \hline	\hline
		$c$ (\textcolor{red}{definition} ) &   day$^{-1}$    & 1       \\\hline
		$a_2$ (Concentration of estrogen for half-maximal cell death) &  pg/g     & varies        \\\hline
		$a_3$ (Concentration of estrogen for half-maximal cell mutation) &  pg/g     & varies  \\\hline
		$l$   Hill's coefficient &    -   & 10  \\\hline
		$k_3$   Tumor growth rate   & day$^{-1}$ & $0.25 \times k_1 = 0.1120$ (CD)\\
		 & & $0.25 \times k_1 =  0.1408$ (HFD) \\\hline
		$K$ (Maximum tumor volume)     &   mm$^{3}$  & $1/m_1$  \\\hline
		$m$     & -    &  1 (unless otherwise stated)  \\ \hline
		$p$  (Effect of treatment)   &  -   &  varies  \\\hline
		$t_f$  (Final time)   &  days  &  50  \\\hline
		$u_a$  (Minimum treatment)   &  -  &  0  \\\hline
		$u_b$  (Maximum treatment)   &  -  &  0.975  \\\hline
		$\omega_S, \omega_R, \omega_u $  (Maximum treatment)   &  -  &  1 (unless otherwise stated) \\\hline
	\end{tabular}	
	%\end{scriptsize} 	 		
\end{table}
\end{comment}

\begin{table}[h!]
%\begin{center}
\begin{minipage}{10cm}
\caption{Values of the parameters in the Model~\eqref{ocp}-\eqref{Model3}.}\label{Tab2} % title of Table
\begin{tiny}
\begin{tabular}{llll}
\toprule
Parameter & Description & Units & Value\\  
\midrule
$k_1$  & Growth rate of sensitive cells  & day$^{-1}$ & 0.55 (calibrated) \\\hline
$\mu$ & Estrogen washout rate & day$^{-1}$ & 5.94   \cite{deshpande1967}\\\hline
$\eta$ & Population competition intensity     & -    &  1 (assumed)  \\ \hline
$m_1$ & Inverse carrying capacity of tumor  & mm$^{-3}$ & 1/2000   (assumed)\\\hline
	$c$ & Death rate &   day$^{-1}$    & 1 (assumed)       \\ \hline
	$l$  & Hill's coefficient    &  -   &  10 (assumed) \\\hline
	$a_1$ & Half maximum estrogen threshold  & pg/g    & 43 (calibrated)\\ \hline
		$a_2$ & Estrogen threshold for sensitive cells to die &  pg/g     & varies        \\\hline
		$a_3$ & Estrogen threshold for conversion to resistant &  pg/g     & varies  \\\hline
		$k_3$  & Growth rate of resistant cells   & day$^{-1}$ & varies \\\hline
		$p$  & Effect of treatment   &  -   &  varies  \\\hline
		$k_2$ & Fat growth rate   & day$^{-1}$ & 0.05 (assumed) \\\hline
		$m_2$ & Inverse carrying capacity of fat   &   mm$^{-3}$  & 0.002711\footnotemark[1]   \\\hline
		$r$ & Estrogen production rate   &pg/g mm$^{-3}$ day$^{-1}$ & 20 (assumed) \\\hline
		$\alpha$ & Fat consumption rate & day$^{-1}$ mm$^{-3}$ & 1.7e-06 (calibrated) \\ \hline
		$t_f$  & Final time   &  days  &  25  \\\hline
		$u_a$  & Minimum treatment   &  -  &  0  \\\hline
		$u_b$  & Maximum treatment  &  -  &  0.99  \\\hline
		$\omega_S, \omega_R, \omega_u $ & Positive weight coefficients   &  -  &  1\footnotemark[2]  \\
\botrule
\end{tabular}
\footnotetext[1]{See appendix for computational details.}
\footnotetext[2]{unless otherwise stated.}
\end{tiny}
\end{minipage}
%\end{center}
\end{table}

%\begin{remark}
%We assume that the growth rate of sensitive cells is as same as the one in Model~\eqref{Model1}, whereas the growth rate of resistant cells is smaller than that. Otherwise, we would have a resistant population growing fast compared to the sensitive cell population. \textcolor{red}{"referee may ask: did you try other values of k3?"}
%\end{remark}
For constant treatment, the parameter $p$ is chosen from the set $\{ 1, 0.025, 0.0125, 0.01, 0.001 \}$. For alternating treatment, we set $u_b=0.99$. Treatment is started on the date corresponding to the earliest time point $t:=t_{tr}$ at which $S + \eta R < \frac{1}{4 m_1}$ so that the tumor reaches a detectable size. The final simulation time is fixed as $t_f=25$ to obtain a unique optimal control (see \cite[Sec.~4]{fister1998optimizing}, for detailed discussion). 

We choose the weight coefficients $\omega_S, \omega_R, \omega_u$ in the cost functional as one or hundred to model different penalization strategies. For instance, values $\omega_S, \omega_R > \omega_u$ refers to penalization of tumor cells more than treatment cost.

%\clearpage
%%%%%%%%%%%%%%%%%%%%%%%%%%%%%%%%%%%%%%%%%
%%%%%%%%%%%%%%%%%%%%%%%%%%%%%%%%%%%%%%%%%
\subsection{Scenario I with \texorpdfstring{ $a_2=20$~pg/g, $a_3=1$~pg/g, $k_3 = \frac{k_1}{2}$} .}
%We depict the results for the case $(a_2, a_3) = (10, 0.001)$ in Fig.~\ref{Fig_E1} (CD \textit{(the first row)} and HFD \textit{(the second row)}). Now, death term of sensitive cells is more influential than the conversion term. We observe that treatment for HFD must start earlier than CD due to the larger tumor size obtained with HFD. We observe that sensitive cells of CD reach to a smaller steady state value than the same population associated with HFD. Indeed, for $p$ smaller than $1/32$, sensitive cells are eliminated for CD. Due to low-level conversion term, resistant cells cannot invade the population. As a result of treatment, estrogen level decreases with $p$. Amount of fat for HFD is higher than CD for all values of $p$ which reveals the diet differences. Even though more fat exists in the system for smaller values of $p$, treatment decreases the estrogen level. However, more fat is consumed in case of larger tumor population due to stronger competition between tumor and fat.
\subsubsection{Scenario Ia: Adaptive resistance with \texorpdfstring{$R_0=0$}.}
We first explore a scenario without preexisting resistance, where initially all cells are assumed to be sensitive to treatment and no resistant cells exist. Instead, endocrine resistance may arise due to adaptation to low estrogen levels. Tumor cells may die if estrogen level is below $a_2=20$ pg/g and they may become resistant if estrogen concentration falls below $a_3=1$ pg/g.  Proliferation rate of the sensitive cells is assumed to be equal to half of the growth rate of sensitive cells.

In Fig.~\ref{Fig_E1}, we show response to constant treatment by displaying the change in tumor size for different values of $p$. The solid line corresponds to the case where no treatment is applied, i.e., $p=1$ and tumor reaches to the carrying capacity as time passes. We mark in the figures the time point at which treatment is started with a dashed vertical line and we observe that treatment is started earlier for HFD than CD. We observe that anti-hormonal treatment results in eradication of tumor for the values $p=0.0125$ and 0.01 for both CD and HFD, whereas the case $p=0.025$ does not lead to tumor eradication for HFD. In case of a drug inhibiting estrogen production 99.9\%, i.e., $p=0.001$, drug resistance is observed.
\begin{comment}
\begin{figure}[h!]
	\centering
	$
	\begin{array}{cc}
	\includegraphics[height=0.35\textwidth]{NoControl_Tumor_Case_1aCD.eps}&
	\includegraphics[height=0.35\textwidth]{NoControl_Tumor_Case_1aHFD.eps}
	\end{array}
	$
	\caption{Scenario Ia: Sum of the sensitive $S$ and resistant $R$ tumor subpopulations over time $t$ for the constant treatment with different values of $p$ associated with CD (\textit{left}) and HFD (\textit{right}). We mark the time point at which treatment is started with a dashed vertical line. }\label{Fig_E1}
\end{figure}
\end{comment}

%Subfigure is not recommended by Springer
\newpage
\begin{figure}[h!]
	\centering
	$
	\begin{array}{cc}
	\includegraphics[height=0.28\textwidth]{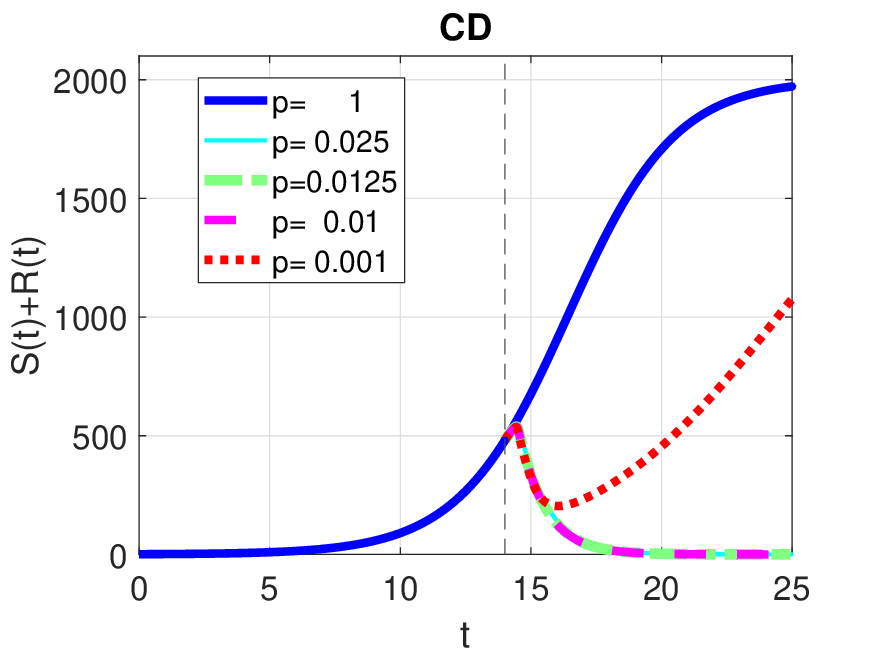}&
	\includegraphics[height=0.28\textwidth]{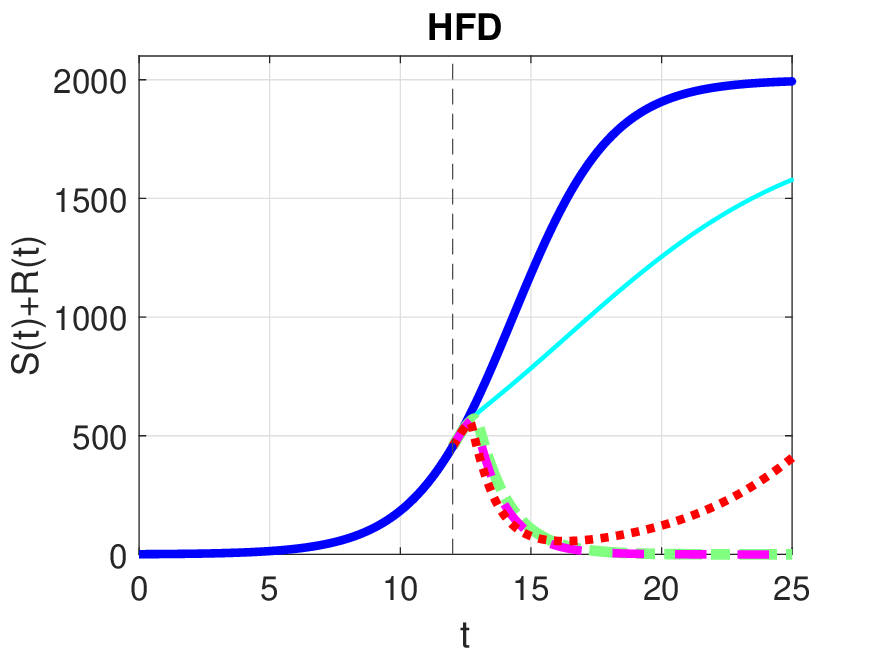}
	\end{array}
	$
	\caption{Scenario Ia: Sum of the sensitive $S$ and resistant $R$ tumor subpopulations over time $t$ for the constant treatment with different values of $p$ associated with CD (\textit{left}) and HFD (\textit{right}). We mark the time point at which treatment is started with a dashed vertical line. }\label{Fig_E1}
\end{figure}
\vspace{-1.5cm}
\begin{figure}[h!]
\centering
\subfloat[]{\includegraphics[scale=0.32]{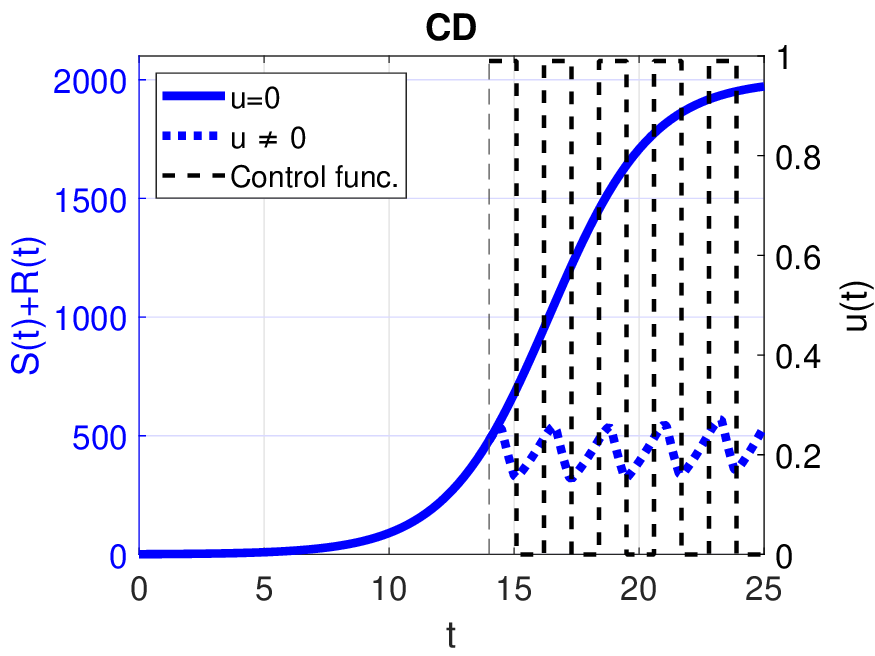}\label{Fig_I1a}}
\subfloat[]{\includegraphics[scale=0.32]{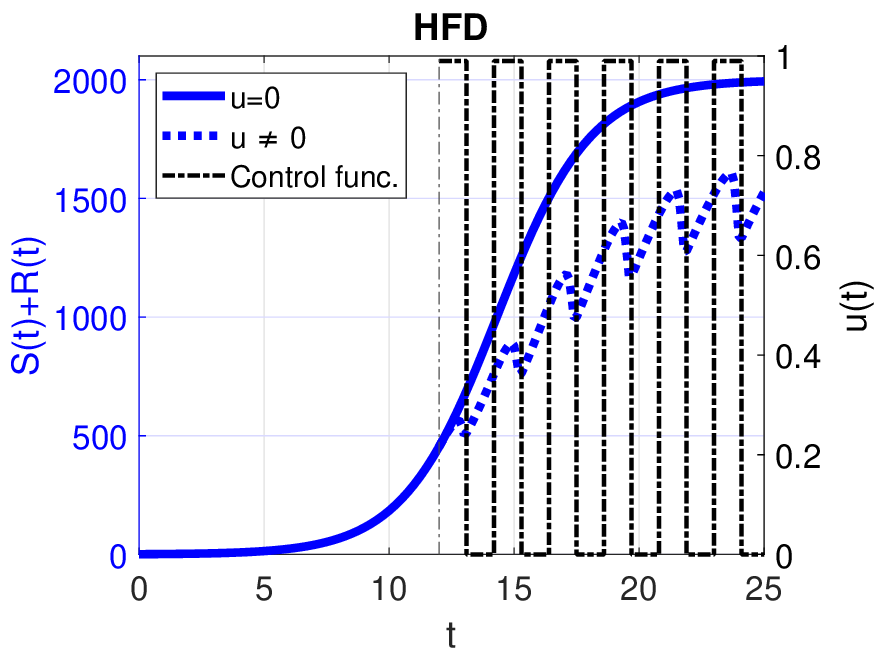}\label{Fig_I1b}}\\
\subfloat[]{\includegraphics[scale=0.32]{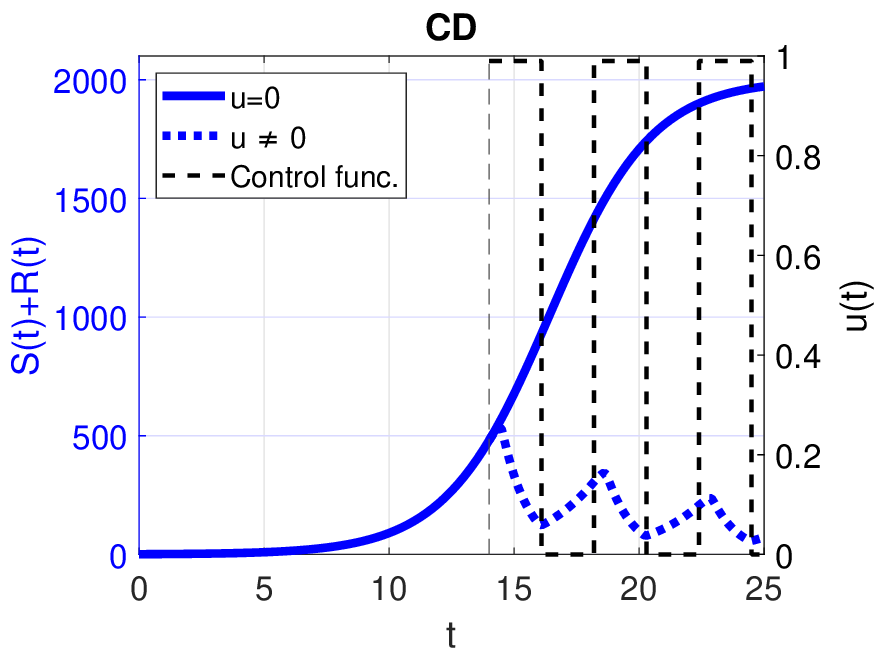}\label{Fig_I1c}}
\subfloat[]{\includegraphics[scale=0.32]{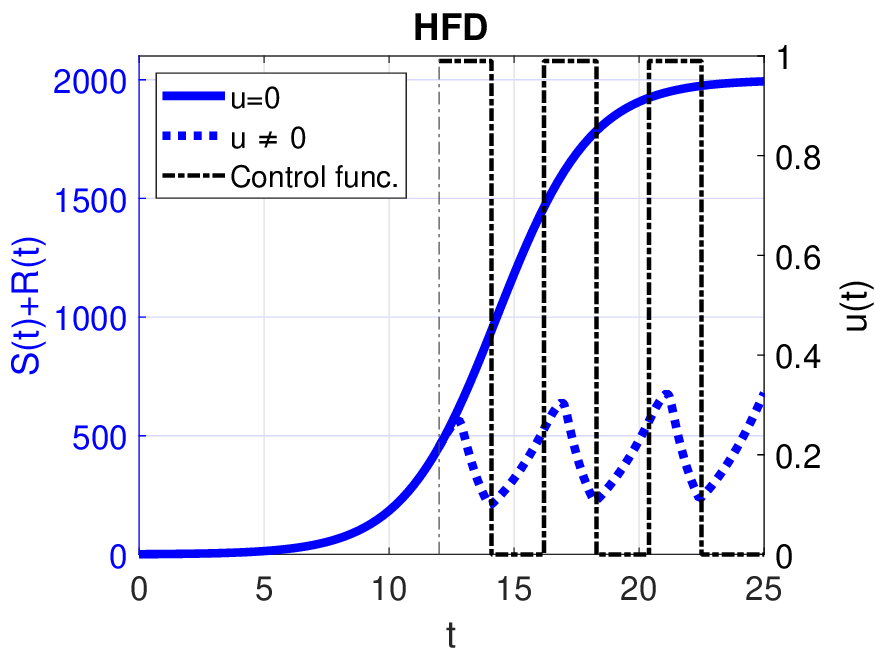}\label{Fig_I1d}}
\caption{Scenario Ia: Left axis refers to the sum of the sensitive $S$ and resistant $R$ tumor subpopulations over time $t$ for alternating treatment with \protect\subref{Fig_I1a}-\protect\subref{Fig_I1b} shorter phases (one day), \protect\subref{Fig_I1c}-\protect\subref{Fig_I1d} longer phases (two days) associated with CD (\textit{left}) and HFD (\textit{right}); right axis refers to treatment schedule. The dotted and solid curves refer to the tumor size with and without treatment, respectively. We mark the time point at which treatment is started with a dashed vertical line.}\label{Fig_I1}
\end{figure}
%\clearpage
%\newpage

%In addition to constant or optimal treatment, we investigate the contribution of intermittent treatment where the treatment is arranged in an on-off fashion based on the pre-decided schedule. We model two different schedules depending on the duration of the treatment and breaks. We start treatment on the same day as we do for optimal treatment. Treatment successfully eliminates the tumor for CD; but, its growth is accelerated when treatment is applied for shorter phases (see Fig.~\ref{Fig_I1a}). We obtain a smaller tumor population for CD in the long-run in case of longer phases (see Fig.~\ref{Fig_I1b})), since the second treatment phase starts before tumor cells increase sharply. Therefore, it's an indication that intermittent treatment is not better than optimal control for CD unless scheduling is planned based on tumor size. Intermittent treatment causes rapid increase in tumor size for HFD. We could achieve a smaller tumor size during treatment; but, it does not last too long. Thus, we can note that optimal/constant treatment does not eliminate the tumor, but stops its growth. 
Next, we proceed with alternating treatment. We simulate two different schedules with long and short treatment phases in Fig~\ref{Fig_I1}. The solid curve refers to the tumor size without treatment. For CD, alternating treatment with shorter phases causes tumor volume to stay within a range. Instead for HFD, tumor size grows over time with respect to the baseline tumor volume. On the other hand, treatment with longer phases leads to tumor reduction for CD, while it stays within a range for HFD.
In case of optimal treatment scheduling, we observe in Fig.~\ref{Fig_T1} that the tumors are eradicated for both CD and HFD  (solid line for no treatment, dashed line for optimal treatment with $\omega_R= \omega_S = \omega_U = 1$). For comparison, optimal control functions $u(t)$ are shown in Fig.~\ref{Fig_U1} for different values of the weight constants. The case $\omega_R, \omega_S > \omega_U$ leads in both CD and HFD to maximum treatment for almost the entire studied period, whereas treatment could be stopped earlier if $\omega_R, \omega_S < \omega_U$. In other words, penalizing tumor cells more than treatment results in longer treatment. There is no big difference between control functions in terms of diet, except for a slightly larger duration of treatment for HFD. To observe the effect of treatment in more detail, we present in Fig.~\ref{Fig_S1_Ia} the dynamics of all variables in the case $\omega_R= \omega_S = \omega_U = 1$. In this example, optimal treatment maintains the estrogen level between $a_3$ and $a_2$, so sensitive cells die, but no resistance occurs. This happens in spite of an increasing fat volume. Therefore, optimal treatment results in successful elimination of tumor without causing drug resistance. 

\newpage
\begin{figure}[h!]
	\centering
	$
	\begin{array}{cc}
	\includegraphics[height=0.28\textwidth]{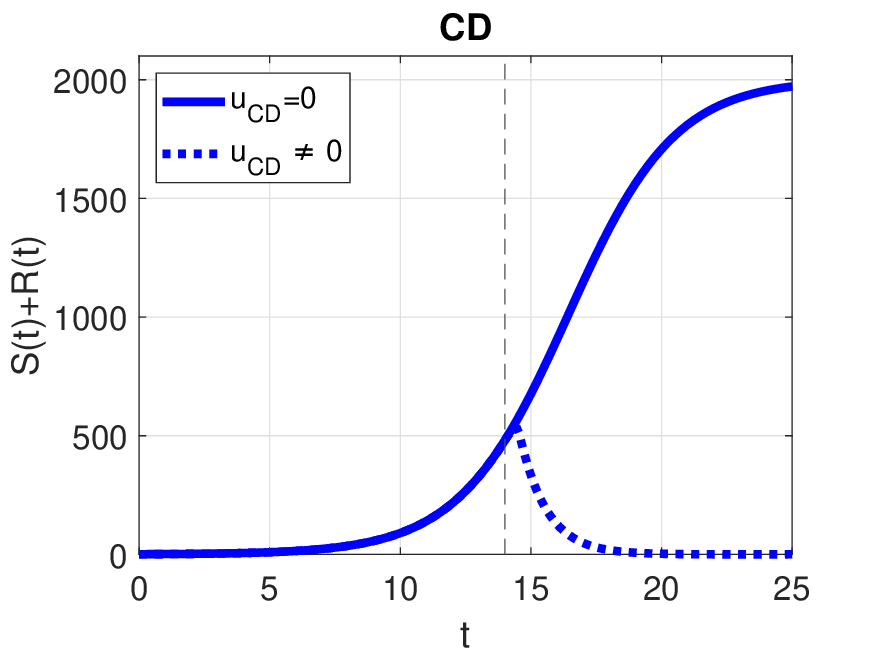}&
	\includegraphics[height=0.28\textwidth]{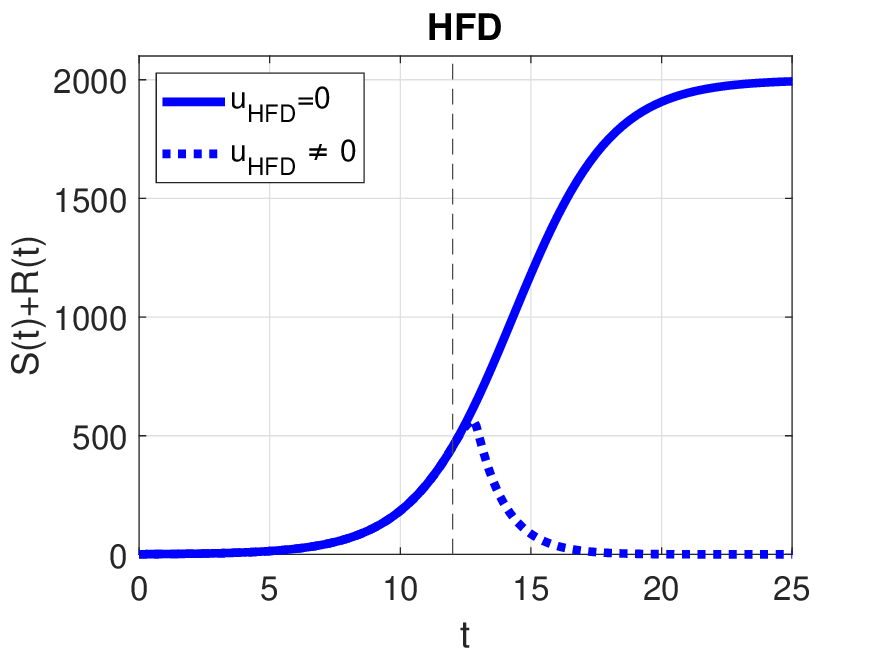}
	\end{array}
	$
	\caption{Scenario Ia: Sum of the sensitive $S$ and resistant $R$ tumor subpopulations over time $t$ for the optimal treatment with different values of $p$ associated with CD (\textit{left}) and HFD (\textit{right}). The solid curve refers to the tumor size without treatment. We mark the time point at which treatment is started with a dashed vertical line.}\label{Fig_T1}
\end{figure}
\vspace{-1.5cm}
\begin{figure}[h!]
	\centering
	$
	\begin{array}{ccc}
	\includegraphics[width=0.3\textwidth]{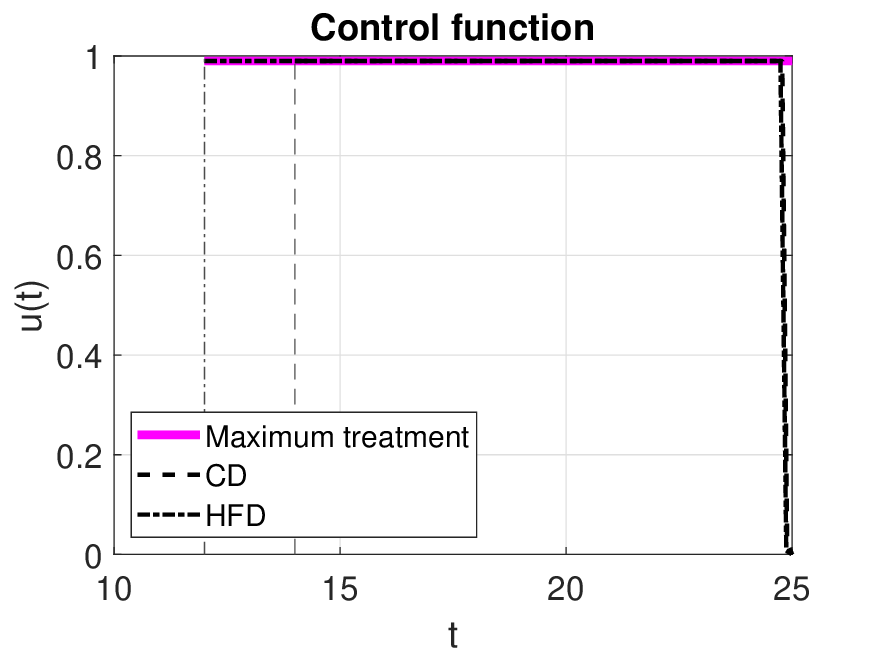}&
	\includegraphics[width=0.3\textwidth]{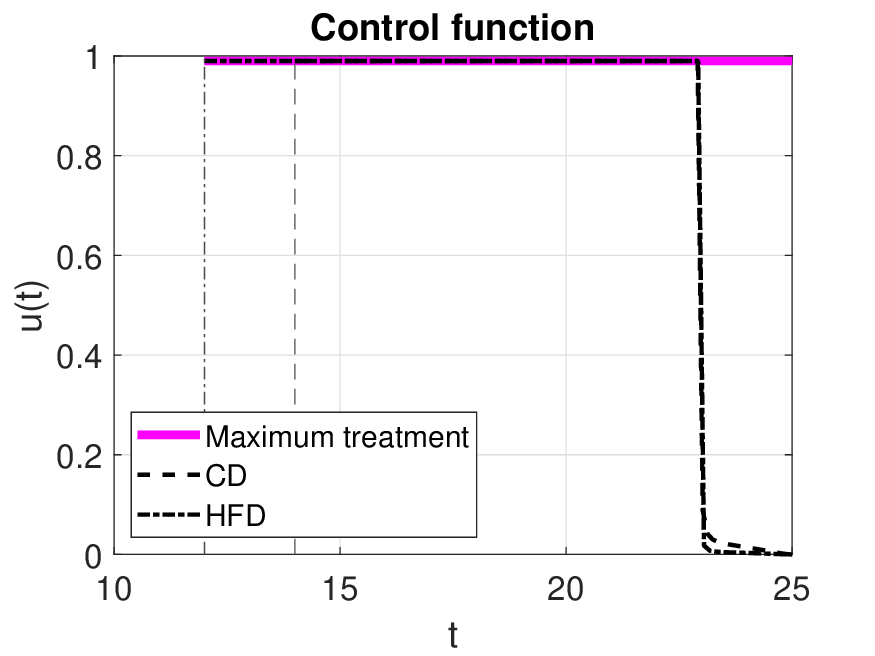}&
	\includegraphics[width=0.3\textwidth]{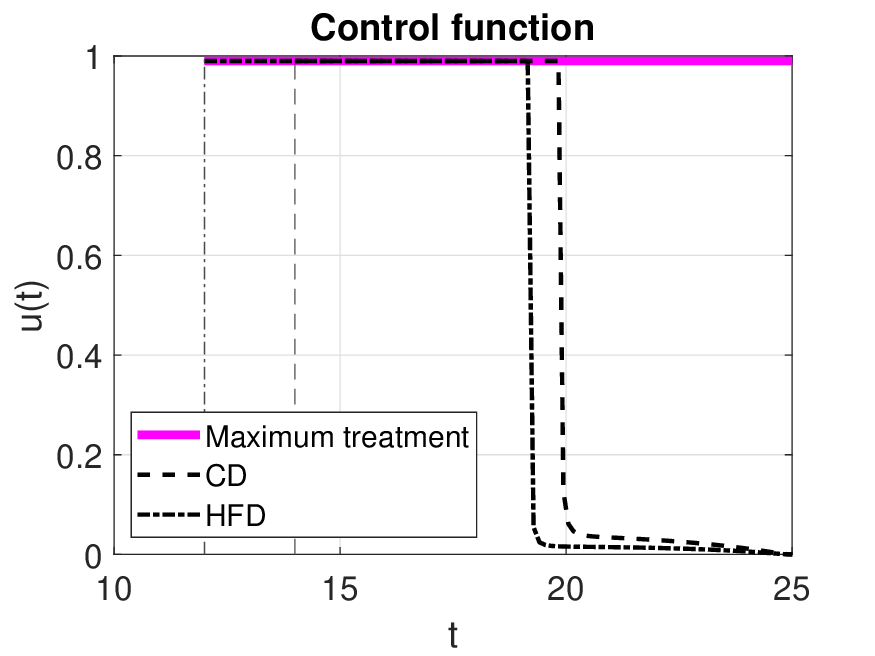}\\
		 {\scriptsize \textrm{(\textit{a})~$\omega_R=\omega_S=100, \omega_U=1$.}}& {\scriptsize \textrm{(\textit{b})~$\omega_R=\omega_S=\omega_U=1$.}} &{\scriptsize \textrm{(\textit{c})~$\omega_R=\omega_S=1, \omega_U=100$.}}		
	\end{array}
	$
	\caption{Scenario Ia: Optimal control function $u$ over time $t$ for three different combinations of weight coefficients $\omega_R, \omega_S, \omega_U$. Dashed and dash-dotted curves refer to the optimal treatment schedules for CD and HFD, respectively. Solid line denotes the maximum treatment. We mark the time point at which treatment is started with dashed and dash-dotted vertical lines for CD and HFD, respectively.}\label{Fig_U1}
\end{figure}
\begin{figure}
	\centering
	$
	\begin{array}{c}
	\includegraphics[height=0.54\textwidth]{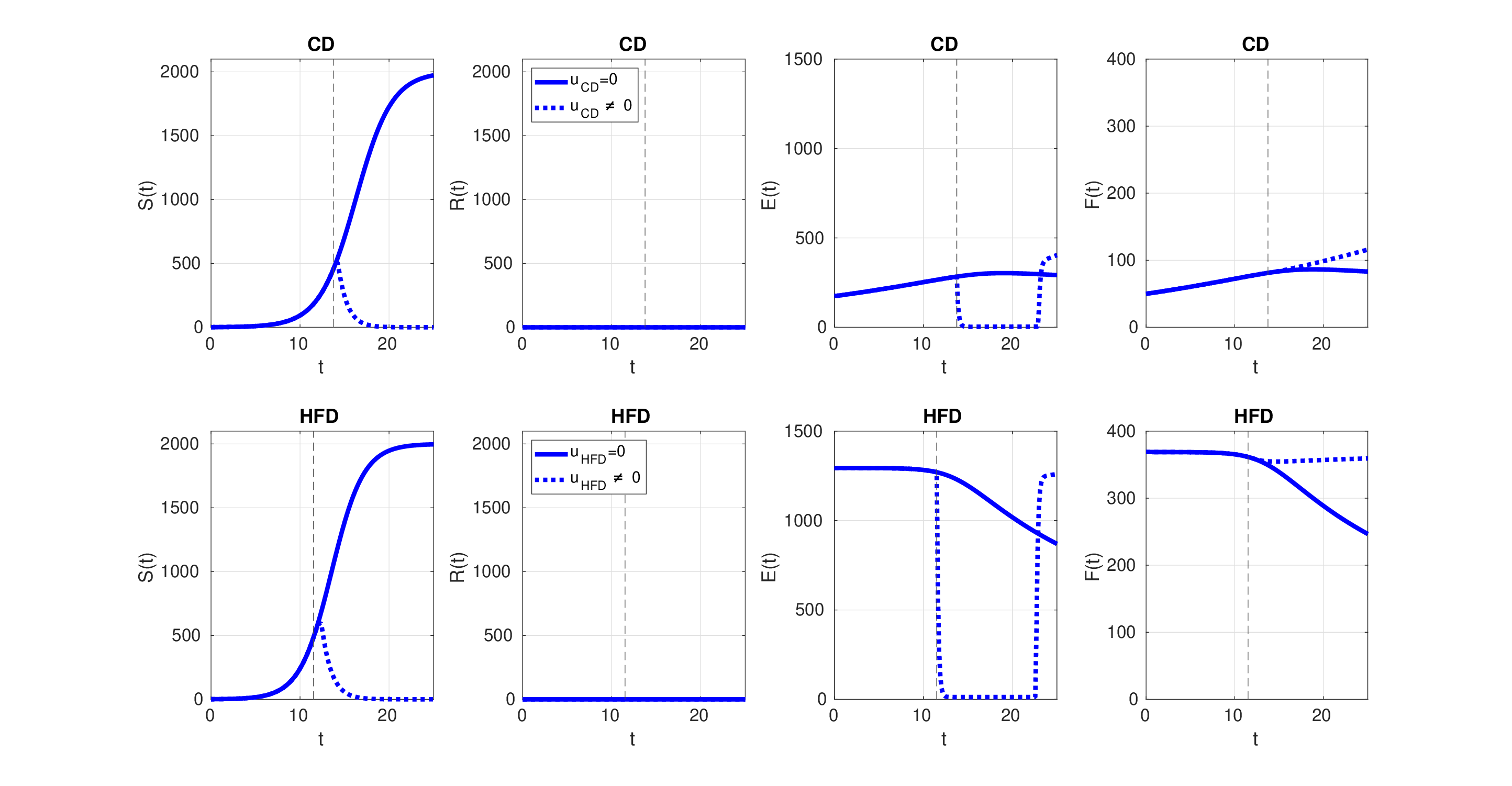}
	\end{array}
	$
	\caption{Scenario Ia: Dynamics of model variables $S$, $R$, $E$ and $F$ over time $t$ associated with CD (\textit{1st row}) and HFD (\textit{2nd row}). The solid curve refers to the tumor size without treatment, dotted curve corresponds to results for optimal treatment.}\label{Fig_S1_Ia}
\end{figure}
\clearpage

\begin{comment}
Furthermore, the value of the cost functional at each iteration is shown in Fig.~\ref{Fig_J1}. Termination criterion is set based on the relative error rather than the number of iterations within the FBS method, so different number of iterations are required for CD and HFD. We see that the value of $J$ decreases during iteration. More iterations are required for HFD, since it is numerically hard to achieve convergence. 
\begin{figure}[h!]
	\centering
	$
	\begin{array}{c}
	\includegraphics[height=0.4\textwidth]{Figures/ocp/J_Case1.eps}
	\end{array}
	$
	\caption{Scenario I: Cost functional.}\label{Fig_J1}
\end{figure}
\end{comment}
%\clearpage
%%%%%%%%%%%%%%%%%%%%%%%%%%%%%%%%%%%%%%%%%
%%%%%%%%%%%%%%%%%%%%%%%%%%%%%%%%%%%%%%%%%
%\clearpage
\subsubsection{Scenario Ib: De novo resistance with \texorpdfstring{$R_0=0.25$}.}
Next we investigate the influence of a preexisting resistant sub-population on the success of constant, alternating and optimal anti-hormonal treatment schedules. In this case, endocrine resistance arise by clonal selection of cells that are endocrine independent for some reasons. 
In Fig.~\ref{Fig_E1b} we show response to constant treatment by displaying the change in tumor size for different values of $p$. We observe that constant treatment is unsuccessful to eliminate the tumor.
\begin{comment}
\begin{figure}[h!]
	\centering
	$
	\begin{array}{cc}
	\includegraphics[height=0.35\textwidth]{NoControl_Tumor_Case_1bCD.eps}&
	\includegraphics[height=0.35\textwidth]{NoControl_Tumor_Case_1bHFD.eps}
	\end{array}
	$
	\caption{Scenario Ib: Sum of the sensitive $S$ and resistant $R$ tumor subpopulations over time $t$ for the constant treatment with different values of $p$ associated with CD (\textit{left}) and HFD (\textit{right}). We mark the time point at which treatment is started with a dashed vertical line.}\label{Fig_E1b}
\end{figure}
\end{comment}

We plot the change of tumor size over time for short and long alternating treatment phases in Fig.~\ref{Fig_S1_pre} where 75$\%$ of the cells are sensitive and 25$\%$ are resistant at the beginning of the simulation. For shorter drug holidays, tumor size increases compared to the initial tumor size. For longer on-off periods, the tumor volume remains within a bounded range for both CD and HFD but the oscillations between remission and growth are bigger in the HFD case. Moreover, the final tumor volume is larger for both cases in comparison with the case of no preexisting resistance showed in Fig.~\ref{Fig_I1}.
\begin{comment}
\begin{figure}[h!]
\subfloat[][]{\includegraphics[scale=0.4]{Alt_Tumor_Case1ab_CD.eps}\label{Fig_I1ab}}
\subfloat[][]{\includegraphics[scale=0.4]{Alt_Tumor_Case1ab_HFD.eps}\label{Fig_I1bb}}\\
\subfloat[][]{\includegraphics[scale=0.4]{Alt_Tumor_Case1bb_CD.eps}\label{Fig_I1cb}}
\subfloat[][]{\includegraphics[scale=0.4]{Alt_Tumor_Case1bb_HFD.eps}\label{Fig_I1db}}
\caption{Scenario Ib: Left axis refers to the sum of the sensitive $S$ and resistant $R$ tumor subpopulations over time $t$ for alternating treatment with \protect\subref{Fig_I1ab}-\protect\subref{Fig_I1bb} shorter phases (one day), \protect\subref{Fig_I1cb}-\protect\subref{Fig_I1db} longer phases (two days) associated with CD (\textit{left}) and HFD (\textit{right}); right axis refers to treatment schedule. The dotted and solid curves refer to the tumor size with and without treatment, respectively. We mark the time point at which treatment is started with a dashed vertical line.}\label{Fig_S1_pre}
\end{figure}
\end{comment}

\clearpage
\begin{figure}[h!]
	\centering
	$
	\begin{array}{cc}
	\includegraphics[height=0.28\textwidth]{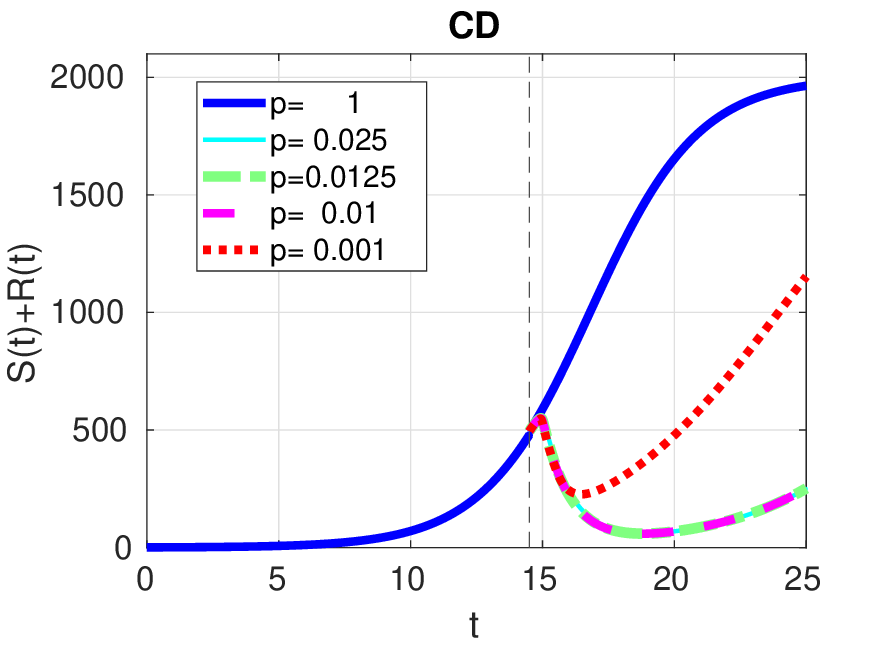}&
	\includegraphics[height=0.28\textwidth]{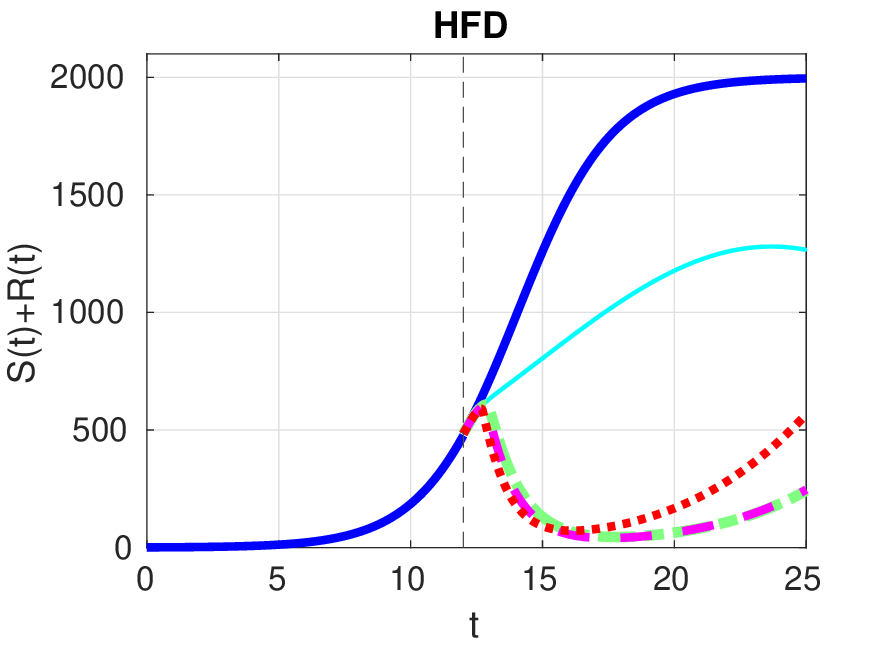}
	\end{array}
	$
	\caption{Scenario Ib: Sum of the sensitive $S$ and resistant $R$ tumor subpopulations over time $t$ for the constant treatment with different values of $p$ associated with CD (\textit{left}) and HFD (\textit{right}). We mark the time point at which treatment is started with a dashed vertical line.}\label{Fig_E1b}
\end{figure}
\vspace{-1cm}
\begin{figure}[h!]
\centering
\subfloat[][]{\includegraphics[scale=0.32]{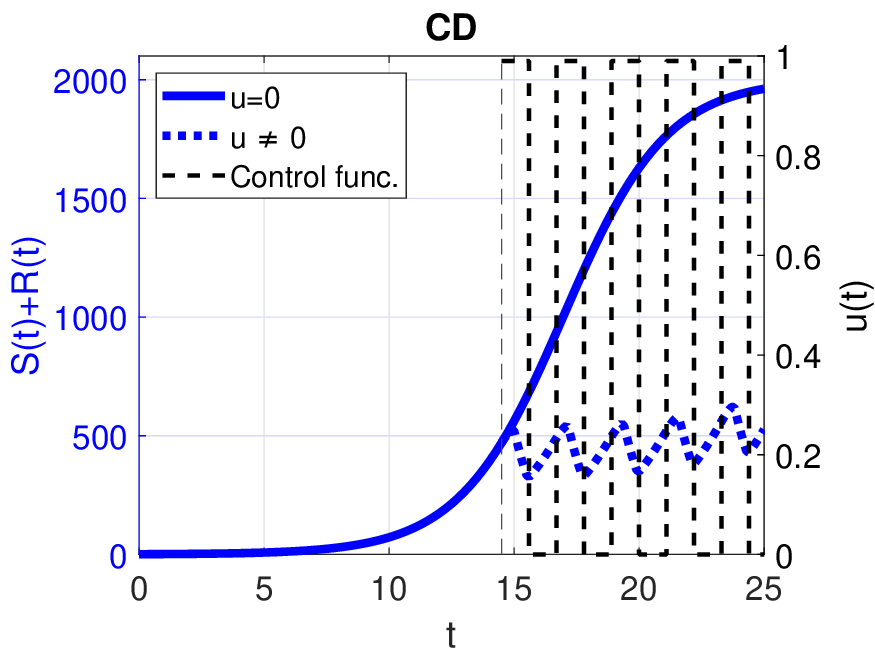}\label{Fig_I1ab}}
\subfloat[][]{\includegraphics[scale=0.32]{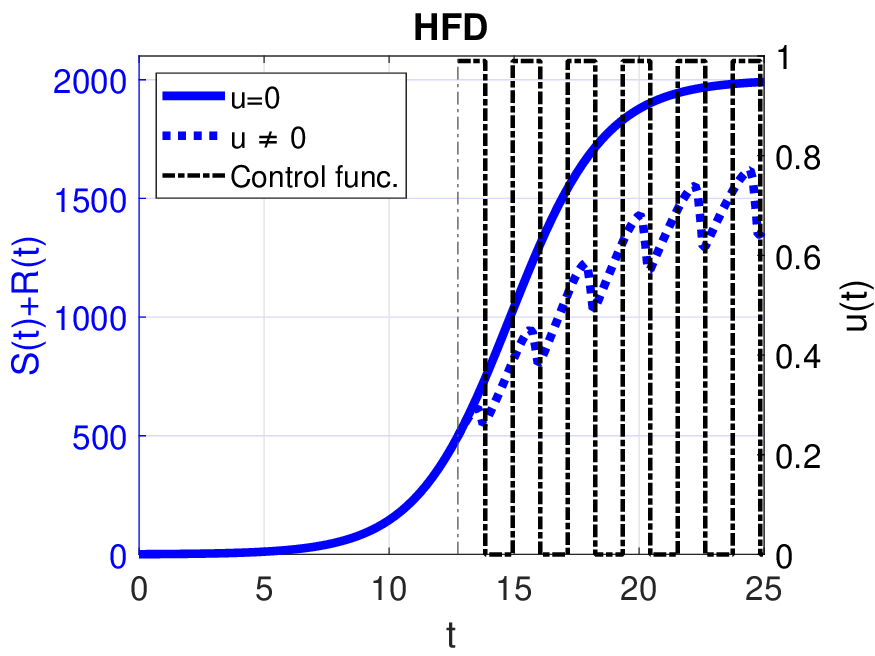}\label{Fig_I1bb}}\\
\subfloat[][]{\includegraphics[scale=0.32]{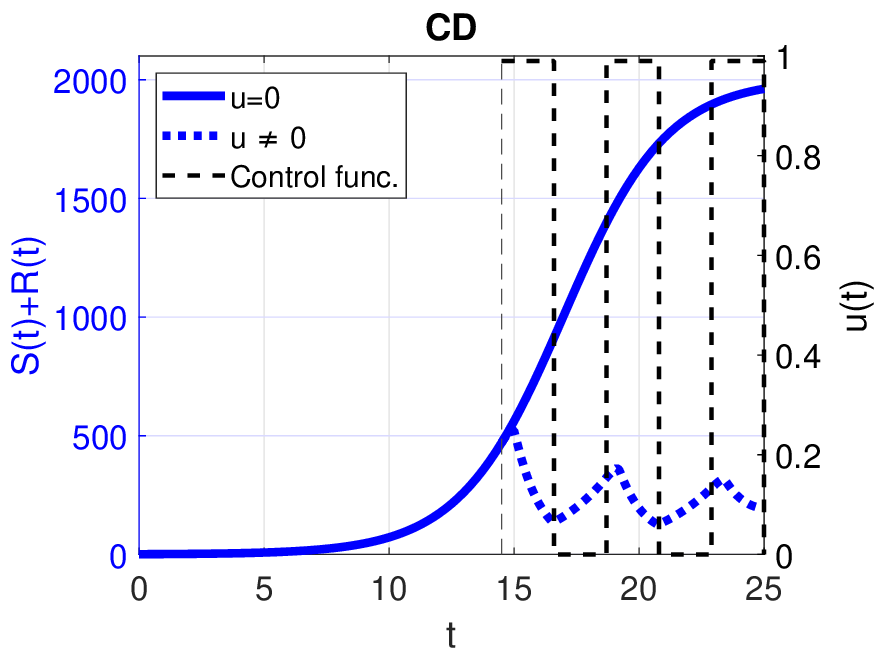}\label{Fig_I1cb}}
\subfloat[][]{\includegraphics[scale=0.32]{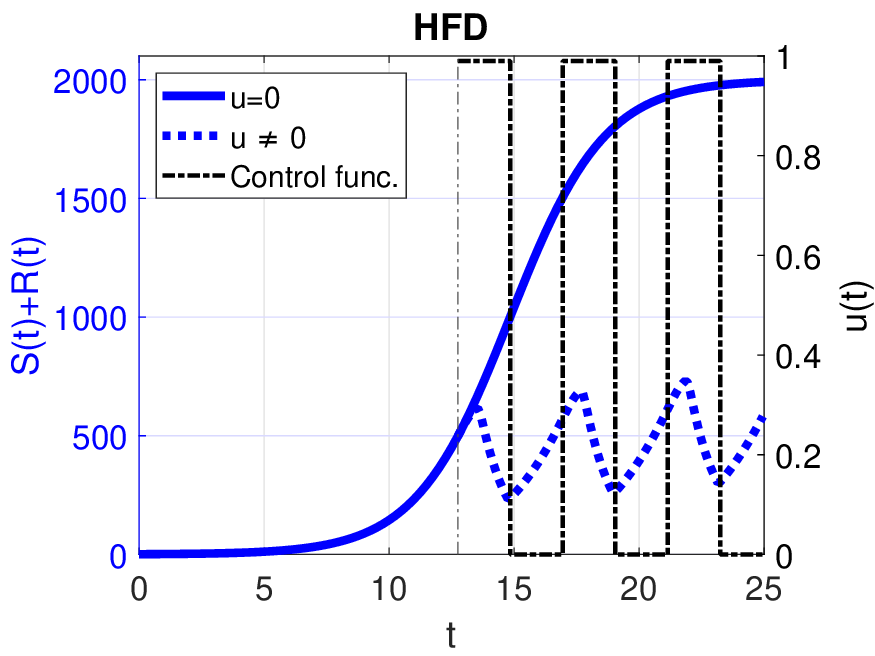}\label{Fig_I1db}}
\caption{Scenario Ib: Left axis refers to the sum of the sensitive $S$ and resistant $R$ tumor subpopulations over time $t$ for alternating treatment with \protect\subref{Fig_I1ab}-\protect\subref{Fig_I1bb} shorter phases (one day), \protect\subref{Fig_I1cb}-\protect\subref{Fig_I1db} longer phases (two days) associated with CD (\textit{left}) and HFD (\textit{right}); right axis refers to treatment schedule. The dotted and solid curves refer to the tumor size with and without treatment, respectively. We mark the time point at which treatment is started with a dashed vertical line.}\label{Fig_S1_pre}
\end{figure}

%\newpage
We plot the results obtained with optimal treatment in Fig.~\ref{Fig_T1b}. While treatment decreases the tumor volume in both CD and HFD cases, resistance cells proliferate and drug resistance occurs. We present the temporal evolution of the model variables in detail in Fig.~\ref{Fig_S1b}. We can see that sensitive cells are killed but resistant cells increase in size as a result of drug-resistance. Optimal control profiles are similar to the case of adaptive resistance in Fig.~\ref{Fig_U1}, so we do not present it here.

\begin{figure}[h!]
	\centering
	$
	\begin{array}{cc}
	\includegraphics[height=0.28\textwidth]{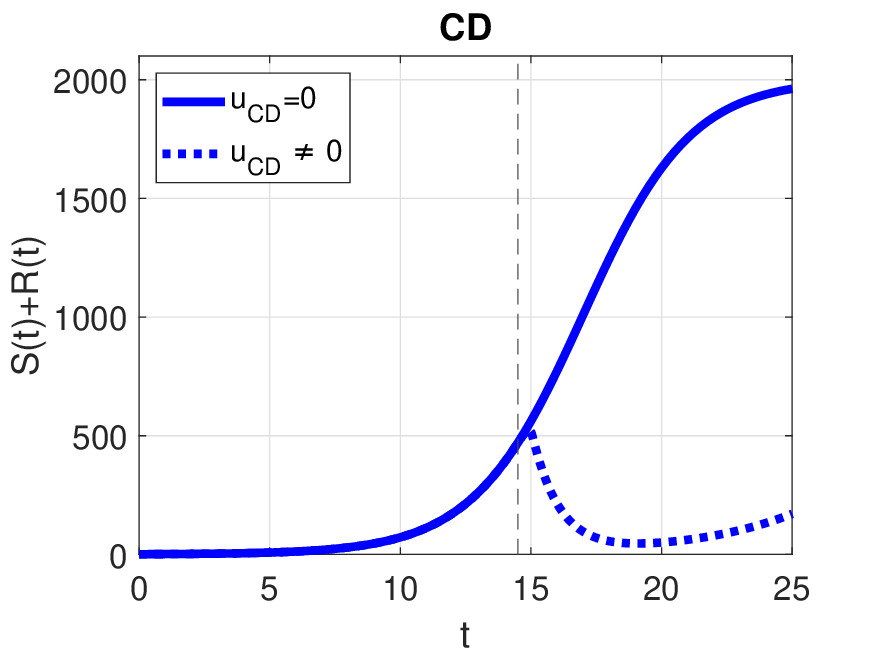}&
	\includegraphics[height=0.28\textwidth]{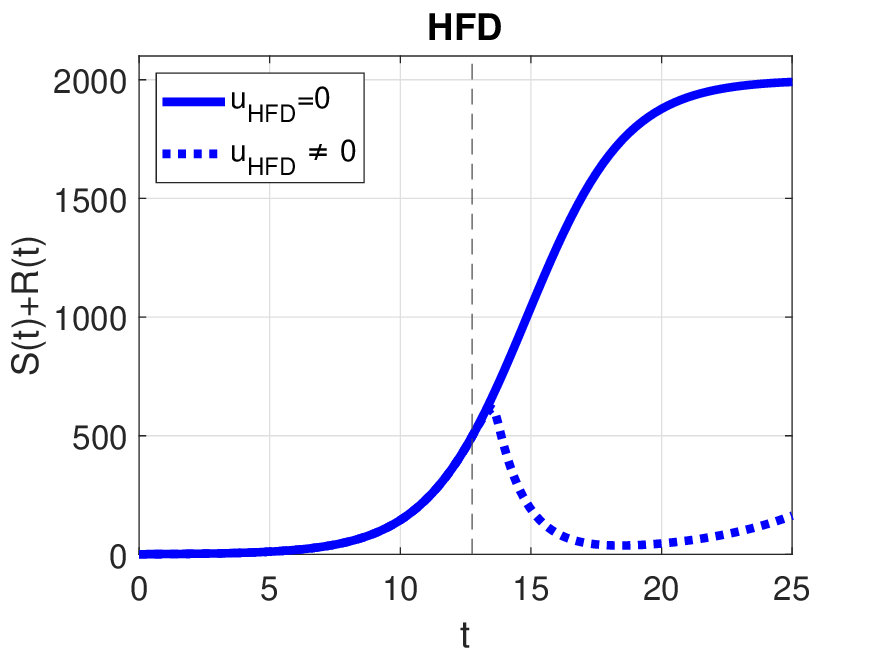}
	\end{array}
	$
	\caption{Scenario Ib: Sum of the sensitive $S$ and resistant $R$ tumor subpopulations over time $t$ for the optimal treatment with different values of $p$ associated with CD (\textit{left}) and HFD (\textit{right}). The solid curve refers to the tumor size without treatment. We mark the time point at which treatment is started with a dashed vertical line.}\label{Fig_T1b}
\end{figure}
\begin{figure}[h!]
	\centering
	$
	\begin{array}{c}
	\includegraphics[height=0.53\textwidth]{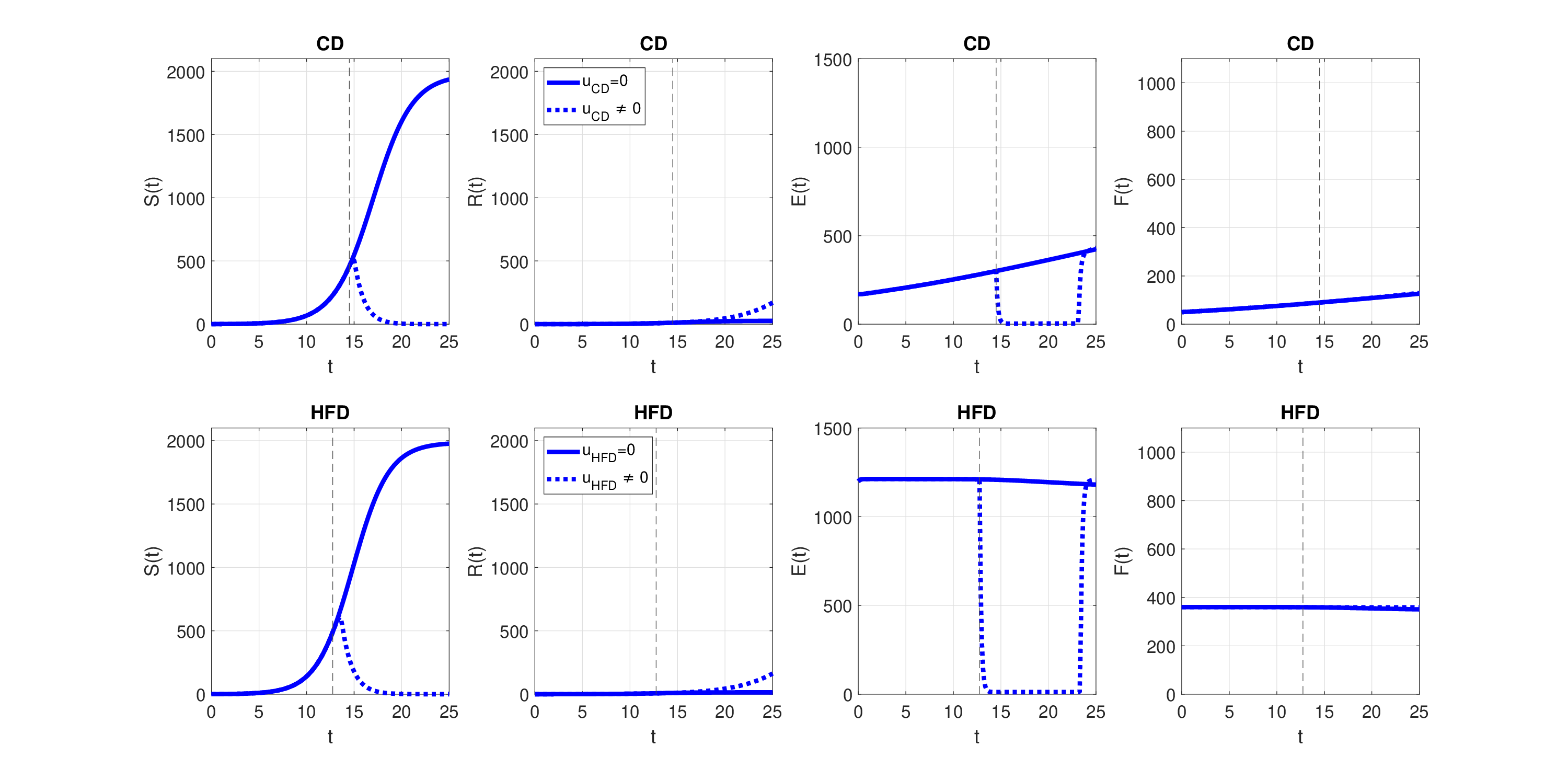}
	\end{array}
	$
	\caption{Scenario Ib: Dynamics of model variables $S$, $R$, $E$ and $F$ over time $t$ associated with CD (\textit{1st row}) and HFD (\textit{2nd row}). The solid curve refers to the tumor size without treatment, dotted curve corresponds to results for optimal treatment.}\label{Fig_S1b}
\end{figure}

\begin{comment}
On the other hand, similar profiles are achieved for CD and HFD (See Fig.~\ref{Fig_C1b}). However, longer treatment is required for HFD in comparison with CD. Even though, treatment could be relaxed as time passes, drug resistance appears to have been inevitable. 
\begin{figure}[h!]
	\centering
	$
	\begin{array}{c}
	\includegraphics[height=0.35\textwidth]{Figures/ocp/Control_Case1b_pre.eps}
	\end{array}
	$
	\caption{Scenario Ib: Optimal control $u(t)$ for $\omega_R=\omega_S=\omega_U=1$.}\label{Fig_C1b}
\end{figure}
\end{comment}

%%%%%%%%%%%%%%%%%%%%%%%%%%%%%%%%%%%%%%%%%
%%%%%%%%%%%%%%%%%%%%%%%%%%%%%%%%%%%%%%%%%
%\clearpage
\subsection{Scenario II: Adaptive resistance with \texorpdfstring{$a_2=10$~pg/g, $a_3=1$~pg/g, $k_3 =\frac{k_1}{2}$}.}
%We proceed with the case $(a_2, a_3) = (10, 0.001)$, $k_2^{*}=k_2/2$ and $m=2$ to investigate the effect of fat growth rate to treatment in Fig.~\ref{Fig_E2}. The parameter $m$ is taken as two to eliminate an infeasible equilibrium point (see Sec.~\ref{model_subsec2}). As different from Fig.~\ref{Fig_E1}, we observe some oscillations in the solution curves and amount of fat in the system is less than the case with $k_2$ in Table~\ref{Tab2}, as expected. This results in smaller estrogen concentration, so slower tumor growth rate is achieved. Sensitive cells for CD are eliminated for $p \leq 1/32$, as we observe in Scenario I. However, treatment with $p=1/40$ for HFD decreases the number of sensitive cells to a value smaller than the one we observe when treatment was started. We can deduce that treatment could lead to more decrease in the tumor size when the growth rate of fat cells is halved.
Here we investigated a situation with no preexisting resistant cells but where the estrogen thresholds for which the cells die or are converted to resistant are closer than in the previous case. Fig.~\ref{Fig_E2} shows the case for constant treatment. All tested treatment cases, except $p=0.001$, result in tumor elimination for CD, but $p=0.01$ leads to tumor reduction until day 25 and higher values of $p$ suppresses tumor growth for HFD.
\begin{comment}
\begin{figure}[h!]
	\centering
	$
	\begin{array}{cc}
	\includegraphics[height=0.35\textwidth]{NoControl_Tumor_Case2a_CD.eps}&
	\includegraphics[height=0.35\textwidth]{NoControl_Tumor_Case2a_HFD.eps}
	\end{array}
	$
	\caption{Scenario II: Sum of the sensitive $S$ and resistant $R$ tumor subpopulations over time $t$ for the constant treatment with different values of $p$ associated with CD (\textit{left}) and HFD (\textit{right}). We mark the time point at which treatment is started with a dashed vertical line.} \label{Fig_E2}
\end{figure}
\end{comment}
Alternating treatment instead, leads to oscillations in tumor size but with a decreasing trend for CD, whereas for HFD a sharp increase in tumor population is observed during drug holidays (see Fig.~\ref{Fig_I2}). 
\begin{comment}
\begin{figure}[h!]
	\centering
	$
	\begin{array}{cc}
	\includegraphics[height=0.35\textwidth]{Alt_Tumor_Case2a_CD.eps} \label{Fig_I22a}&
	\includegraphics[height=0.35\textwidth]{Alt_Tumor_Case2a_HFD.eps} \label{Fig_I22b}
	\end{array}
	$
	\caption{Scenario II: Left axis refers to the sum of the sensitive $S$ and resistant $R$ tumor subpopulations over time $t$ for alternating treatment associated with CD (\textit{left}) and HFD (\textit{right}); right axis refers to treatment schedule. The dotte and solid curves refer to the tumor size with and without treatment, respectively. We mark the time point at which treatment is started with a dashed vertical line.}\label{Fig_I2}
\end{figure}
\end{comment}

\begin{comment}
\begin{figure}[h!]
	\centering
	$
	\begin{array}{c}
	\includegraphics[height=0.35\textwidth]{Figures/ocp/J_Case7c.eps}
	\end{array}
	$
	\caption{Scenario II: Cost functional.}
\end{figure}
\end{comment}

The results for optimal treatment shown in Fig.~\ref{Fig_T2} reveal that tumor is eradicated for CD, similar to Fig.~\ref{Fig_T1}. Interestingly, the tumor remains at the end of the treatment for HFD.
\begin{comment}
\begin{figure}[h!]
	\centering
	$
	\begin{array}{cc}
	\includegraphics[height=0.35\textwidth]{Tumor_Case2a_CD.eps}&
	\includegraphics[height=0.35\textwidth]{Tumor_Case2a_HFD.eps}
	\end{array}
	$
	\caption{Scenario II: Sum of the sensitive $S$ and resistant $R$ tumor subpopulations over time $t$ for the optimal treatment with different values of $p$ associated with CD (\textit{left}) and HFD (\textit{right}). The solid curve refers to the tumor size without treatment. We mark the time point at which treatment is started with a dashed vertical line.}\label{Fig_T2}
\end{figure}
\end{comment}
%%%%
\newpage
\vspace{-1cm}
\begin{figure}[h!]
	\centering
	$
	\begin{array}{cc}
	\includegraphics[height=0.28\textwidth]{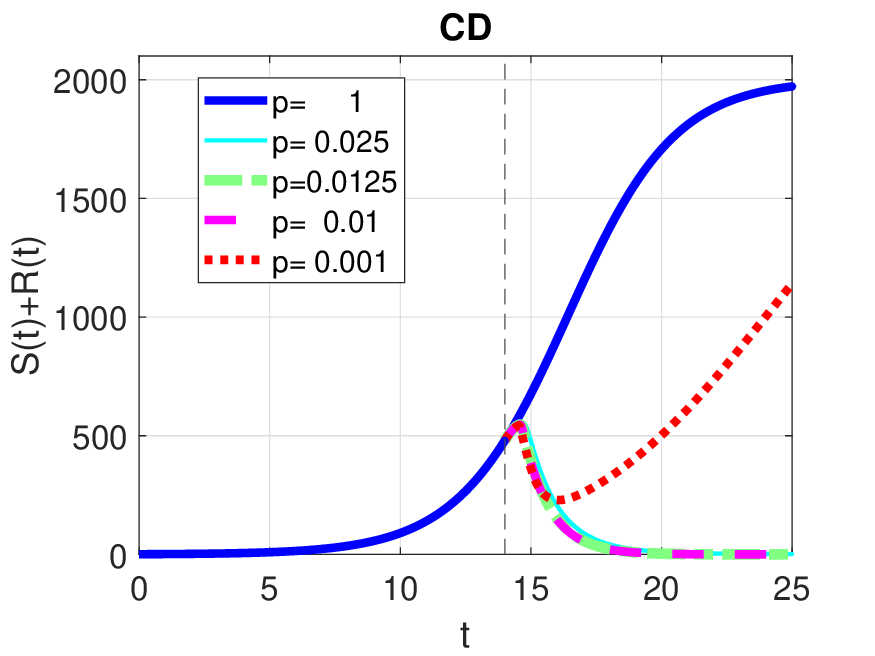}&
	\includegraphics[height=0.28\textwidth]{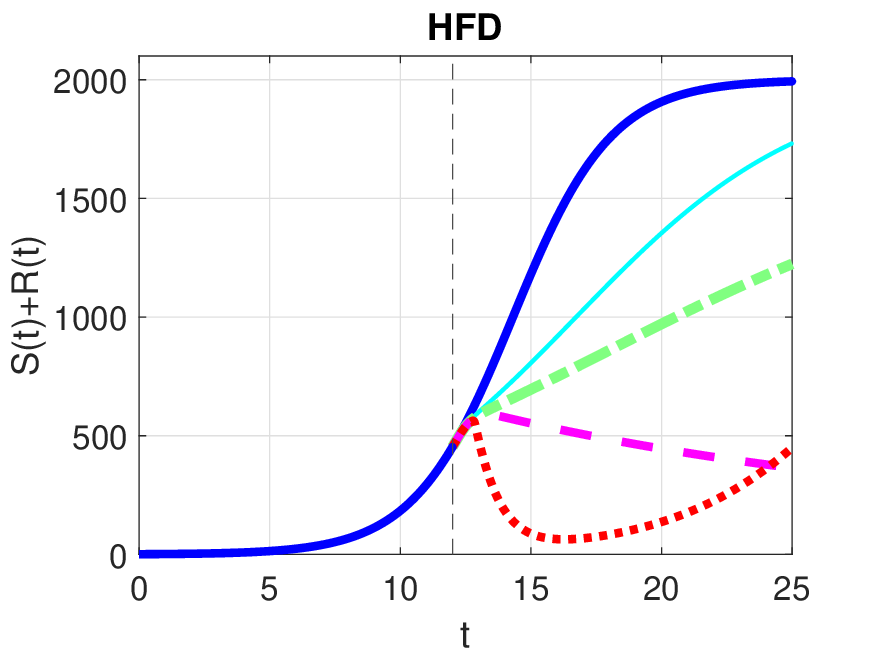}
	\end{array}
	$
	\caption{Scenario II: Sum of the sensitive $S$ and resistant $R$ tumor subpopulations over time $t$ for the constant treatment with different values of $p$ associated with CD (\textit{left}) and HFD (\textit{right}). We mark the time point at which treatment is started with a dashed vertical line.} \label{Fig_E2}
\end{figure}
\vspace{-1.5cm}
\begin{figure}[h!]
	\centering
	$
	\begin{array}{cc}
	\includegraphics[height=0.28\textwidth]{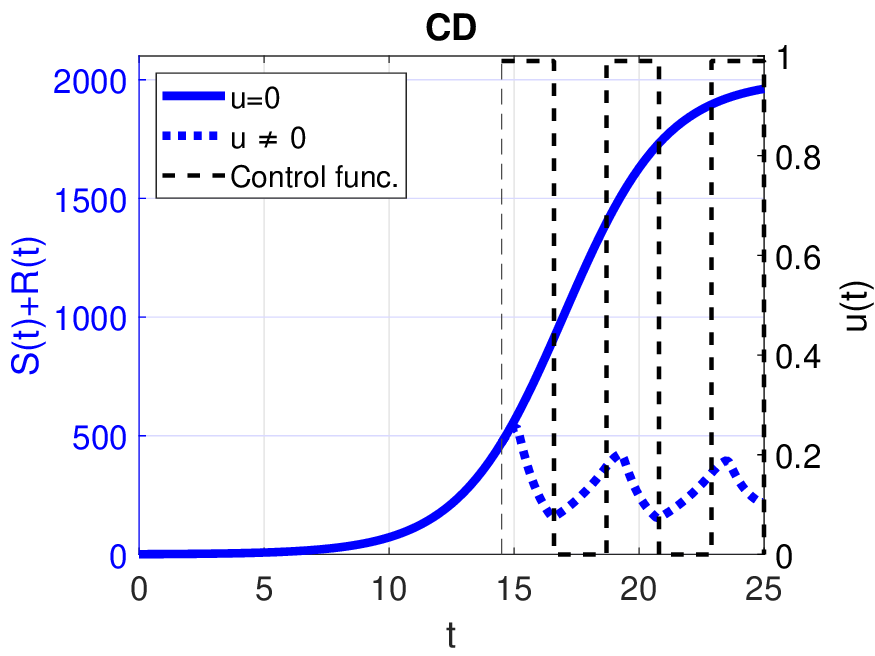} \label{Fig_I22a}&
	\includegraphics[height=0.28\textwidth]{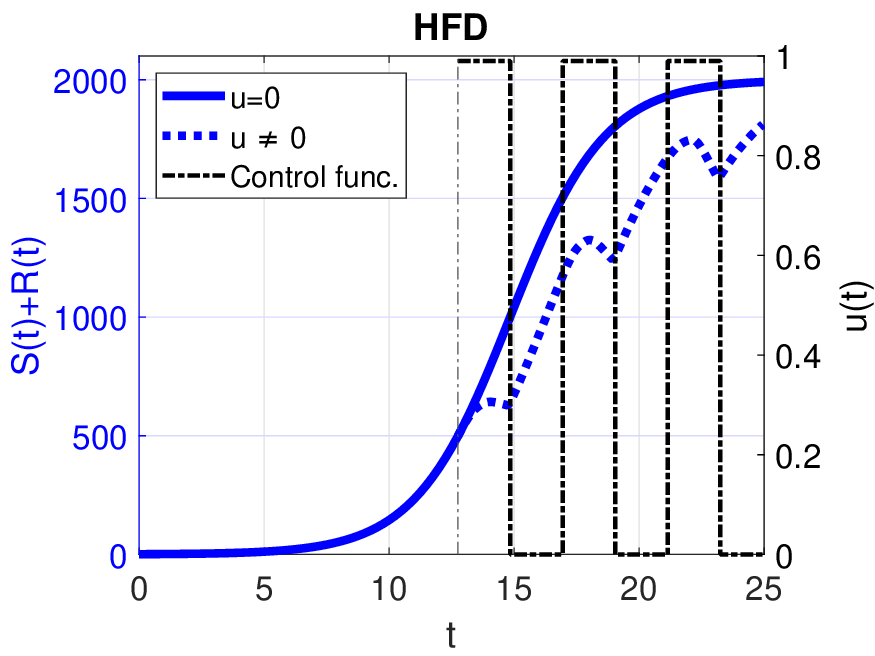} \label{Fig_I22b}
	\end{array}
	$
	\caption{Scenario II: Left axis refers to the sum of the sensitive $S$ and resistant $R$ tumor subpopulations over time $t$ for alternating treatment associated with CD (\textit{left}) and HFD (\textit{right}); right axis refers to treatment schedule. The dotte and solid curves refer to the tumor size with and without treatment, respectively. We mark the time point at which treatment is started with a dashed vertical line.}\label{Fig_I2}
\end{figure}

%\clearpage
We compare optimal treatment schedules in Fig.~\ref{Fig_C2} for CD and HFD. We note that treatment must be applied for long time for HFD than CD, while it could be relaxed earlier for CD. Thus, optimal anti-hormonal treatment gives the most promising results among three different treatment choices in terms of reduction in tumor volume and time to lessen treatment could also be seen.
\begin{figure}[h!]
	\centering
	$
	\begin{array}{cc}
	\includegraphics[height=0.28\textwidth]{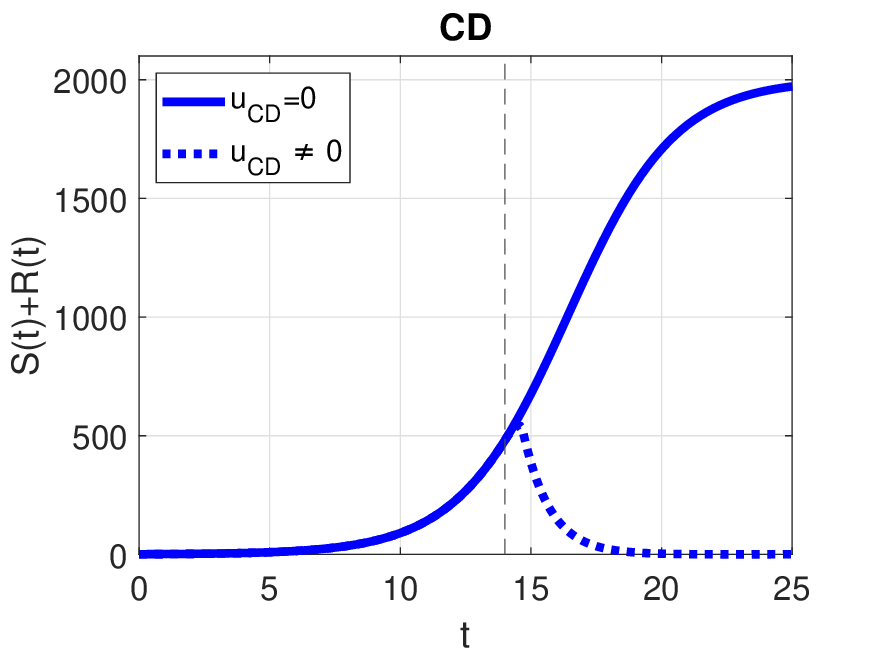}&
	\includegraphics[height=0.28\textwidth]{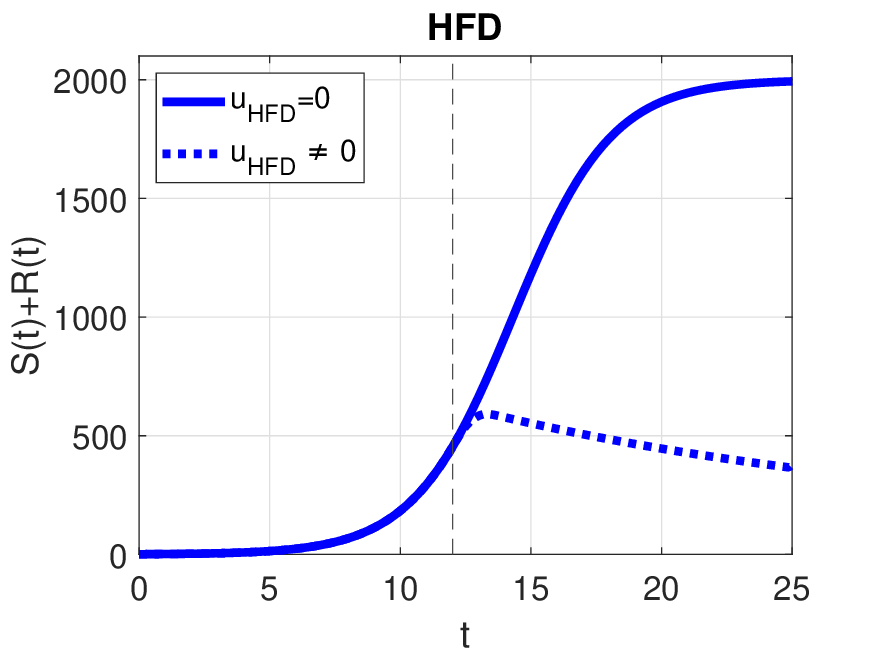}
	\end{array}
	$
	\caption{Scenario II: Sum of the sensitive $S$ and resistant $R$ tumor subpopulations over time $t$ for the optimal treatment with different values of $p$ associated with CD (\textit{left}) and HFD (\textit{right}). The solid curve refers to the tumor size without treatment. We mark the time point at which treatment is started with a dashed vertical line.}\label{Fig_T2}
\end{figure}
\begin{figure}[h!]
	\centering
	$
	\begin{array}{c}
	\includegraphics[height=0.28\textwidth]{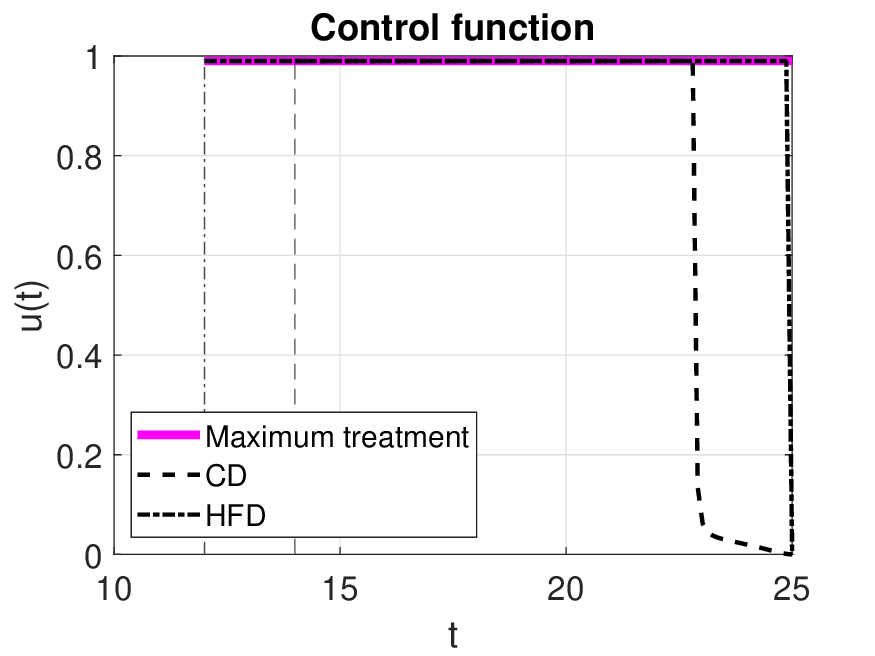}
	\end{array}
	$
	\caption{Scenario II: Optimal control function $u$ over time $t$ with $\omega_R=\omega_S=\omega_U=1$. Dashed and dash-dotted curves refer to the optimal treatment schedules for CD and HFD, respectively. Solid line denotes the maximum treatment. We mark the time point at which treatment is started with dashed and dash-dotted vertical lines for CD and HFD, respectively.}\label{Fig_C2}
\end{figure}

\begin{comment}
\begin{figure}[h!]
	\centering
	$
	\begin{array}{ccc}
	\includegraphics[width=0.4\textwidth]{Figures/ocp/Control_Case2b.eps}&
	\includegraphics[width=0.4\textwidth]{Figures/ocp/Control_Case2a.eps}&
	\includegraphics[width=0.4\textwidth]{Figures/ocp/Control_Case2c.eps}\\
		 {\scriptsize \textrm{(\textit{a})~$\omega_R=\omega_S=10, \omega_U=1$.}}& {\scriptsize \textrm{(\textit{b})~$\omega_R=\omega_S=\omega_U=1$.}} &{\scriptsize \textrm{(\textit{c})~$\omega_R=\omega_S=1, \omega_U=10$.}}		
	\end{array}
	$
	\caption{Scenario II: Control function $u$.}
\end{figure}
\end{comment}

%%%%%%%%%%%%%%%%%%%%%%%%%%%%%%%%%%%%%%%%%
%%%%%%%%%%%%%%%%%%%%%%%%%%%%%%%%%%%%%%%%%
%\clearpage
\subsection{Scenario III: Adaptive resistance with \texorpdfstring{$a_2=a_3=10$ pg/g, $k_3 = \frac{k_1}{4}$}.}
Finally, we investigate a scenario where death and conversion terms are equivalent, namely $a_2 = a_3 = 10$. We present temporal evolution of all model variables for constant treatment in Fig.~\ref{Fig_E4}. Estrogen level is successfully decreased, but it leads to drug resistance for CD for all choices of the parameter $p$. On the other hand, for HFD, the case $p=0.025$ is not strong enough to kill sensitive cells, so resistance cells do not proliferate. However, other treatment choices result in resistance and treatment fails.
\begin{comment}
\begin{figure}[h!]
	\centering
	$
	\begin{array}{c}
	\includegraphics[height=0.4\textwidth]{Figures/Extension/NoControl_Tumor_Case3a.eps}
	\end{array}
	$
	\caption{Scenario III: Sum of the sensitive $S$ and resistant $R$ tumor subpopulations over time $t$ for the constant treatment with different values of $p$ associated with CD (\textit{left}) and HFD (\textit{right}). We mark the time point at which treatment is started with a dashed vertical line. ~\eqref{Model2}.}\label{Fig_E3}
\end{figure}
\end{comment}
%If we look at Fig.~\ref{Fig_E4} in detail, additionally, we see that treatment decreases the estrogen level, so sensitive cells are eliminated. However, resistant cells increase due to the dominant conversion term. Thus, we can deduce that the elbow in Fig.~\ref{Fig_E3} is caused by resistance to treatment. Now, we investigate the contribution of optimal treatment over the constant treatment.
\begin{figure}[h!]
	\centering
	$
	\begin{array}{c}
	\includegraphics[height=0.53\textwidth]{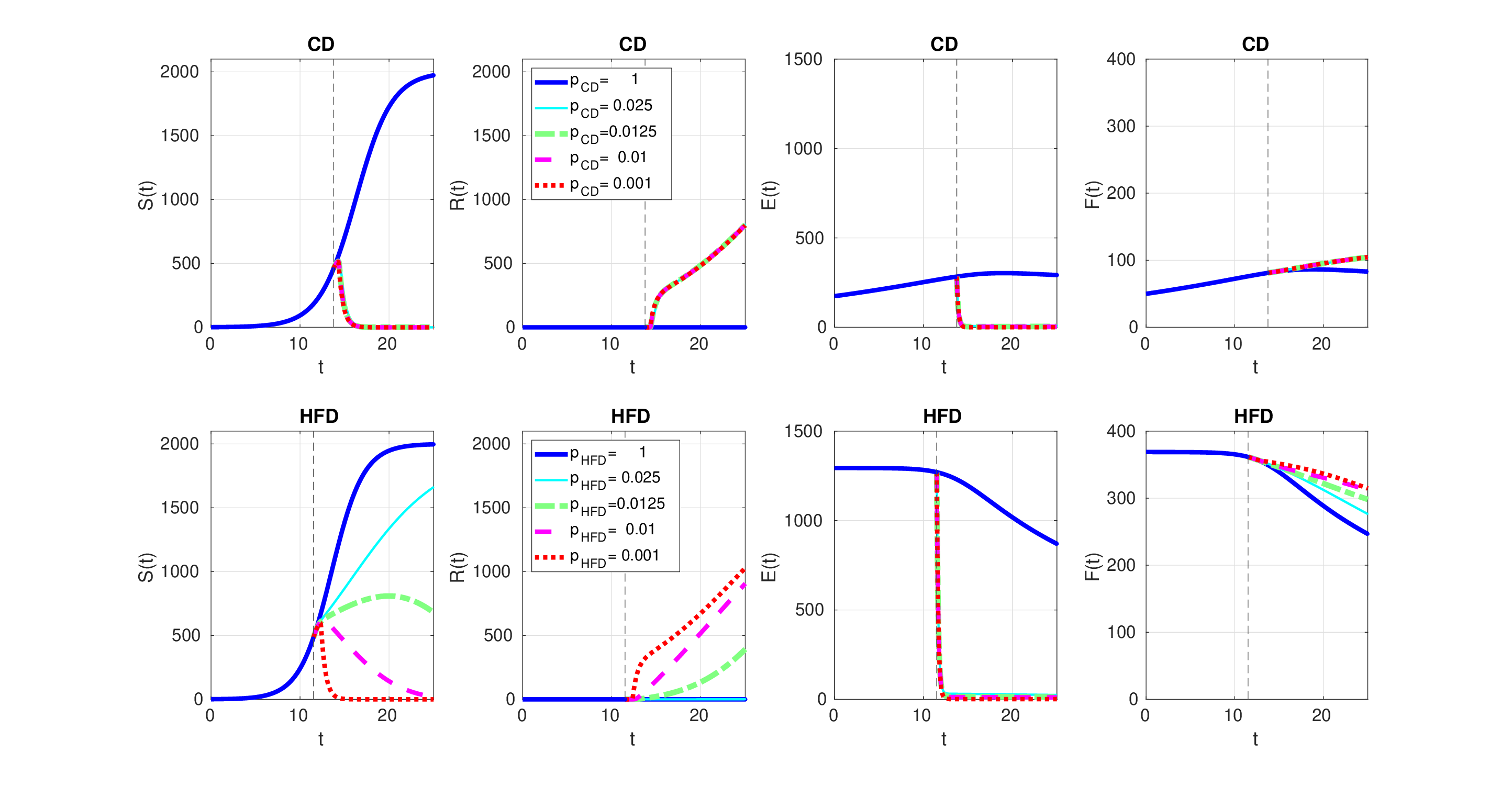}
	\end{array}
	$
	\caption{Scenario III: Dynamics of model variables $S$, $R$, $E$ and $F$ over time $t$ associated with CD (\textit{1st row}) and HFD (\textit{2nd row}). }\label{Fig_E4}
\end{figure}
On the other hand, alternating treatment is not a successful strategy (see Fig.~\ref{Fig_I3}). 
\vspace{-1.5cm}
\begin{figure}[h!]
	\centering
	$
	\begin{array}{cc}
	\includegraphics[height=0.28\textwidth]{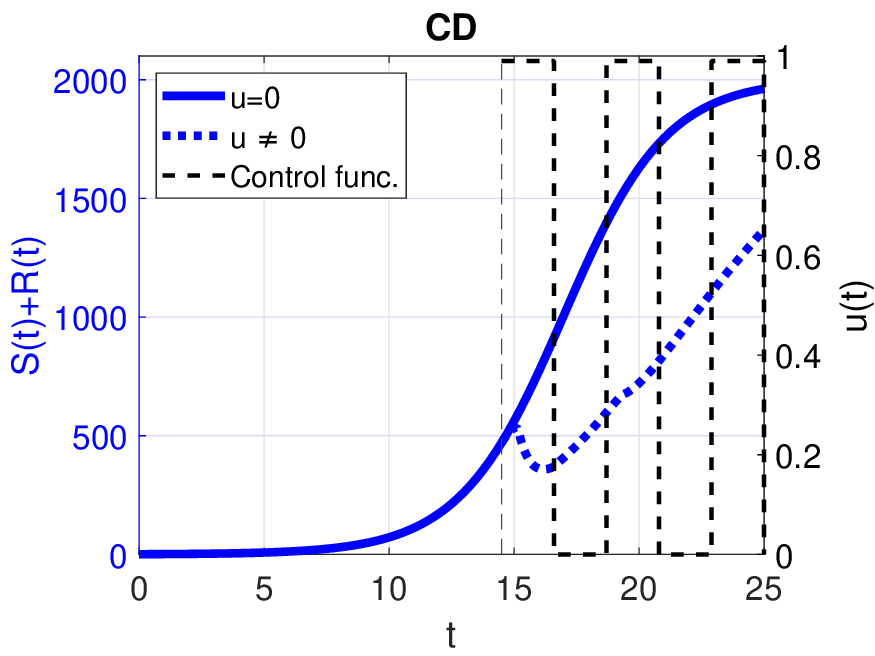}&
	\includegraphics[height=0.28\textwidth]{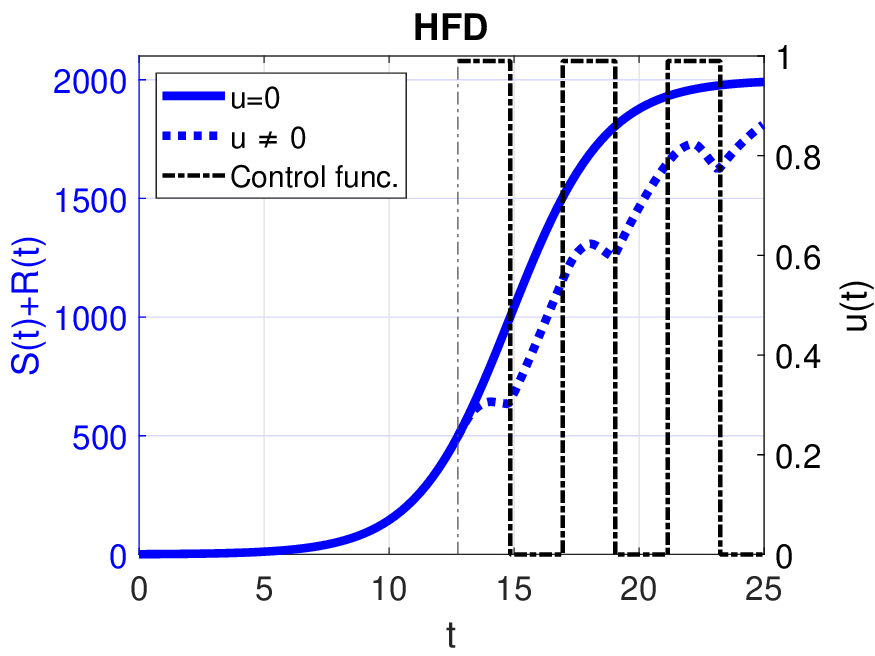}
	\end{array}
	$
	\caption{Scenario III: Left axis refers to the sum of the sensitive $S$ and resistant $R$ tumor subpopulations over time $t$ for alternating treatment associated with CD (\textit{left}) and HFD (\textit{right}); right axis refers to treatment schedule. The solid line refers to the tumor size without treatment. We mark the time point at which treatment is started with a dashed vertical line.}\label{Fig_I3}
\end{figure}
%

\begin{comment}
\begin{figure}[h!]
	\centering
	$
	\begin{array}{c}
	\includegraphics[height=0.4\textwidth]{Figures/ocp/J_Case8.eps}
	\end{array}
	$
	\caption{Scenario III: Control function $u(t)$.}
\end{figure}
\end{comment}

%Tumor size for optimal treatment is quite similar to the one obtained with the constant treatment. Sensitive cells are converted into the resistant cells and treatment does not eliminate the tumor. 
Finally, optimal AI treatment results in drug resistance as seen in Fig.~\ref{Fig_T3}. Initial tumor size reduction is followed by cell proliferation. Even though treatment is stopped earlier for CD than HFD (see Fig.~\ref{Fig_C1b}), it is not possible to eliminate resistance due to equal cell death and conversion terms in the model.
\begin{figure}[h!]
	\centering
	$
	\begin{array}{cc}
	\includegraphics[height=0.28\textwidth]{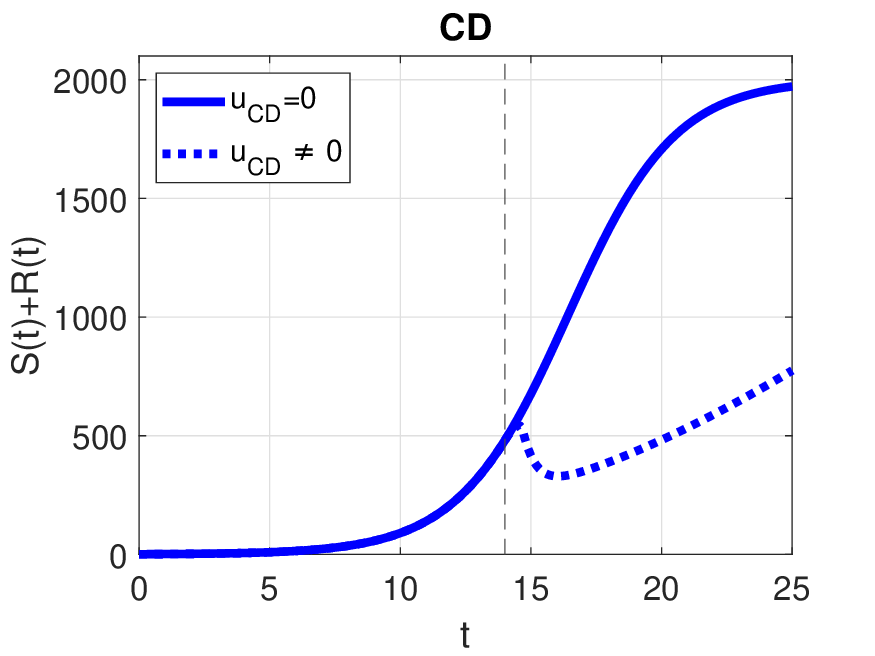}&
	\includegraphics[height=0.28\textwidth]{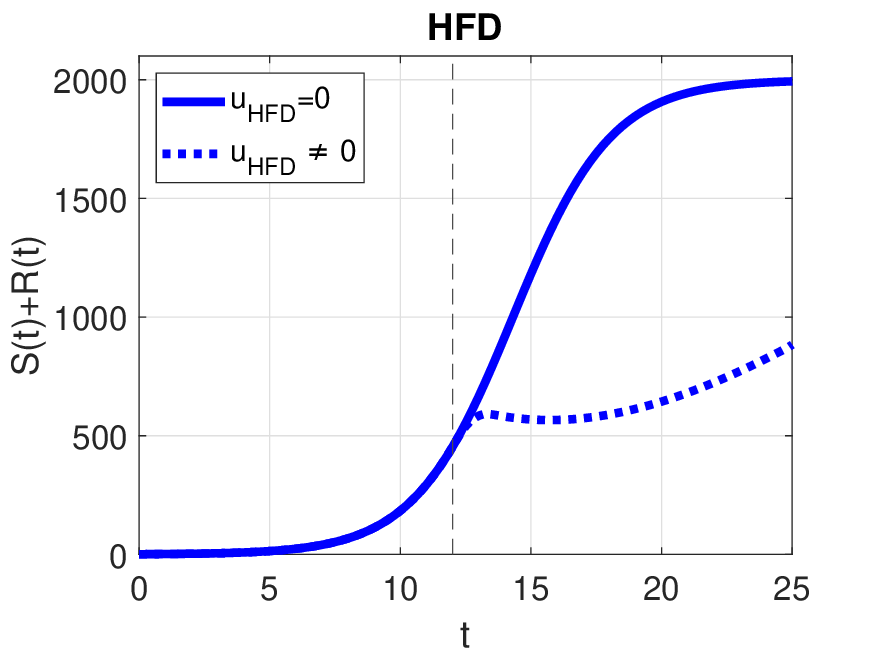}
	\end{array}
	$
	\caption{Scenario III: Sum of the sensitive $S$ and resistant $R$ tumor subpopulations over time $t$ for the optimal treatment with different values of $p$ associated with CD (\textit{left}) and HFD (\textit{right}). The solid curve refers to the tumor size without treatment. We mark the time point at which treatment is started with a dashed vertical line.}\label{Fig_T3}
\end{figure}

%When we have a look at the optimal control in Fig.~\ref{Fig_C8b}, we observe that treatment could be stopped earlier compared to the constant treatment. Treatment for HFD must be stopped earlier than CD; but, duration of the treatment is almost the same, since treatment for HFD starts earlier due to detectable tumor size. Effect of the weight coefficients can be compares as well in Fig.~\ref{Fig_C8}. More penalization of tumor cells results in longer treatment, whereas it can be shortened if treatment is penalized more.
\begin{figure}[h!]
	\centering
	$
	\begin{array}{c}
	\includegraphics[height=0.28\textwidth]{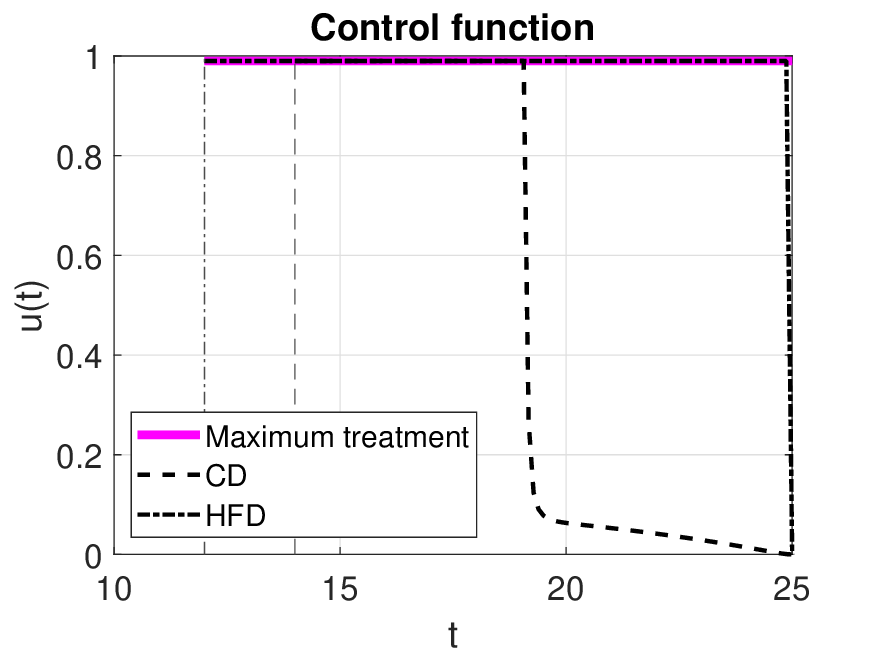}
	\end{array}
	$
	\caption{Scenario III: Optimal control function $u$ over time $t$ with $\omega_R=\omega_S=\omega_U=1$. Dashed and dash-dotted curves refer to the optimal treatment schedules for CD and HFD, respectively. Solid line denotes the maximum treatment. We mark the time point at which treatment is started with dashed and dash-dotted vertical lines for CD and HFD, respectively.}\label{Fig_C1b}
\end{figure}

A detailed picture of model variables is presented in Fig.~\ref{Fig_E3} and it reveals that treatment kills sensitive cells due to low estrogen level; but, then resistance occurs. 
\begin{figure}[h!]
	\centering
	$
	\begin{array}{c}
	\includegraphics[height=0.53\textwidth]{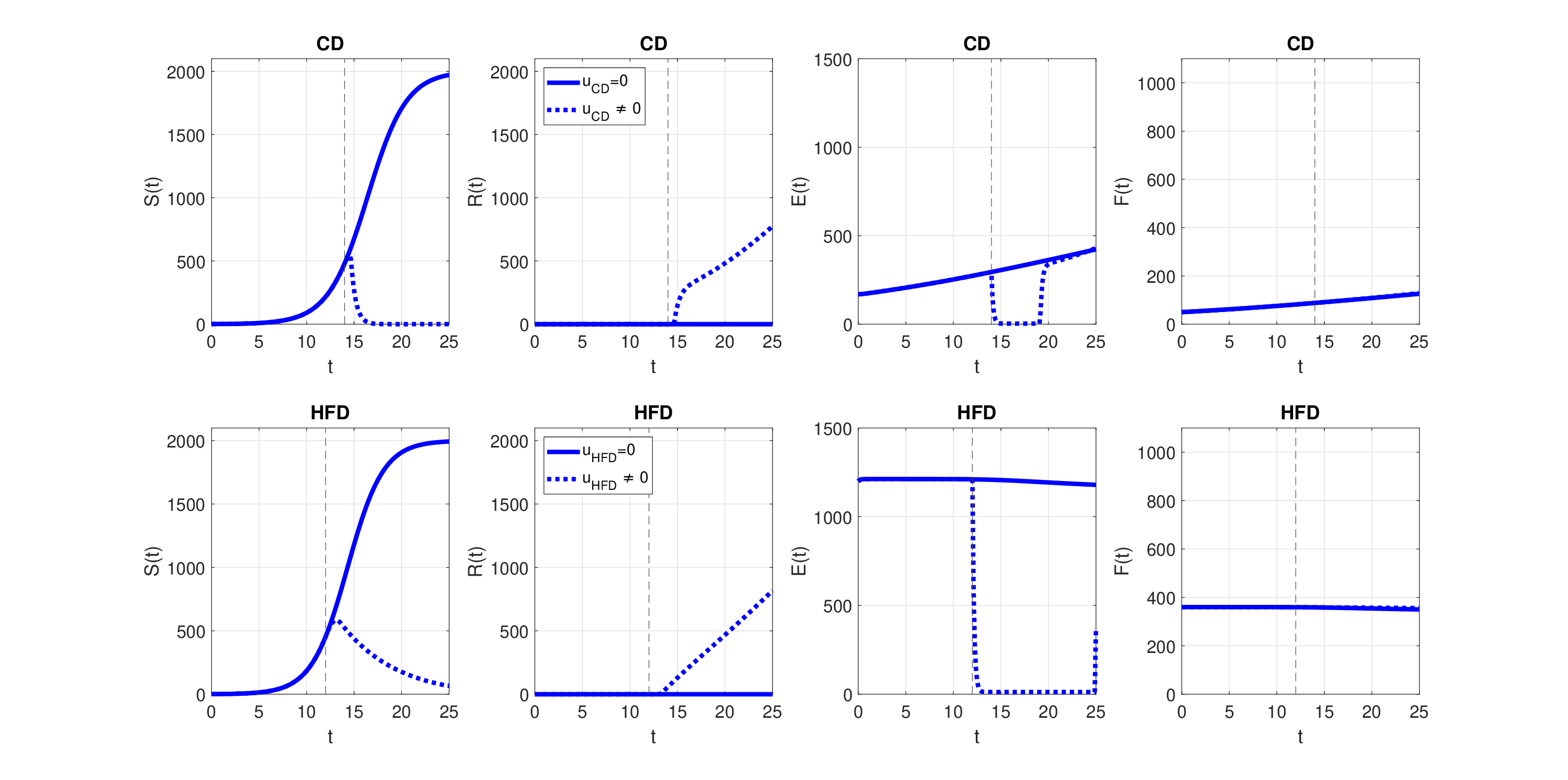}
	\end{array}
	$
	\caption{Scenario III: Dynamics of model variables $S$, $R$, $E$ and $F$ over time $t$ associated with CD (\textit{1st row}) and HFD (\textit{2nd row}). The solid curve refers to the tumor size without treatment, dotted curve corresponds to results for optimal treatment.}\label{Fig_E3}
\end{figure}

\begin{comment}
\begin{figure}[h!]
	\centering
	$
	\begin{array}{ccc}
		\includegraphics[width=0.4\textwidth]{Figures/ocp/Control_Case8_v2.eps}&
	\includegraphics[width=0.4\textwidth]{Figures/ocp/Control_Case8.eps}&
	\includegraphics[width=0.4\textwidth]{Figures/ocp/Control_Case8_v4.eps}\\
	 {\scriptsize \textrm{(\textit{a})~$\omega_R=\omega_S=10, \omega_U=1$.}}& 	{\scriptsize \textrm{(\textit{b})~$\omega_R=\omega_S=\omega_U=1$.}} &{\scriptsize \textrm{(\textit{c})~$\omega_R=\omega_S=1, \omega_U=10$.}}		
	\end{array}
	$
	\caption{Scenario III: Control function $u$.}
\end{figure}
\end{comment}

\subsection{Conclusions of the simulation results}
We compared outcomes for different treatments in a series of hypothetical tumors with differential sensitivities and rates of resistance to the local estrogen availability. We observed that in tumors where the difference between estrogen thresholds for cancer cells to die and to adapt to low estrogen levels is large, then constant treatment with an appropriate dose or optimal treatment are the best for the case of only adaptive resistance. However, if the difference between the thresholds is smaller, then optimal treatments are better, specially in the HFD case. In case of preexisting resistance, thresholds for cancer cells to die and to adapt to low estrogen levels is large, optimal treatment or constant treatment with appropriate dose gives the best outcome. When death of cancer cells and their adaptation to level of estrogen occur at the same threshold value, optimal treatment is best choice. Importantly, treatment outcome and optimal treatments schedules differ based on diet.

\section{Discussion}\label{conc}

%Modeling
Given the rising obesity rates around the world, novel strategies are urgently needed to evaluate and optimise endocrine treatment of breast cancer in women with high BMI. In this study we focused on modeling the effect that fat-induced production of estrogen has on tumor growth. While our model is able to capture the trends in the experimental data for CD and HFD mice, we recognise that other factors associated with the adipose tissue and not considered in our current model, such as inflammatory cytokines, leptin or insulin, could be influencing tumour growth differently in the CD and HFD cases. These are subjects that deserve further investigation \cite{hillers2022breast}.

%Resistance
By incorporating AI treatment and resistance in our model, we can simulate treatment outcomes in CD and HFD mice. However, as we do not have data on treatment, the choice of parameters related to sensitivity and resistant to treatment were made by explorative simulations. For instance, we assumed cost of resistance in the sense that the growth rate of resistance cells is smaller than the growth rate of sensitive cells. Otherwise, rapidly increasing resistant cells would always dominate the tumor. In addition to this, more than two tumor subpopulations with differential drug-response to AI could exists.  When AI treatment data in these mice are available, it would be possible to obtain the number of subpopulations, their fractions and their growth rates trough a novel phenotypic deconvolution method \cite{Kohn-Luque2022}.

% Uniqueness of OCP and % Implementation
Besides constant and alternating treatments, we investigated optimal scheduling trough OCPs. In this framework, we underline that one of the theoretical challenges is to prove uniqueness of the optimal control on a specific time interval $[0, t_f]$, since the value of $t_f$ cannot be found explicitly, and it is bounded by some constants depending on the solutions of the state and adjoint equation. We observe that a larger time interval leads to convergence issues, which is an indication of the uniqueness of the solution on a smaller time interval. We have also experienced that the more complicated the ODE model used in the OCP constraint is, the smaller the time interval where a unique solution can be found. Furthermore, uniqueness could be proved using constant tumor growth rate, but we believe this is not a correct representation of ER-positive tumor subtype. 

%Literature- OKE's study
Being a breast cancer modeling study with optimal control analysis, Oke et al. constructed a model of four variables (including normal cells, tumor cells, natural killer cells and estrogen concentration) with implementation of anti-cancer drugs and a ketogenic diet \cite{oke2018optimal}. They modelled the ketogenic diet as a parameter affecting tumor growth, while anti-cancer drug was modeled as an intervention strategy leading to tumor death, and estrogen concentration to decrease, so that suppression of immune cell activation was relaxed. In addition, optimal values of the parameters corresponding to anti-cancer drugs and ketogenic diet were searched to minimize the total tumor size and estrogen concentration on a prespecified time interval within a quadratic optimal control setting. The authors noted that activities of cancer cells are reduced with the introduction of a ketogenic diet and they underlined the risk of ketoacidosis as a results of too much ketogenic diet. The authors found  based on stability analysis of tumor-free equilibrium point that tumor cells could be eliminated with treatment and ketogenic diet, if the reproduction number of the system was reduced to a value less than one. This is in contrast with our simulations, where HFD does not result in better treatment outcomes. Interestingly, it has been shown that different fat diets, i.e. based on olive vs corn oil, influence breast tumor growth and progression differently \cite{costa2004high, solanas2009differential}, adding complexity to the challenge of optimizing breast cancer treatment and diet.

Overall, the most striking observations from our  simulations are that optimal aromatase inhibitor treatment schedules and the corresponding outcomes differ based on diet, which suggests that low fat diet and other measures to reduce the amount of fat could be introduced to improve treatment outcomes in obese patients. In our ongoing studies, we are modeling such patient-specific treatments making use of individual level data from the NeoLetExe trial \cite{bahrami2019neoletexe}, a neoadjuvant study exploring the lack of cross-resistance between the aromatase inhibitor letrozole and the aromatase inactivator exemestane. The effect of switching to a low-fat diet is not necessarily immediate, because it also depends on the lifestyle and how the body is prone to accumulate fat. Although our extended model might be able to capture lifestyle effects different than diet, this remains to be investigated.

\clearpage

\backmatter

\textbf{Data and code availability}\label{Data_code}
Simulations in this study were performed using MATLAB\textsuperscript{\textregistered}~R2022 \cite{MATLAB:2022}. All data and code used in this article are publicly available in the online repository of the Oslo Center for Biostatistics and Epidemiology \href{(OCBE)}{https://github.com/ocbe-uio/optimal_BC_treatment.git}. 

\clearpage

\textbf{Acknowledgments}
Authors thank Toni Hurtado for fruitful discussion. Tu\u{g}ba Akman was supported by TUBITAK (The Scientific and Technological Research Council of Turkey) under the 2219 Program. This project received funding from the RESCUER (RESistance Under Combinatorial Treatment in ER+ and ER- Breast Cancer) Project - European Union’s Horizon 2020 Research and Innovation Programme under Grant Agreement No. 847912, ERA-NET: Resistance under treatment in breast cancer (RESCUER) Research Council of Norway project code 311188 and from BigInsight with Norges Forskningsråd project number 237718. 
\clearpage

\section*{Declarations}
\textbf{Conflict of interest}
The authors declare that they have no known competing financial interests or personal relationships that could have appeared to influence the work reported in this paper.

\clearpage

\begin{comment}
\section*{Declarations}
Conflict of interest The authors declare that they have no known competing financial interests or personal
relationships that could have appeared to influence the work reported in this paper.

%Some journals require declarations to be submitted in a standardised format. Please check the Instructions for Authors of the journal to which you are submitting to see if you need to complete this section. If yes, your manuscript must contain the following sections under the heading `Declarations':

\begin{itemize}
\item Funding 
\item Conflict of interest/Competing interests (check journal-specific guidelines for which heading to use)
\item Ethics approval 
\item Consent to participate
\item Consent for publication
\item Availability of data and materials
\item Code availability 
\item Authors' contributions
\end{itemize}

\noindent
If any of the sections are not relevant to your manuscript, please include the heading and write `Not applicable' for that section.

\clearpage
\end{comment}

%\begin{appendices}
%\input{supplement}
%\end{appendices}
\section{Supplementary Material}\label{supp}

\textbf{Computation of fat volume at $t=15$ and carrying capacity of adipocytes:}
Growth of adipose tissue is based on two processes affected by genetic and diet differences, namely cell number increase (hyperplasia) and cell size increase (hypertrophy) \cite{jo2009hypertrophy}. Therefore, computation of the amount of fat at days $t=0$ and $t=15$ in the system are based on these two mechanisms. We start by calculating the average number of adipocytes per image and per $mm^2$ for CD and HFD, separately, based on the study~\cite[Fig.~S2F]{hillers2018obesity}. 

On the other hand, Bozec and Hannemann report that adipocyte size increases in mice fed with HFD, so we set the diameters of adipocytes for CD and HFD as 0.1 mm and 0.2 mm, respectively, as reported in the study \cite[Fig.6C]{bozec2016mechanism}, which leads to two different values for the amount of fat at time $t=15$ for CD and HFD. In addition, mice are fed with control diet after tumor injection in the experiment, so we assume that carrying capacity of adipocytes associated with mice fed with CD and HFD are determined by the amount of fat for HFD at $t=15$. 
%Adipocyte size may range from $<20 \mu m$ to 300 $\mu m$ in diameter \cite{stenkula2018adipose}. 

Procedure to estimate the amount of fat in the cube with volume $V$ could be summarized as follows: Let $n$ be the number of adipocytes per $mm^2$ and $d$ be the average diameter ($mm$) of an adipocyte. Then, $\sqrt[3]{V}$ is the size of the cube in $mm$ with volume $V$ and $ \frac{\sqrt[3]{V}}{d}$ is the number of layers with height $d$ in the cube, where $d$ is the diameter of an adipocyte. In addition, we compute $\sqrt[3]{V} \times n$ adipocytes per $mm^2$. Finally, the approximate number of adipocytes in the cube with volume $V$ could be aproximated via the product $ \frac{\sqrt[3]{V}}{d} \times \sqrt[3]{V} \times n = \frac{n \sqrt[3]{V^2}}{d} $.

\clearpage

\textbf{Profile likelihood calculations} \label{prof_like}
We present profile likelihood calculations for model~\eqref{Model1} after performing calibration \cite{kreutz2012likelihood} in Fig.~\ref{Fig_S22}. We used "arPLEInit" and "ple" functions of d2d software. The figure indicates practical identifiability of the model corresponding to the 95\% confidence level for the parameters $a_1$ and $k_1$. For parameter $\alpha$, this threshold is above 83\%. 
\begin{figure}[h!]
	\centering
	$
	\begin{array}{c}
	\includegraphics[height=0.75\textwidth]{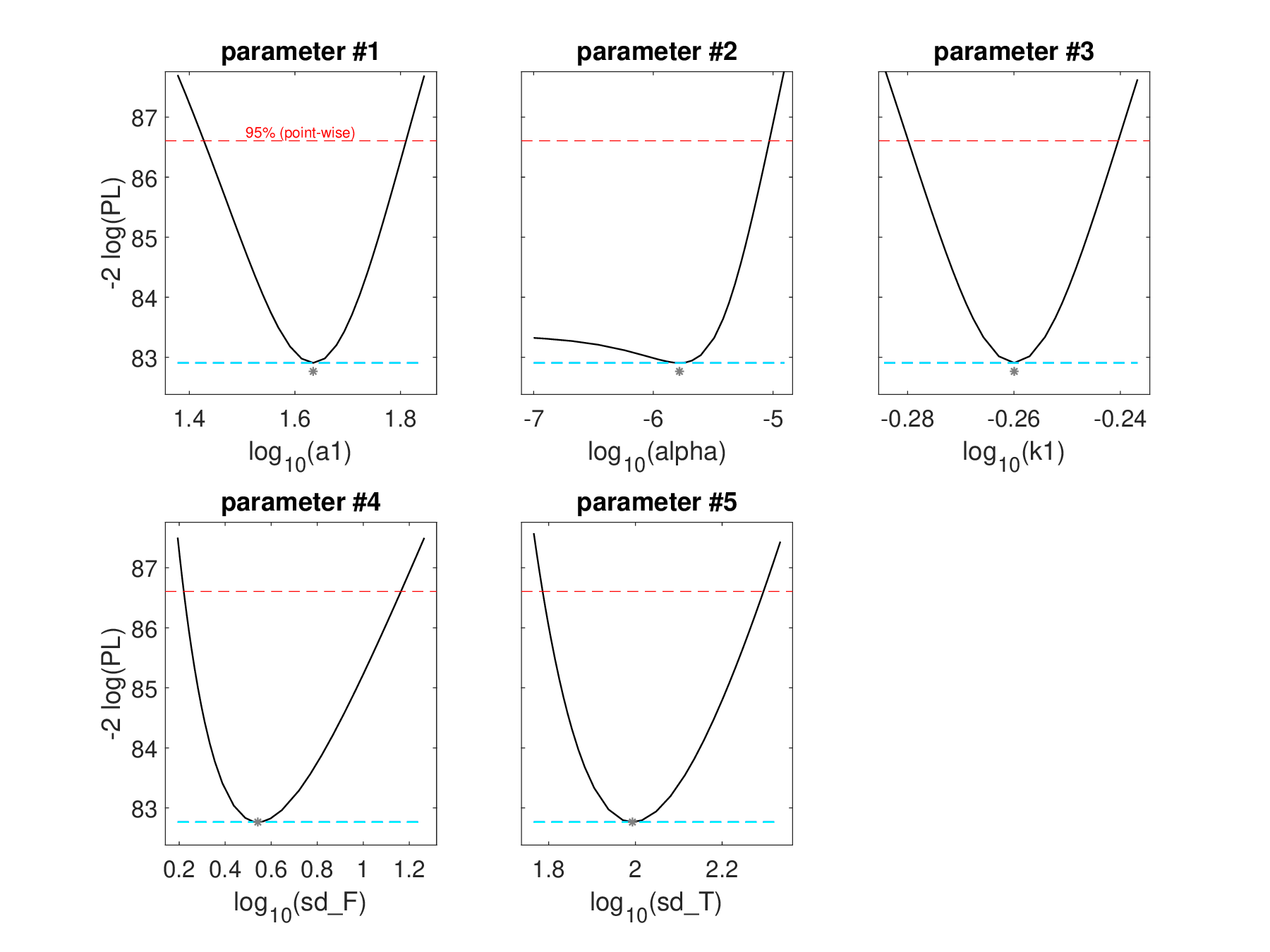}
	\end{array}
	$
	\caption{Profile likelihood computations for the calibrated parameters $a_1, \alpha, k_1$.}\label{Fig_S22}
\end{figure}

\clearpage
\textbf{Sensitivity analysis}
%\textcolor{red}{Referee: why didn't you perform sensitivity analysis for Model I?} 

We present the results of the global sensitivity analysis for Eq.~\ref{Model1} to capture the relative changes of the variables with respect to the parameters at days 5, 15 and 25. Here, the method of partial rank correlation coefficient (PRCC) presented by Marino et al. is used \cite{marino2008methodology}. 

Global sensitivity analysis requires a uniformly distributed sample space which is constructed for each parameter with 1000 sample values. Here, intervals for each parameter are constructed with the end points as $\pm$ 5 $\%$ of the baseline parameter. LHS-PRCC MATLAB code given in the web site \cite{PRCCcode} have been used and modified for the current model. The idea behind PRCC is to assign a value between -1 and +1 to each relation. The magnitude of this value determines the strength and the sign indicates the trend of the relation between the parameter and the variable. We briefly compare the most sensitive parameters where sensitivities for CD and HFD are grouped on the left and right panel in Fig.~\ref{Fig_PRCC1}. Firstly, we observe that parameters $k_1$ and $r$ have positive sensitivities for tumor volume, whereas negative sensitivity with respect to $\mu$ decreases over time in magnitude. Sensitivity of estrogen concentration is positive for $r$ and  negative for $\mu$, as expected. In addition, fat volume is negatively sensitive to $k_1$ and positive sensitivity with respect to $m_1$ increases over time and it is effect is smaller for HFD than CD case.

%Firstly, we observe that the parameters $k_1, p$ and $r$ have positive sensitivities for both sensitive tumor subpopulation, whereas fat growth rate $k_2$ has positive and negative effect on sensitive and resistant subpopulations, respectively. It means that sensitive cell population is prone to increase in size as fat cell proliferation increases, but opposite behaviour is observed for resistant cells. On the other hand, parameters $a_1$ and $\mu$ have negative sensitivities and their effect lessen over time. Resistant cells are positively affected by the positive changes in $a_3$, $\mu$, negatively in $l, p, r$.  Estrogen concentration is affected positively by the parameters $p$ and $r$, where $\mu$ has a negative impact on it. Fat volume is positively sensitive to $k_2$ and carrying capacity.

%\subsection{Sensitivity analysis}
\begin{figure}[h!]
\subfloat[][]{\includegraphics[scale=0.4]{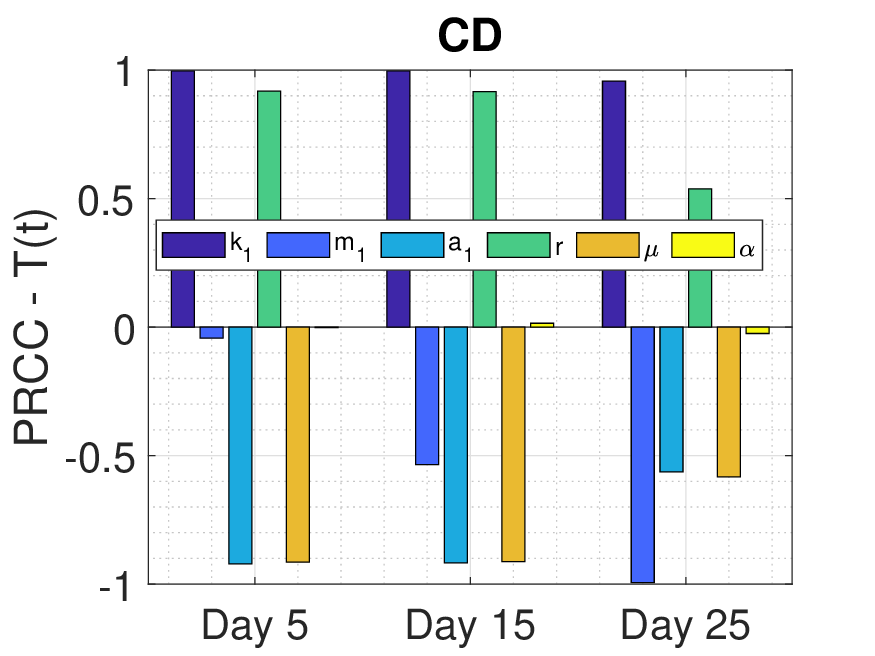}\label{Fig_PRCC1a}}
\subfloat[][]{\includegraphics[scale=0.4]{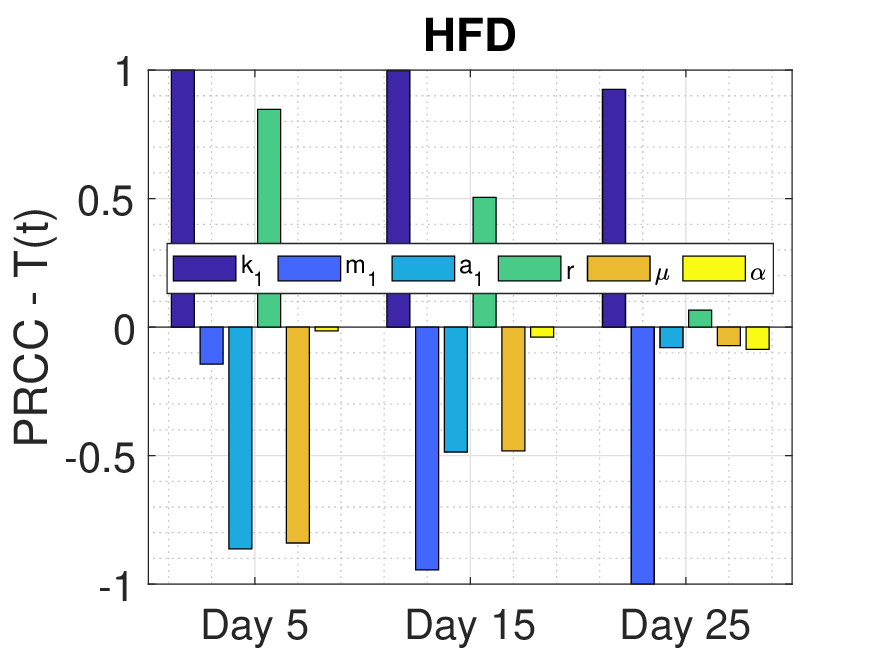}\label{Fig_PRCC1b}}\\
\subfloat[][]{\includegraphics[scale=0.4]{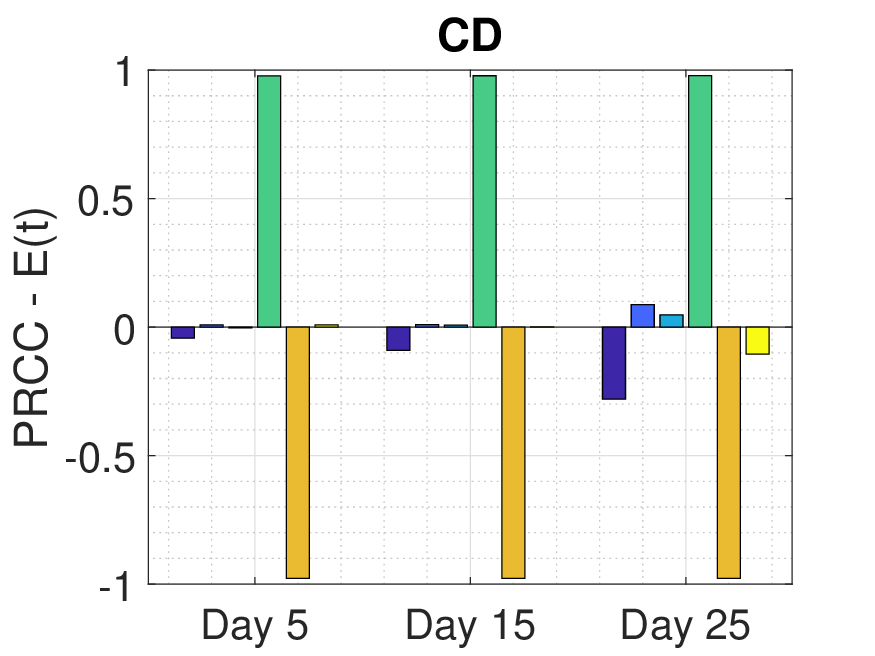}\label{Fig_PRCC2a}}
\subfloat[][]{\includegraphics[scale=0.4]{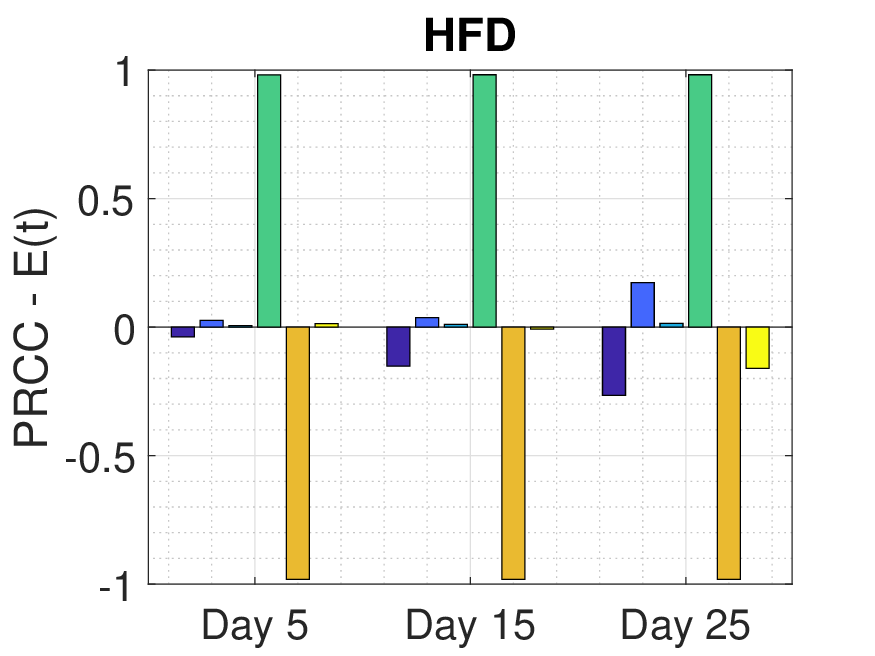}\label{Fig_PRCC2b}}\\
\subfloat[][]{\includegraphics[scale=0.4]{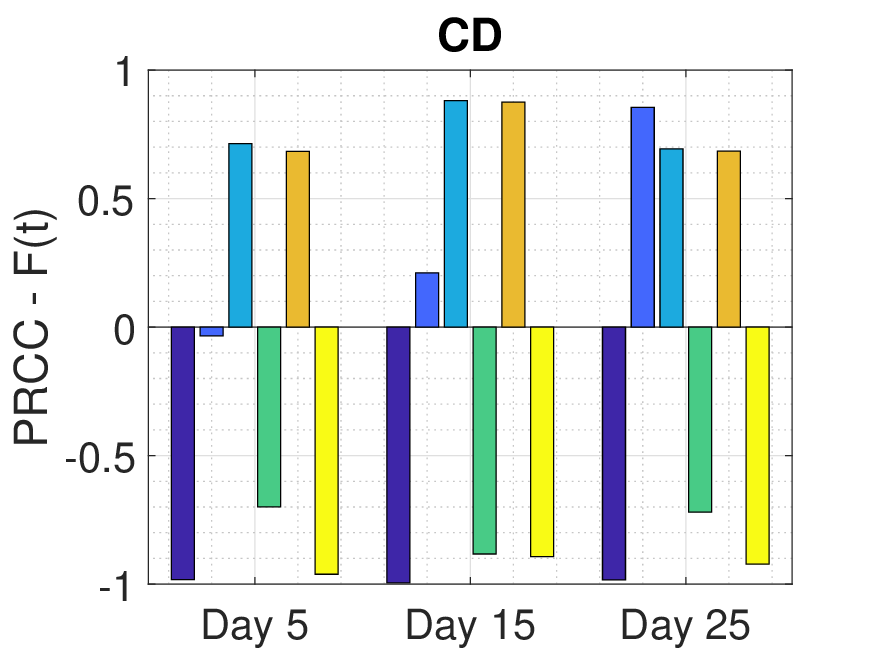}\label{Fig_PRCC3a}}
\subfloat[][]{\includegraphics[scale=0.4]{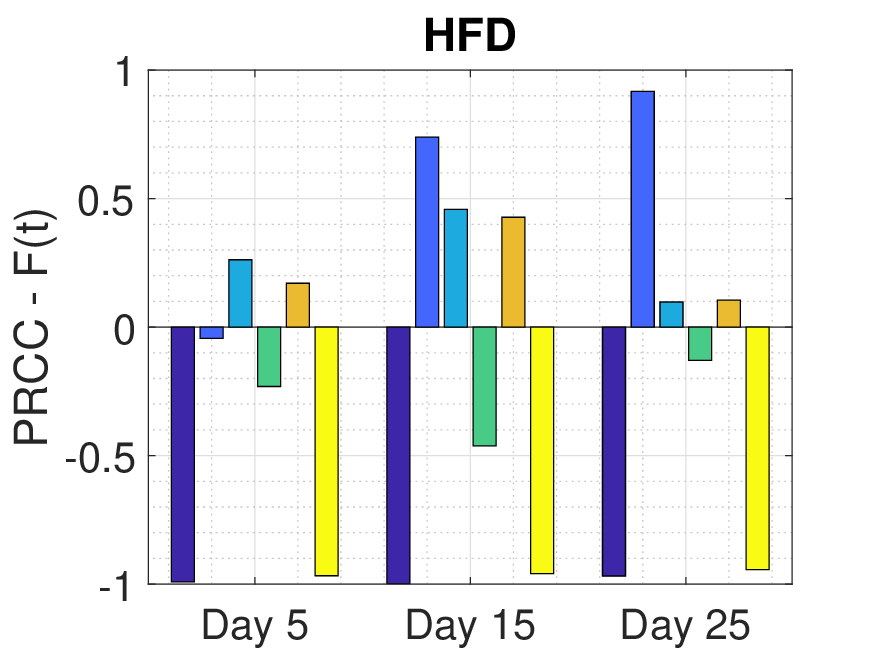}\label{Fig_PRCC3b}}\\
\caption{Sensitivity analysis for CD (\textit{left column}) and HFD (\textit{right column}) for Eq.~\ref{Model1} at days 5, 15 and 25 where \protect\subref{Fig_PRCC1a} - \protect\subref{Fig_PRCC1b} T(t),
\protect\subref{Fig_PRCC2a} - \protect\subref{Fig_PRCC2b} E(t), \protect\subref{Fig_PRCC3a} - \protect\subref{Fig_PRCC3b} F(t).}\label{Fig_PRCC1}
\end{figure}

\clearpage
%\bibliography{mybibfile}% common bib file

\begin{thebibliography}{93}
\ifx \bisbn   \undefined \def \bisbn  #1{ISBN #1}\fi
\ifx \binits  \undefined \def \binits#1{#1}\fi
\ifx \bauthor  \undefined \def \bauthor#1{#1}\fi
\ifx \batitle  \undefined \def \batitle#1{#1}\fi
\ifx \bjtitle  \undefined \def \bjtitle#1{#1}\fi
\ifx \bvolume  \undefined \def \bvolume#1{\textbf{#1}}\fi
\ifx \byear  \undefined \def \byear#1{#1}\fi
\ifx \bissue  \undefined \def \bissue#1{#1}\fi
\ifx \bfpage  \undefined \def \bfpage#1{#1}\fi
\ifx \blpage  \undefined \def \blpage #1{#1}\fi
\ifx \burl  \undefined \def \burl#1{\textsf{#1}}\fi
\ifx \doiurl  \undefined \def \doiurl#1{\url{https://doi.org/#1}}\fi
\ifx \betal  \undefined \def \betal{\textit{et al.}}\fi
\ifx \binstitute  \undefined \def \binstitute#1{#1}\fi
\ifx \binstitutionaled  \undefined \def \binstitutionaled#1{#1}\fi
\ifx \bctitle  \undefined \def \bctitle#1{#1}\fi
\ifx \beditor  \undefined \def \beditor#1{#1}\fi
\ifx \bpublisher  \undefined \def \bpublisher#1{#1}\fi
\ifx \bbtitle  \undefined \def \bbtitle#1{#1}\fi
\ifx \bedition  \undefined \def \bedition#1{#1}\fi
\ifx \bseriesno  \undefined \def \bseriesno#1{#1}\fi
\ifx \blocation  \undefined \def \blocation#1{#1}\fi
\ifx \bsertitle  \undefined \def \bsertitle#1{#1}\fi
\ifx \bsnm \undefined \def \bsnm#1{#1}\fi
\ifx \bsuffix \undefined \def \bsuffix#1{#1}\fi
\ifx \bparticle \undefined \def \bparticle#1{#1}\fi
\ifx \barticle \undefined \def \barticle#1{#1}\fi
\bibcommenthead
\ifx \bconfdate \undefined \def \bconfdate #1{#1}\fi
\ifx \botherref \undefined \def \botherref #1{#1}\fi
\ifx \url \undefined \def \url#1{\textsf{#1}}\fi
\ifx \bchapter \undefined \def \bchapter#1{#1}\fi
\ifx \bbook \undefined \def \bbook#1{#1}\fi
\ifx \bcomment \undefined \def \bcomment#1{#1}\fi
\ifx \oauthor \undefined \def \oauthor#1{#1}\fi
\ifx \citeauthoryear \undefined \def \citeauthoryear#1{#1}\fi
\ifx \endbibitem  \undefined \def \endbibitem {}\fi
\ifx \bconflocation  \undefined \def \bconflocation#1{#1}\fi
\ifx \arxivurl  \undefined \def \arxivurl#1{\textsf{#1}}\fi
\csname PreBibitemsHook\endcsname

%%% 1
\bibitem{wu2016substantial}
\begin{barticle}
\bauthor{\bsnm{Wu}, \binits{S.}},
\bauthor{\bsnm{Powers}, \binits{S.}},
\bauthor{\bsnm{Zhu}, \binits{W.}},
\bauthor{\bsnm{Hannun}, \binits{Y.A.}}:
\batitle{Substantial contribution of extrinsic risk factors to cancer
  development}.
\bjtitle{Nature}
\bvolume{529}(\bissue{7584}),
\bfpage{43}--\blpage{47}
(\byear{2016})
\end{barticle}
\endbibitem

%%% 2
\bibitem{neuhouser2015overweight}
\begin{barticle}
\bauthor{\bsnm{Neuhouser}, \binits{M.L.}},
\bauthor{\bsnm{Aragaki}, \binits{A.K.}},
\bauthor{\bsnm{Prentice}, \binits{R.L.}},
\bauthor{\bsnm{Manson}, \binits{J.E.}},
\bauthor{\bsnm{Chlebowski}, \binits{R.}},
\bauthor{\bsnm{Carty}, \binits{C.L.}},
\bauthor{\bsnm{Ochs-Balcom}, \binits{H.M.}},
\bauthor{\bsnm{Thomson}, \binits{C.A.}},
\bauthor{\bsnm{Caan}, \binits{B.J.}},
\bauthor{\bsnm{Tinker}, \binits{L.F.}}, \betal:
\batitle{Overweight, obesity, and postmenopausal invasive breast cancer risk: a
  secondary analysis of the women's health initiative randomized clinical
  trials}.
\bjtitle{JAMA oncology}
\bvolume{1}(\bissue{5}),
\bfpage{611}--\blpage{621}
(\byear{2015})
\end{barticle}
\endbibitem

%%% 3
\bibitem{clark1984correlations}
\begin{barticle}
\bauthor{\bsnm{Clark}, \binits{G.M.}},
\bauthor{\bsnm{Osborne}, \binits{C.K.}},
\bauthor{\bsnm{McGuire}, \binits{W.L.}}:
\batitle{Correlations between estrogen receptor, progesterone receptor, and
  patient characteristics in human breast cancer.}
\bjtitle{Journal of clinical oncology}
\bvolume{2}(\bissue{10}),
\bfpage{1102}--\blpage{1109}
(\byear{1984})
\end{barticle}
\endbibitem

%%% 4
\bibitem{johnston2003aromatase}
\begin{barticle}
\bauthor{\bsnm{Johnston}, \binits{S.R.}},
\bauthor{\bsnm{Dowsett}, \binits{M.}}:
\batitle{Aromatase inhibitors for breast cancer: lessons from the laboratory}.
\bjtitle{Nature Reviews Cancer}
\bvolume{3}(\bissue{11}),
\bfpage{821}--\blpage{831}
(\byear{2003})
\end{barticle}
\endbibitem

%%% 5
\bibitem{pearson1982antiestrogen}
\begin{barticle}
\bauthor{\bsnm{Pearson}, \binits{O.H.}},
\bauthor{\bsnm{Manni}, \binits{A.}},
\bauthor{\bsnm{Arafah}, \binits{B.M.}}:
\batitle{Antiestrogen treatment of breast cancer: an overview}.
\bjtitle{Cancer research}
\bvolume{42}(\bissue{8\_Supplement}),
\bfpage{3424}--\blpage{3428}
(\byear{1982})
\end{barticle}
\endbibitem

%%% 6
\bibitem{folkerd2012suppression}
\begin{barticle}
\bauthor{\bsnm{Folkerd}, \binits{E.J.}},
\bauthor{\bsnm{Dixon}, \binits{J.M.}},
\bauthor{\bsnm{Renshaw}, \binits{L.}},
\bauthor{\bsnm{A'Hern}, \binits{R.P.}},
\bauthor{\bsnm{Dowsett}, \binits{M.}}:
\batitle{Suppression of plasma estrogen levels by letrozole and anastrozole is
  related to body mass index in patients with breast cancer}.
\bjtitle{Journal of clinical oncology}
\bvolume{30}(\bissue{24}),
\bfpage{2977}--\blpage{2980}
(\byear{2012})
\end{barticle}
\endbibitem

%%% 7
\bibitem{ioannides2014effect}
\begin{barticle}
\bauthor{\bsnm{Ioannides}, \binits{S.}},
\bauthor{\bsnm{Barlow}, \binits{P.}},
\bauthor{\bsnm{Elwood}, \binits{J.}},
\bauthor{\bsnm{Porter}, \binits{D.}}:
\batitle{Effect of obesity on aromatase inhibitor efficacy in postmenopausal,
  hormone receptor-positive breast cancer: a systematic review}.
\bjtitle{Breast cancer research and treatment}
\bvolume{147}(\bissue{2}),
\bfpage{237}--\blpage{248}
(\byear{2014})
\end{barticle}
\endbibitem

%%% 8
\bibitem{jiralerspong2016obesity}
\begin{barticle}
\bauthor{\bsnm{Jiralerspong}, \binits{S.}},
\bauthor{\bsnm{Goodwin}, \binits{P.J.}}:
\batitle{Obesity and breast cancer prognosis: evidence, challenges, and
  opportunities}.
\bjtitle{Journal of Clinical Oncology}
\bvolume{34}(\bissue{35}),
\bfpage{4203}--\blpage{4216}
(\byear{2016})
\end{barticle}
\endbibitem

%%% 9
\bibitem{bahrami2021lack}
\begin{barticle}
\bauthor{\bsnm{Bahrami}, \binits{N.}},
\bauthor{\bsnm{Jabeen}, \binits{S.}},
\bauthor{\bsnm{Tahiri}, \binits{A.}},
\bauthor{\bsnm{Sauer}, \binits{T.}},
\bauthor{\bsnm{{\O}deg{\aa}rd}, \binits{H.P.}},
\bauthor{\bsnm{Geisler}, \binits{S.B.}},
\bauthor{\bsnm{Gravdehaug}, \binits{B.}},
\bauthor{\bsnm{Reitsma}, \binits{L.C.}},
\bauthor{\bsnm{Sels{\aa}s}, \binits{K.}},
\bauthor{\bsnm{Kristensen}, \binits{V.}}, \betal:
\batitle{Lack of cross-resistance between non-steroidal and steroidal aromatase
  inhibitors in breast cancer patients: the potential role of the adipokine
  leptin}.
\bjtitle{Breast Cancer Research and Treatment}
\bvolume{190}(\bissue{3}),
\bfpage{435}--\blpage{449}
(\byear{2021})
\end{barticle}
\endbibitem

%%% 10
\bibitem{wang2015aromatase}
\begin{barticle}
\bauthor{\bsnm{Wang}, \binits{X.}},
\bauthor{\bsnm{Simpson}, \binits{E.R.}},
\bauthor{\bsnm{Brown}, \binits{K.A.}}:
\batitle{Aromatase overexpression in dysfunctional adipose tissue links obesity
  to postmenopausal breast cancer}.
\bjtitle{The Journal of steroid biochemistry and molecular biology}
\bvolume{153},
\bfpage{35}--\blpage{44}
(\byear{2015})
\end{barticle}
\endbibitem

%%% 11
\bibitem{gelsomino2020leptin}
\begin{barticle}
\bauthor{\bsnm{Gelsomino}, \binits{L.}},
\bauthor{\bsnm{Giordano}, \binits{C.}},
\bauthor{\bsnm{La~Camera}, \binits{G.}},
\bauthor{\bsnm{Sisci}, \binits{D.}},
\bauthor{\bsnm{Marsico}, \binits{S.}},
\bauthor{\bsnm{Campana}, \binits{A.}},
\bauthor{\bsnm{Tarallo}, \binits{R.}},
\bauthor{\bsnm{Rinaldi}, \binits{A.}},
\bauthor{\bsnm{Fuqua}, \binits{S.}},
\bauthor{\bsnm{Leggio}, \binits{A.}}, \betal:
\batitle{Leptin signaling contributes to aromatase inhibitor resistant breast
  cancer cell growth and activation of macrophages}.
\bjtitle{Biomolecules}
\bvolume{10}(\bissue{4}),
\bfpage{543}
(\byear{2020})
\end{barticle}
\endbibitem

%%% 12
\bibitem{goodwin2010obesity}
\begin{botherref}
\oauthor{\bsnm{Goodwin}, \binits{P.J.}},
\oauthor{\bsnm{Pritchard}, \binits{K.I.}}:
Obesity and hormone therapy in breast cancer: an unfinished puzzle.
American Society of Clinical Oncology
(2010)
\end{botherref}
\endbibitem

%%% 13
\bibitem{lonning2014relationship}
\begin{barticle}
\bauthor{\bsnm{L{\o}nning}, \binits{P.E.}},
\bauthor{\bsnm{Haynes}, \binits{B.P.}},
\bauthor{\bsnm{Dowsett}, \binits{M.}}:
\batitle{Relationship of body mass index with aromatisation and plasma and
  tissue oestrogen levels in postmenopausal breast cancer patients treated with
  aromatase inhibitors}.
\bjtitle{European Journal of Cancer}
\bvolume{50}(\bissue{6}),
\bfpage{1055}--\blpage{1064}
(\byear{2014})
\end{barticle}
\endbibitem

%%% 14
\bibitem{sendur2012efficacy}
\begin{barticle}
\bauthor{\bsnm{Sendur}, \binits{M.}},
\bauthor{\bsnm{Aksoy}, \binits{S.}},
\bauthor{\bsnm{Zengin}, \binits{N.}},
\bauthor{\bsnm{Altundag}, \binits{K.}}:
\batitle{Efficacy of adjuvant aromatase inhibitor in hormone receptor-positive
  postmenopausal breast cancer patients according to the body mass index}.
\bjtitle{British journal of cancer}
\bvolume{107}(\bissue{11}),
\bfpage{1815}--\blpage{1819}
(\byear{2012})
\end{barticle}
\endbibitem

%%% 15
\bibitem{bordeleau2010multicenter}
\begin{barticle}
\bauthor{\bsnm{Bordeleau}, \binits{L.}},
\bauthor{\bsnm{Pritchard}, \binits{K.I.}},
\bauthor{\bsnm{Loprinzi}, \binits{C.L.}},
\bauthor{\bsnm{Ennis}, \binits{M.}},
\bauthor{\bsnm{Jugovic}, \binits{O.}},
\bauthor{\bsnm{Warr}, \binits{D.}},
\bauthor{\bsnm{Haq}, \binits{R.}},
\bauthor{\bsnm{Goodwin}, \binits{P.J.}}:
\batitle{Multicenter, randomized, cross-over clinical trial of venlafaxine
  versus gabapentin for the management of hot flashes in breast cancer
  survivors}.
\bjtitle{Journal of Clinical Oncology}
\bvolume{28}(\bissue{35}),
\bfpage{5147}--\blpage{5152}
(\byear{2010})
\end{barticle}
\endbibitem

%%% 16
\bibitem{ligibel2012risk}
\begin{barticle}
\bauthor{\bsnm{Ligibel}, \binits{J.A.}},
\bauthor{\bsnm{James~O’Malley}, \binits{A.}},
\bauthor{\bsnm{Fisher}, \binits{M.}},
\bauthor{\bsnm{Daniel}, \binits{G.W.}},
\bauthor{\bsnm{Winer}, \binits{E.P.}},
\bauthor{\bsnm{Keating}, \binits{N.L.}}:
\batitle{Risk of myocardial infarction, stroke, and fracture in a cohort of
  community-based breast cancer patients}.
\bjtitle{Breast cancer research and treatment}
\bvolume{131}(\bissue{2}),
\bfpage{589}--\blpage{597}
(\byear{2012})
\end{barticle}
\endbibitem

%%% 17
\bibitem{daldorff2017cotargeting}
\begin{barticle}
\bauthor{\bsnm{Daldorff}, \binits{S.}},
\bauthor{\bsnm{Mathiesen}, \binits{R.M.R.}},
\bauthor{\bsnm{Yri}, \binits{O.E.}},
\bauthor{\bsnm{{\O}deg{\aa}rd}, \binits{H.P.}},
\bauthor{\bsnm{Geisler}, \binits{J.}}:
\batitle{Cotargeting of {CYP}-19 (aromatase) and emerging, pivotal signalling
  pathways in metastatic breast cancer}.
\bjtitle{British Journal of Cancer}
\bvolume{116}(\bissue{1}),
\bfpage{10}--\blpage{20}
(\byear{2017})
\end{barticle}
\endbibitem

%%% 18
\bibitem{ma2015mechanisms}
\begin{barticle}
\bauthor{\bsnm{Ma}, \binits{C.X.}},
\bauthor{\bsnm{Reinert}, \binits{T.}},
\bauthor{\bsnm{Chmielewska}, \binits{I.}},
\bauthor{\bsnm{Ellis}, \binits{M.J.}}:
\batitle{Mechanisms of aromatase inhibitor resistance}.
\bjtitle{Nature Reviews Cancer}
\bvolume{15}(\bissue{5}),
\bfpage{261}--\blpage{275}
(\byear{2015})
\end{barticle}
\endbibitem

%%% 19
\bibitem{jeselsohn2015esr1}
\begin{barticle}
\bauthor{\bsnm{Jeselsohn}, \binits{R.}},
\bauthor{\bsnm{Buchwalter}, \binits{G.}},
\bauthor{\bsnm{De~Angelis}, \binits{C.}},
\bauthor{\bsnm{Brown}, \binits{M.}},
\bauthor{\bsnm{Schiff}, \binits{R.}}:
\batitle{{ESR1} mutations—a mechanism for acquired endocrine resistance in
  breast cancer}.
\bjtitle{Nature reviews Clinical oncology}
\bvolume{12}(\bissue{10}),
\bfpage{573}--\blpage{583}
(\byear{2015})
\end{barticle}
\endbibitem

%%% 20
\bibitem{norton1977tumor}
\begin{barticle}
\bauthor{\bsnm{Norton}, \binits{L.}},
\bauthor{\bsnm{Simon}, \binits{R.}}, \betal:
\batitle{Tumor size, sensitivity to therapy, and design of treatment
  schedules}.
\bjtitle{Cancer Treat Rep}
\bvolume{61}(\bissue{7}),
\bfpage{1307}--\blpage{1317}
(\byear{1977})
\end{barticle}
\endbibitem

%%% 21
\bibitem{enderling2006mathematical}
\begin{barticle}
\bauthor{\bsnm{Enderling}, \binits{H.}},
\bauthor{\bsnm{Anderson}, \binits{A.R.}},
\bauthor{\bsnm{Chaplain}, \binits{M.A.}},
\bauthor{\bsnm{Munro}, \binits{A.J.}},
\bauthor{\bsnm{Vaidya}, \binits{J.S.}}:
\batitle{Mathematical modelling of radiotherapy strategies for early breast
  cancer}.
\bjtitle{Journal of Theoretical Biology}
\bvolume{241}(\bissue{1}),
\bfpage{158}--\blpage{171}
(\byear{2006})
\end{barticle}
\endbibitem

%%% 22
\bibitem{enderling2007mathematical}
\begin{barticle}
\bauthor{\bsnm{Enderling}, \binits{H.}},
\bauthor{\bsnm{Chaplain}, \binits{M.A.}},
\bauthor{\bsnm{Anderson}, \binits{A.R.}},
\bauthor{\bsnm{Vaidya}, \binits{J.S.}}:
\batitle{A mathematical model of breast cancer development, local treatment and
  recurrence}.
\bjtitle{Journal of theoretical biology}
\bvolume{246}(\bissue{2}),
\bfpage{245}--\blpage{259}
(\byear{2007})
\end{barticle}
\endbibitem

%%% 23
\bibitem{frieboes2009prediction}
\begin{barticle}
\bauthor{\bsnm{Frieboes}, \binits{H.B.}},
\bauthor{\bsnm{Edgerton}, \binits{M.E.}},
\bauthor{\bsnm{Fruehauf}, \binits{J.P.}},
\bauthor{\bsnm{Rose}, \binits{F.R.}},
\bauthor{\bsnm{Worrall}, \binits{L.K.}},
\bauthor{\bsnm{Gatenby}, \binits{R.A.}},
\bauthor{\bsnm{Ferrari}, \binits{M.}},
\bauthor{\bsnm{Cristini}, \binits{V.}}:
\batitle{Prediction of drug response in breast cancer using integrative
  experimental/computational modeling}.
\bjtitle{Cancer research}
\bvolume{69}(\bissue{10}),
\bfpage{4484}--\blpage{4492}
(\byear{2009})
\end{barticle}
\endbibitem

%%% 24
\bibitem{roe2011mathematical}
\begin{barticle}
\bauthor{\bsnm{Roe-Dale}, \binits{R.}},
\bauthor{\bsnm{Isaacson}, \binits{D.}},
\bauthor{\bsnm{Kupferschmid}, \binits{M.}}:
\batitle{A mathematical model of breast cancer treatment with {CMF} and
  doxorubicin}.
\bjtitle{Bulletin of mathematical biology}
\bvolume{73}(\bissue{3}),
\bfpage{585}--\blpage{608}
(\byear{2011})
\end{barticle}
\endbibitem

%%% 25
\bibitem{yankeelov2013clinically}
\begin{barticle}
\bauthor{\bsnm{Yankeelov}, \binits{T.E.}},
\bauthor{\bsnm{Atuegwu}, \binits{N.}},
\bauthor{\bsnm{Hormuth}, \binits{D.}},
\bauthor{\bsnm{Weis}, \binits{J.A.}},
\bauthor{\bsnm{Barnes}, \binits{S.L.}},
\bauthor{\bsnm{Miga}, \binits{M.I.}},
\bauthor{\bsnm{Rericha}, \binits{E.C.}},
\bauthor{\bsnm{Quaranta}, \binits{V.}}:
\batitle{Clinically relevant modeling of tumor growth and treatment response}.
\bjtitle{Science translational medicine}
\bvolume{5}(\bissue{187}),
\bfpage{187}--\blpage{91879}
(\byear{2013})
\end{barticle}
\endbibitem

%%% 26
\bibitem{lai2018modeling}
\begin{barticle}
\bauthor{\bsnm{Lai}, \binits{X.}},
\bauthor{\bsnm{Stiff}, \binits{A.}},
\bauthor{\bsnm{Duggan}, \binits{M.}},
\bauthor{\bsnm{Wesolowski}, \binits{R.}},
\bauthor{\bsnm{Carson~III}, \binits{W.E.}},
\bauthor{\bsnm{Friedman}, \binits{A.}}:
\batitle{Modeling combination therapy for breast cancer with {BET} and immune
  checkpoint inhibitors}.
\bjtitle{Proceedings of the National Academy of Sciences}
\bvolume{115}(\bissue{21}),
\bfpage{5534}--\blpage{5539}
(\byear{2018})
\end{barticle}
\endbibitem

%%% 27
\bibitem{jarrett2019mathematical}
\begin{barticle}
\bauthor{\bsnm{Jarrett}, \binits{A.M.}},
\bauthor{\bsnm{Bloom}, \binits{M.J.}},
\bauthor{\bsnm{Godfrey}, \binits{W.}},
\bauthor{\bsnm{Syed}, \binits{A.K.}},
\bauthor{\bsnm{Ekrut}, \binits{D.A.}},
\bauthor{\bsnm{Ehrlich}, \binits{L.I.}},
\bauthor{\bsnm{Yankeelov}, \binits{T.E.}},
\bauthor{\bsnm{Sorace}, \binits{A.G.}}:
\batitle{Mathematical modelling of trastuzumab-induced immune response in an in
  vivo murine model of {HER}2+ breast cancer}.
\bjtitle{Mathematical medicine and biology: a journal of the IMA}
\bvolume{36}(\bissue{3}),
\bfpage{381}--\blpage{410}
(\byear{2019})
\end{barticle}
\endbibitem

%%% 28
\bibitem{lai2019toward}
\begin{barticle}
\bauthor{\bsnm{Lai}, \binits{X.}},
\bauthor{\bsnm{Geier}, \binits{O.M.}},
\bauthor{\bsnm{Fleischer}, \binits{T.}},
\bauthor{\bsnm{Garred}, \binits{{\O}.}},
\bauthor{\bsnm{Borgen}, \binits{E.}},
\bauthor{\bsnm{Funke}, \binits{S.W.}},
\bauthor{\bsnm{Kumar}, \binits{S.}},
\bauthor{\bsnm{Rognes}, \binits{M.E.}},
\bauthor{\bsnm{Seierstad}, \binits{T.}},
\bauthor{\bsnm{B{\o}rresen-Dale}, \binits{A.-L.}}, \betal:
\batitle{Toward personalized computer simulation of breast cancer treatment:
  {A} multiscale pharmacokinetic and pharmacodynamic model informed by
  multitype patient data}.
\bjtitle{Cancer research}
\bvolume{79}(\bissue{16}),
\bfpage{4293}--\blpage{4304}
(\byear{2019})
\end{barticle}
\endbibitem

%%% 29
\bibitem{lai2022scalable}
\begin{barticle}
\bauthor{\bsnm{Lai}, \binits{X.}},
\bauthor{\bsnm{Task{\'e}n}, \binits{H.A.}},
\bauthor{\bsnm{Mo}, \binits{T.}},
\bauthor{\bsnm{Funke}, \binits{S.W.}},
\bauthor{\bsnm{Frigessi}, \binits{A.}},
\bauthor{\bsnm{Rognes}, \binits{M.E.}},
\bauthor{\bsnm{K{\"o}hn-Luque}, \binits{A.}}:
\batitle{A scalable solver for a stochastic, hybrid cellular automaton model of
  personalized breast cancer therapy}.
\bjtitle{International Journal for Numerical Methods in Biomedical Engineering}
\bvolume{38}(\bissue{1}),
\bfpage{3542}
(\byear{2022})
\end{barticle}
\endbibitem

%%% 30
\bibitem{he2020mathematical}
\begin{barticle}
\bauthor{\bsnm{He}, \binits{W.}},
\bauthor{\bsnm{Demas}, \binits{D.M.}},
\bauthor{\bsnm{Conde}, \binits{I.P.}},
\bauthor{\bsnm{Shajahan-Haq}, \binits{A.N.}},
\bauthor{\bsnm{Baumann}, \binits{W.T.}}:
\batitle{Mathematical modelling of breast cancer cells in response to endocrine
  therapy and {Cdk4/6} inhibition}.
\bjtitle{Journal of the Royal Society Interface}
\bvolume{17}(\bissue{169}),
\bfpage{20200339}
(\byear{2020})
\end{barticle}
\endbibitem

%%% 31
\bibitem{chen2013modeling}
\begin{barticle}
\bauthor{\bsnm{Chen}, \binits{C.}},
\bauthor{\bsnm{Baumann}, \binits{W.T.}},
\bauthor{\bsnm{Clarke}, \binits{R.}},
\bauthor{\bsnm{Tyson}, \binits{J.J.}}:
\batitle{Modeling the estrogen receptor to growth factor receptor signaling
  switch in human breast cancer cells}.
\bjtitle{FEBS letters}
\bvolume{587}(\bissue{20}),
\bfpage{3327}--\blpage{3334}
(\byear{2013})
\end{barticle}
\endbibitem

%%% 32
\bibitem{chen2014mathematical}
\begin{barticle}
\bauthor{\bsnm{Chen}, \binits{C.}},
\bauthor{\bsnm{Baumann}, \binits{W.T.}},
\bauthor{\bsnm{Xing}, \binits{J.}},
\bauthor{\bsnm{Xu}, \binits{L.}},
\bauthor{\bsnm{Clarke}, \binits{R.}},
\bauthor{\bsnm{Tyson}, \binits{J.J.}}:
\batitle{Mathematical models of the transitions between endocrine therapy
  responsive and resistant states in breast cancer}.
\bjtitle{Journal of the Royal Society Interface}
\bvolume{11}(\bissue{96}),
\bfpage{20140206}
(\byear{2014})
\end{barticle}
\endbibitem

%%% 33
\bibitem{ouifki2022mathematical}
\begin{barticle}
\bauthor{\bsnm{Ouifki}, \binits{R.}},
\bauthor{\bsnm{Oke}, \binits{S.I.}}:
\batitle{Mathematical model for the estrogen paradox in breast cancer
  treatment}.
\bjtitle{Journal of Mathematical Biology}
\bvolume{84}(\bissue{4}),
\bfpage{1}--\blpage{32}
(\byear{2022})
\end{barticle}
\endbibitem

%%% 34
\bibitem{schattler2015optimal}
\begin{botherref}
\oauthor{\bsnm{Sch{\"a}ttler}, \binits{H.}},
\oauthor{\bsnm{Ledzewicz}, \binits{U.}}:
Optimal control for mathematical models of cancer therapies.
An application of geometric methods
(2015)
\end{botherref}
\endbibitem

%%% 35
\bibitem{jarrett2020optimal}
\begin{barticle}
\bauthor{\bsnm{Jarrett}, \binits{A.M.}},
\bauthor{\bsnm{Faghihi}, \binits{D.}},
\bauthor{\bsnm{Hormuth}, \binits{D.A.}},
\bauthor{\bsnm{Lima}, \binits{E.A.}},
\bauthor{\bsnm{Virostko}, \binits{J.}},
\bauthor{\bsnm{Biros}, \binits{G.}},
\bauthor{\bsnm{Patt}, \binits{D.}},
\bauthor{\bsnm{Yankeelov}, \binits{T.E.}}:
\batitle{Optimal control theory for personalized therapeutic regimens in
  oncology: Background, history, challenges, and opportunities}.
\bjtitle{Journal of clinical medicine}
\bvolume{9}(\bissue{5}),
\bfpage{1314}
(\byear{2020})
\end{barticle}
\endbibitem

%%% 36
\bibitem{yildiz2018new}
\begin{barticle}
\bauthor{\bsnm{Akman~Y{\i}ld{\i}z}, \binits{T.}},
\bauthor{\bsnm{Arshad}, \binits{S.}},
\bauthor{\bsnm{Baleanu}, \binits{D.}}:
\batitle{New observations on optimal cancer treatments for a fractional tumor
  growth model with and without singular kernel}.
\bjtitle{Chaos, Solitons \& Fractals}
\bvolume{117},
\bfpage{226}--\blpage{239}
(\byear{2018})
\end{barticle}
\endbibitem

%%% 37
\bibitem{akman2018optimal}
\begin{barticle}
\bauthor{\bsnm{Akman~Y{\i}ld{\i}z}, \binits{T.}},
\bauthor{\bsnm{Arshad}, \binits{S.}},
\bauthor{\bsnm{Baleanu}, \binits{D.}}:
\batitle{Optimal chemotherapy and immunotherapy schedules for a cancer-obesity
  model with {C}aputo time fractional derivative}.
\bjtitle{Mathematical Methods in the Applied Sciences}
\bvolume{41}(\bissue{18}),
\bfpage{9390}--\blpage{9407}
(\byear{2018})
\end{barticle}
\endbibitem

%%% 38
\bibitem{de2001mathematical}
\begin{barticle}
\bauthor{\bsnm{De~Pillis}, \binits{L.G.}},
\bauthor{\bsnm{Radunskaya}, \binits{A.}}:
\batitle{A mathematical tumor model with immune resistance and drug therapy: an
  optimal control approach}.
\bjtitle{Computational and Mathematical Methods in Medicine}
\bvolume{3}(\bissue{2}),
\bfpage{79}--\blpage{100}
(\byear{2001})
\end{barticle}
\endbibitem

%%% 39
\bibitem{panetta2003optimal}
\begin{barticle}
\bauthor{\bsnm{Panetta}, \binits{J.C.}},
\bauthor{\bsnm{Fister}, \binits{K.R.}}:
\batitle{Optimal control applied to competing chemotherapeutic cell-kill
  strategies}.
\bjtitle{SIAM Journal on Applied Mathematics}
\bvolume{63}(\bissue{6}),
\bfpage{1954}--\blpage{1971}
(\byear{2003})
\end{barticle}
\endbibitem

%%% 40
\bibitem{ledzewicz2022structure}
\begin{barticle}
\bauthor{\bsnm{Ledzewicz}, \binits{U.}},
\bauthor{\bsnm{Sch{\"a}ttler}, \binits{H.}}:
\batitle{The structure of optimal protocols for a mathematical model of
  chemotherapy with antiangiogenic effects}.
\bjtitle{SIAM Journal on Control and Optimization}
\bvolume{60}(\bissue{2}),
\bfpage{1092}--\blpage{1116}
(\byear{2022})
\end{barticle}
\endbibitem

%%% 41
\bibitem{ledzewicz2007antiangiogenic}
\begin{barticle}
\bauthor{\bsnm{Ledzewicz}, \binits{U.}},
\bauthor{\bsnm{Sch{\"a}ttler}, \binits{H.}}:
\batitle{Antiangiogenic therapy in cancer treatment as an optimal control
  problem}.
\bjtitle{SIAM Journal on Control and Optimization}
\bvolume{46}(\bissue{3}),
\bfpage{1052}--\blpage{1079}
(\byear{2007})
\end{barticle}
\endbibitem

%%% 42
\bibitem{colli2021optimal}
\begin{barticle}
\bauthor{\bsnm{Colli}, \binits{P.}},
\bauthor{\bsnm{Gomez}, \binits{H.}},
\bauthor{\bsnm{Lorenzo}, \binits{G.}},
\bauthor{\bsnm{Marinoschi}, \binits{G.}},
\bauthor{\bsnm{Reali}, \binits{A.}},
\bauthor{\bsnm{Rocca}, \binits{E.}}:
\batitle{Optimal control of cytotoxic and antiangiogenic therapies on prostate
  cancer growth}.
\bjtitle{Mathematical Models and Methods in Applied Sciences}
\bvolume{31}(\bissue{07}),
\bfpage{1419}--\blpage{1468}
(\byear{2021})
\end{barticle}
\endbibitem

%%% 43
\bibitem{castiglione2007cancer}
\begin{barticle}
\bauthor{\bsnm{Castiglione}, \binits{F.}},
\bauthor{\bsnm{Piccoli}, \binits{B.}}:
\batitle{Cancer immunotherapy, mathematical modeling and optimal control}.
\bjtitle{Journal of theoretical Biology}
\bvolume{247}(\bissue{4}),
\bfpage{723}--\blpage{732}
(\byear{2007})
\end{barticle}
\endbibitem

%%% 44
\bibitem{ledzewicz2012multi}
\begin{barticle}
\bauthor{\bsnm{Ledzewicz}, \binits{U.}},
\bauthor{\bsnm{Sch{\"a}ttler}, \binits{H.}}:
\batitle{Multi-input optimal control problems for combined tumor
  anti-angiogenic and radiotherapy treatments}.
\bjtitle{Journal of Optimization Theory and Applications}
\bvolume{153}(\bissue{1}),
\bfpage{195}--\blpage{224}
(\byear{2012})
\end{barticle}
\endbibitem

%%% 45
\bibitem{sharp2020designing}
\begin{barticle}
\bauthor{\bsnm{Sharp}, \binits{J.A.}},
\bauthor{\bsnm{Browning}, \binits{A.P.}},
\bauthor{\bsnm{Mapder}, \binits{T.}},
\bauthor{\bsnm{Baker}, \binits{C.M.}},
\bauthor{\bsnm{Burrage}, \binits{K.}},
\bauthor{\bsnm{Simpson}, \binits{M.J.}}:
\batitle{Designing combination therapies using multiple optimal controls}.
\bjtitle{Journal of Theoretical Biology}
\bvolume{497},
\bfpage{110277}
(\byear{2020})
\end{barticle}
\endbibitem

%%% 46
\bibitem{costa1992optimal}
\begin{barticle}
\bauthor{\bsnm{Costa}, \binits{M.}},
\bauthor{\bsnm{Boldrini}, \binits{J.}},
\bauthor{\bsnm{Bassanezi}, \binits{R.}}:
\batitle{Optimal chemical control of populations developing drug resistance}.
\bjtitle{Mathematical Medicine and Biology: A Journal of the IMA}
\bvolume{9}(\bissue{3}),
\bfpage{215}--\blpage{226}
(\byear{1992})
\end{barticle}
\endbibitem

%%% 47
\bibitem{carrere2017optimization}
\begin{barticle}
\bauthor{\bsnm{Carrere}, \binits{C.}}:
\batitle{Optimization of an in vitro chemotherapy to avoid resistant tumours}.
\bjtitle{Journal of Theoretical Biology}
\bvolume{413},
\bfpage{24}--\blpage{33}
(\byear{2017})
\end{barticle}
\endbibitem

%%% 48
\bibitem{oke2018optimal}
\begin{barticle}
\bauthor{\bsnm{Oke}, \binits{S.I.}},
\bauthor{\bsnm{Matadi}, \binits{M.B.}},
\bauthor{\bsnm{Xulu}, \binits{S.S.}}:
\batitle{Optimal control analysis of a mathematical model for breast cancer}.
\bjtitle{Mathematical and Computational Applications}
\bvolume{23}(\bissue{2}),
\bfpage{21}
(\byear{2018})
\end{barticle}
\endbibitem

%%% 49
\bibitem{lima2022optimizing}
\begin{botherref}
\oauthor{\bsnm{Lima}, \binits{E.A.}},
\oauthor{\bsnm{Wyde}, \binits{R.A.}},
\oauthor{\bsnm{Sorace}, \binits{A.G.}},
\oauthor{\bsnm{Yankeelov}, \binits{T.E.}}:
Optimizing combination therapy in a murine model of {HER2}+ breast cancer.
Computer Methods in Applied Mechanics and Engineering,
115484
(2022)
\end{botherref}
\endbibitem

%%% 50
\bibitem{hillers2018obesity}
\begin{barticle}
\bauthor{\bsnm{Hillers}, \binits{L.E.}},
\bauthor{\bsnm{D'Amato}, \binits{J.V.}},
\bauthor{\bsnm{Chamberlin}, \binits{T.}},
\bauthor{\bsnm{Paderta}, \binits{G.}},
\bauthor{\bsnm{Arendt}, \binits{L.M.}}:
\batitle{Obesity-activated adipose-derived stromal cells promote breast cancer
  growth and invasion}.
\bjtitle{Neoplasia}
\bvolume{20}(\bissue{11}),
\bfpage{1161}--\blpage{1174}
(\byear{2018})
\end{barticle}
\endbibitem

%%% 51
\bibitem{le2020eo771}
\begin{barticle}
\bauthor{\bsnm{Le~Naour}, \binits{A.}},
\bauthor{\bsnm{Koffi}, \binits{Y.}},
\bauthor{\bsnm{Diab}, \binits{M.}},
\bauthor{\bsnm{Le~Guennec}, \binits{D.}},
\bauthor{\bsnm{Roug{\'e}}, \binits{S.}},
\bauthor{\bsnm{Aldekwer}, \binits{S.}},
\bauthor{\bsnm{Goncalves-Mendes}, \binits{N.}},
\bauthor{\bsnm{Talvas}, \binits{J.}},
\bauthor{\bsnm{Farges}, \binits{M.-C.}},
\bauthor{\bsnm{Caldefie-Chezet}, \binits{F.}}, \betal:
\batitle{{EO}771, the first luminal {B} mammary cancer cell line from {C57BL/6}
  mice}.
\bjtitle{Cancer cell international}
\bvolume{20}(\bissue{1}),
\bfpage{1}--\blpage{13}
(\byear{2020})
\end{barticle}
\endbibitem

%%% 52
\bibitem{benzekry2014classical}
\begin{barticle}
\bauthor{\bsnm{Benzekry}, \binits{S.}},
\bauthor{\bsnm{Lamont}, \binits{C.}},
\bauthor{\bsnm{Beheshti}, \binits{A.}},
\bauthor{\bsnm{Tracz}, \binits{A.}},
\bauthor{\bsnm{Ebos}, \binits{J.M.}},
\bauthor{\bsnm{Hlatky}, \binits{L.}},
\bauthor{\bsnm{Hahnfeldt}, \binits{P.}}:
\batitle{Classical mathematical models for description and prediction of
  experimental tumor growth}.
\bjtitle{PLoS computational biology}
\bvolume{10}(\bissue{8}),
\bfpage{1003800}
(\byear{2014})
\end{barticle}
\endbibitem

%%% 53
\bibitem{simpson2003sources}
\begin{barticle}
\bauthor{\bsnm{Simpson}, \binits{E.R.}}:
\batitle{Sources of estrogen and their importance}.
\bjtitle{The Journal of steroid biochemistry and molecular biology}
\bvolume{86}(\bissue{3-5}),
\bfpage{225}--\blpage{230}
(\byear{2003})
\end{barticle}
\endbibitem

%%% 54
\bibitem{marchand2018increased}
\begin{barticle}
\bauthor{\bsnm{Marchand}, \binits{G.B.}},
\bauthor{\bsnm{Carreau}, \binits{A.-M.}},
\bauthor{\bsnm{Weisnagel}, \binits{S.J.}},
\bauthor{\bsnm{Bergeron}, \binits{J.}},
\bauthor{\bsnm{Labrie}, \binits{F.}},
\bauthor{\bsnm{Lemieux}, \binits{S.}},
\bauthor{\bsnm{Tchernof}, \binits{A.}}:
\batitle{Increased body fat mass explains the positive association between
  circulating estradiol and insulin resistance in postmenopausal women}.
\bjtitle{American Journal of Physiology-Endocrinology and Metabolism}
\bvolume{314}(\bissue{5}),
\bfpage{448}--\blpage{456}
(\byear{2018})
\end{barticle}
\endbibitem

%%% 55
\bibitem{deshpande1967}
\begin{barticle}
\bauthor{\bsnm{Deshpande}, \binits{N.}},
\bauthor{\bsnm{Jensen}, \binits{V.}},
\bauthor{\bsnm{Bulbrook}, \binits{R.}},
\bauthor{\bsnm{Berne}, \binits{T.}},
\bauthor{\bsnm{Ellis}, \binits{F.}}:
\batitle{Accumulation of tritiated oestradiol by human breast tissue}.
\bjtitle{Steroids}
\bvolume{10}(\bissue{3}),
\bfpage{219}--\blpage{232}
(\byear{1967})
\end{barticle}
\endbibitem

%%% 56
\bibitem{hoy2017adipocyte}
\begin{barticle}
\bauthor{\bsnm{Hoy}, \binits{A.J.}},
\bauthor{\bsnm{Balaban}, \binits{S.}},
\bauthor{\bsnm{Saunders}, \binits{D.N.}}:
\batitle{Adipocyte--tumor cell metabolic crosstalk in breast cancer}.
\bjtitle{Trends in molecular medicine}
\bvolume{23}(\bissue{5}),
\bfpage{381}--\blpage{392}
(\byear{2017})
\end{barticle}
\endbibitem

%%% 57
\bibitem{wang2017mammary}
\begin{botherref}
\oauthor{\bsnm{Wang}, \binits{Y.Y.}},
\oauthor{\bsnm{Attan{\'e}}, \binits{C.}},
\oauthor{\bsnm{Milhas}, \binits{D.}},
\oauthor{\bsnm{Dirat}, \binits{B.}},
\oauthor{\bsnm{Dauvillier}, \binits{S.}},
\oauthor{\bsnm{Guerard}, \binits{A.}},
\oauthor{\bsnm{Gilhodes}, \binits{J.}},
\oauthor{\bsnm{Lazar}, \binits{I.}},
\oauthor{\bsnm{Alet}, \binits{N.}},
\oauthor{\bsnm{Laurent}, \binits{V.}}, et al.:
Mammary adipocytes stimulate breast cancer invasion through metabolic
  remodeling of tumor cells.
JCI insight
\textbf{2}(4)
(2017)
\end{botherref}
\endbibitem

%%% 58
\bibitem{schatzman2002numerical}
\begin{bbook}
\bauthor{\bsnm{Schatzman}, \binits{M.}}:
\bbtitle{Numerical Analysis: a Mathematical Introduction}.
\bpublisher{Oxford University Press},
\blocation{Oxford}
(\byear{2002})
\end{bbook}
\endbibitem

%%% 59
\bibitem{yue1999aromatase}
\begin{barticle}
\bauthor{\bsnm{Yue}, \binits{W.}},
\bauthor{\bsnm{Santen}, \binits{R.}},
\bauthor{\bsnm{Wang}, \binits{J.}},
\bauthor{\bsnm{Hamilton}, \binits{C.}},
\bauthor{\bsnm{Demers}, \binits{L.}}:
\batitle{Aromatase within the breast.}
\bjtitle{Endocrine-Related Cancer}
\bvolume{6}(\bissue{2}),
\bfpage{157}--\blpage{164}
(\byear{1999})
\end{barticle}
\endbibitem

%%% 60
\bibitem{raue2015data2dynamics}
\begin{barticle}
\bauthor{\bsnm{Raue}, \binits{A.}},
\bauthor{\bsnm{Steiert}, \binits{B.}},
\bauthor{\bsnm{Schelker}, \binits{M.}},
\bauthor{\bsnm{Kreutz}, \binits{C.}},
\bauthor{\bsnm{Maiwald}, \binits{T.}},
\bauthor{\bsnm{Hass}, \binits{H.}},
\bauthor{\bsnm{Vanlier}, \binits{J.}},
\bauthor{\bsnm{T{\"o}nsing}, \binits{C.}},
\bauthor{\bsnm{Adlung}, \binits{L.}},
\bauthor{\bsnm{Engesser}, \binits{R.}}, \betal:
\batitle{Data2dynamics: a modeling environment tailored to parameter estimation
  in dynamical systems}.
\bjtitle{Bioinformatics}
\bvolume{31}(\bissue{21}),
\bfpage{3558}--\blpage{3560}
(\byear{2015})
\end{barticle}
\endbibitem

%%% 61
\bibitem{raue2013lessons}
\begin{barticle}
\bauthor{\bsnm{Raue}, \binits{A.}},
\bauthor{\bsnm{Schilling}, \binits{M.}},
\bauthor{\bsnm{Bachmann}, \binits{J.}},
\bauthor{\bsnm{Matteson}, \binits{A.}},
\bauthor{\bsnm{Schelke}, \binits{M.}},
\bauthor{\bsnm{Kaschek}, \binits{D.}},
\bauthor{\bsnm{Hug}, \binits{S.}},
\bauthor{\bsnm{Kreutz}, \binits{C.}},
\bauthor{\bsnm{Harms}, \binits{B.D.}},
\bauthor{\bsnm{Theis}, \binits{F.J.}}, \betal:
\batitle{Lessons learned from quantitative dynamical modeling in systems
  biology}.
\bjtitle{PloS one}
\bvolume{8}(\bissue{9}),
\bfpage{74335}
(\byear{2013})
\end{barticle}
\endbibitem

%%% 62
\bibitem{kreutz2012likelihood}
\begin{barticle}
\bauthor{\bsnm{Kreutz}, \binits{C.}},
\bauthor{\bsnm{Raue}, \binits{A.}},
\bauthor{\bsnm{Timmer}, \binits{J.}}:
\batitle{Likelihood based observability analysis and confidence intervals for
  predictions of dynamic models}.
\bjtitle{BMC Systems Biology}
\bvolume{6}(\bissue{1}),
\bfpage{1}--\blpage{9}
(\byear{2012})
\end{barticle}
\endbibitem

%%% 63
\bibitem{chumsri2011aromatase}
\begin{barticle}
\bauthor{\bsnm{Chumsri}, \binits{S.}},
\bauthor{\bsnm{Howes}, \binits{T.}},
\bauthor{\bsnm{Bao}, \binits{T.}},
\bauthor{\bsnm{Sabnis}, \binits{G.}},
\bauthor{\bsnm{Brodie}, \binits{A.}}:
\batitle{Aromatase, aromatase inhibitors, and breast cancer}.
\bjtitle{The Journal of steroid biochemistry and molecular biology}
\bvolume{125}(\bissue{1-2}),
\bfpage{13}--\blpage{22}
(\byear{2011})
\end{barticle}
\endbibitem

%%% 64
\bibitem{normanno2005mechanisms}
\begin{barticle}
\bauthor{\bsnm{Normanno}, \binits{N.}},
\bauthor{\bsnm{Di~Maio}, \binits{M.}},
\bauthor{\bsnm{De~Maio}, \binits{E.}},
\bauthor{\bsnm{De~Luca}, \binits{A.}},
\bauthor{\bsnm{De~Matteis}, \binits{A.}},
\bauthor{\bsnm{Giordano}, \binits{A.}},
\bauthor{\bsnm{Perrone}, \binits{F.}}:
\batitle{Mechanisms of endocrine resistance and novel therapeutic strategies in
  breast cancer}.
\bjtitle{Endocrine-related cancer}
\bvolume{12}(\bissue{4}),
\bfpage{721}--\blpage{747}
(\byear{2005})
\end{barticle}
\endbibitem

%%% 65
\bibitem{doisneau2003estrogen}
\begin{barticle}
\bauthor{\bsnm{Doisneau-Sixou}, \binits{S.}},
\bauthor{\bsnm{Sergio}, \binits{C.}},
\bauthor{\bsnm{Carroll}, \binits{J.}},
\bauthor{\bsnm{Hui}, \binits{R.}},
\bauthor{\bsnm{Musgrove}, \binits{E.}},
\bauthor{\bsnm{Sutherland}, \binits{R.}}:
\batitle{Estrogen and antiestrogen regulation of cell cycle progression in
  breast cancer cells.}
\bjtitle{Endocrine-related cancer}
\bvolume{10}(\bissue{2}),
\bfpage{179}--\blpage{186}
(\byear{2003})
\end{barticle}
\endbibitem

%%% 66
\bibitem{ku2016mathematical}
\begin{barticle}
\bauthor{\bsnm{Ku-Carrillo}, \binits{R.A.}},
\bauthor{\bsnm{Delgadillo}, \binits{S.E.}},
\bauthor{\bsnm{Chen-Charpentier}, \binits{B.}}:
\batitle{A mathematical model for the effect of obesity on cancer growth and on
  the immune system response}.
\bjtitle{Applied Mathematical Modelling}
\bvolume{40}(\bissue{7-8}),
\bfpage{4908}--\blpage{4920}
(\byear{2016})
\end{barticle}
\endbibitem

%%% 67
\bibitem{de2008optimal}
\begin{barticle}
\bauthor{\bparticle{de} \bsnm{Pillis}, \binits{L.G.}},
\bauthor{\bsnm{Fister}, \binits{K.R.}},
\bauthor{\bsnm{Gu}, \binits{W.}},
\bauthor{\bsnm{Head}, \binits{T.}},
\bauthor{\bsnm{Maples}, \binits{K.}},
\bauthor{\bsnm{Neal}, \binits{T.}},
\bauthor{\bsnm{Murugan}, \binits{A.}},
\bauthor{\bsnm{Kozai}, \binits{K.}}:
\batitle{Optimal control of mixed immunotherapy and chemotherapy of tumors}.
\bjtitle{Journal of Biological systems}
\bvolume{16}(\bissue{01}),
\bfpage{51}--\blpage{80}
(\byear{2008})
\end{barticle}
\endbibitem

%%% 68
\bibitem{sharma2016analysis}
\begin{barticle}
\bauthor{\bsnm{Sharma}, \binits{S.}},
\bauthor{\bsnm{Samanta}, \binits{G.}}:
\batitle{Analysis of the dynamics of a tumor--immune system with chemotherapy
  and immunotherapy and quadratic optimal control}.
\bjtitle{Differential Equations and Dynamical Systems}
\bvolume{24}(\bissue{2}),
\bfpage{149}--\blpage{171}
(\byear{2016})
\end{barticle}
\endbibitem

%%% 69
\bibitem{ledzewicz2004comparison}
\begin{barticle}
\bauthor{\bsnm{Ledzewicz}, \binits{U.}},
\bauthor{\bsnm{Brown}, \binits{T.}},
\bauthor{\bsnm{Sch{\"a}ttler}, \binits{H.}}:
\batitle{Comparison of optimal controls for a model in cancer chemotherapy with
  {L}1 and {L}2-type objectives}.
\bjtitle{Optimization Methods and Software}
\bvolume{19}(\bissue{3-4}),
\bfpage{339}--\blpage{350}
(\byear{2004})
\end{barticle}
\endbibitem

%%% 70
\bibitem{sharp2019optimal}
\begin{barticle}
\bauthor{\bsnm{Sharp}, \binits{J.A.}},
\bauthor{\bsnm{Browning}, \binits{A.P.}},
\bauthor{\bsnm{Mapder}, \binits{T.}},
\bauthor{\bsnm{Burrage}, \binits{K.}},
\bauthor{\bsnm{Simpson}, \binits{M.J.}}:
\batitle{Optimal control of acute myeloid leukaemia}.
\bjtitle{Journal of theoretical biology}
\bvolume{470},
\bfpage{30}--\blpage{42}
(\byear{2019})
\end{barticle}
\endbibitem

%%% 71
\bibitem{ledzewicz2020role}
\begin{barticle}
\bauthor{\bsnm{Ledzewicz}, \binits{U.}},
\bauthor{\bsnm{Sch{\"a}ttler}, \binits{H.}}:
\batitle{On the role of the objective in the optimization of compartmental
  models for biomedical therapies}.
\bjtitle{Journal of optimization theory and applications}
\bvolume{187}(\bissue{2}),
\bfpage{305}--\blpage{335}
(\byear{2020})
\end{barticle}
\endbibitem

%%% 72
\bibitem{osborne2005aromatase}
\begin{barticle}
\bauthor{\bsnm{Osborne}, \binits{C.}},
\bauthor{\bsnm{Tripathy}, \binits{D.}}:
\batitle{Aromatase inhibitors: rationale and use in breast cancer}.
\bjtitle{Annual review of medicine}
\bvolume{56},
\bfpage{103}
(\byear{2005})
\end{barticle}
\endbibitem

%%% 73
\bibitem{cuzick2005aromatase}
\begin{barticle}
\bauthor{\bsnm{Cuzick}, \binits{J.}}:
\batitle{Aromatase inhibitors for breast cancer prevention}.
\bjtitle{Journal of Clinical Oncology}
\bvolume{23}(\bissue{8}),
\bfpage{1636}--\blpage{1643}
(\byear{2005})
\end{barticle}
\endbibitem

%%% 74
\bibitem{hadji2010guidelines}
\begin{barticle}
\bauthor{\bsnm{Hadji}, \binits{P.}}:
\batitle{Guidelines for osteoprotection in breast cancer patients on an
  aromatase inhibitor}.
\bjtitle{Breast Care}
\bvolume{5}(\bissue{5}),
\bfpage{290}--\blpage{296}
(\byear{2010})
\end{barticle}
\endbibitem

%%% 75
\bibitem{fleming2012deterministic}
\begin{bbook}
\bauthor{\bsnm{Fleming}, \binits{W.H.}},
\bauthor{\bsnm{Rishel}, \binits{R.W.}}:
\bbtitle{Deterministic and Stochastic Optimal Control}
vol. \bseriesno{1}.
\bpublisher{Springer},
\blocation{Berlin, Heidelberg, New York}
(\byear{1975})
\end{bbook}
\endbibitem

%%% 76
\bibitem{lukes1982differential}
\begin{botherref}
\oauthor{\bsnm{Lukes}, \binits{D.L.}}:
Differential equations: classical to controlled
\textbf{162}
(1982)
\end{botherref}
\endbibitem

%%% 77
\bibitem{burden2004optimal}
\begin{barticle}
\bauthor{\bsnm{Burden}, \binits{T.}},
\bauthor{\bsnm{Ernstberger}, \binits{J.}},
\bauthor{\bsnm{Fister}, \binits{K.R.}}:
\batitle{Optimal control applied to immunotherapy}.
\bjtitle{Discrete \& Continuous Dynamical Systems-B}
\bvolume{4}(\bissue{1}),
\bfpage{135}
(\byear{2004})
\end{barticle}
\endbibitem

%%% 78
\bibitem{geisler2005aromatase}
\begin{barticle}
\bauthor{\bsnm{Geisler}, \binits{J.}},
\bauthor{\bsnm{L{\o}nning}, \binits{P.E.}}:
\batitle{Aromatase inhibition: translation into a successful therapeutic
  approach}.
\bjtitle{Clinical Cancer Research}
\bvolume{11}(\bissue{8}),
\bfpage{2809}--\blpage{2821}
(\byear{2005})
\end{barticle}
\endbibitem

%%% 79
\bibitem{sasano2009situ}
\begin{barticle}
\bauthor{\bsnm{Sasano}, \binits{H.}},
\bauthor{\bsnm{Miki}, \binits{Y.}},
\bauthor{\bsnm{Nagasaki}, \binits{S.}},
\bauthor{\bsnm{Suzuki}, \binits{T.}}:
\batitle{In situ estrogen production and its regulation in human breast
  carcinoma: from endocrinology to intracrinology}.
\bjtitle{Pathology international}
\bvolume{59}(\bissue{11}),
\bfpage{777}--\blpage{789}
(\byear{2009})
\end{barticle}
\endbibitem

%%% 80
\bibitem{geisler2003breast}
\begin{barticle}
\bauthor{\bsnm{Geisler}, \binits{J.}}:
\batitle{Breast cancer tissue estrogens and their manipulation with aromatase
  inhibitors and inactivators}.
\bjtitle{The Journal of steroid biochemistry and molecular biology}
\bvolume{86}(\bissue{3-5}),
\bfpage{245}--\blpage{253}
(\byear{2003})
\end{barticle}
\endbibitem

%%% 81
\bibitem{lenhart2007optimal}
\begin{bbook}
\bauthor{\bsnm{Lenhart}, \binits{S.}},
\bauthor{\bsnm{Workman}, \binits{J.T.}}:
\bbtitle{Optimal Control Applied to Biological Models}.
\bpublisher{Chapman and Hall/CRC},
\blocation{New York}
(\byear{2007})
\end{bbook}
\endbibitem

%%% 82
\bibitem{vatcheva2021social}
\begin{barticle}
\bauthor{\bsnm{Vatcheva}, \binits{K.P.}},
\bauthor{\bsnm{Sifuentes}, \binits{J.}},
\bauthor{\bsnm{Oraby}, \binits{T.}},
\bauthor{\bsnm{Maldonado}, \binits{J.C.}},
\bauthor{\bsnm{Huber}, \binits{T.}},
\bauthor{\bsnm{Villalobos}, \binits{M.C.}}:
\batitle{Social distancing and testing as optimal strategies against the spread
  of {COVID}-19 in the {R}io {G}rande {V}alley of {T}exas}.
\bjtitle{Infectious Disease Modelling}
\bvolume{6},
\bfpage{729}--\blpage{742}
(\byear{2021})
\end{barticle}
\endbibitem

%%% 83
\bibitem{MATLAB:2022}
\begin{bbook}
\bauthor{\bsnm{MATLAB}}:
\bbtitle{9.13.0.2080170 (R2022b) Update 1}.
\bpublisher{The MathWorks Inc.},
\blocation{Natick, Massachusetts}
(\byear{2022})
\end{bbook}
\endbibitem

%%% 84
\bibitem{fister1998optimizing}
\begin{barticle}
\bauthor{\bsnm{Fister}, \binits{K.R.}},
\bauthor{\bsnm{Lenhart}, \binits{S.}},
\bauthor{\bsnm{McNally}, \binits{J.S.}}:
\batitle{Optimizing chemotherapy in an hiv model}.
\bjtitle{Electronic Journal of Differential Equations}
\bvolume{1998}(\bissue{32}),
\bfpage{1}--\blpage{12}
(\byear{1998})
\end{barticle}
\endbibitem

%%% 85
\bibitem{hillers2022breast}
\begin{botherref}
\oauthor{\bsnm{Hillers-Ziemer}, \binits{L.E.}},
\oauthor{\bsnm{Kuziel}, \binits{G.}},
\oauthor{\bsnm{Williams}, \binits{A.E.}},
\oauthor{\bsnm{Moore}, \binits{B.N.}},
\oauthor{\bsnm{Arendt}, \binits{L.M.}}:
Breast cancer microenvironment and obesity: challenges for therapy.
Cancer and Metastasis Reviews,
1--21
(2022)
\end{botherref}
\endbibitem

%%% 86
\bibitem{Kohn-Luque2022}
\begin{barticle}
\bauthor{\bsnm{Köhn-Luque}, \binits{A.}},
\bauthor{\bsnm{Myklebust}, \binits{E.M.}},
\bauthor{\bsnm{Tadele}, \binits{D.S.}},
\bauthor{\bsnm{Giliberto}, \binits{M.}},
\bauthor{\bsnm{Schmiester}, \binits{L.}},
\bauthor{\bsnm{Noory}, \binits{J.}},
\bauthor{\bsnm{Harivel}, \binits{E.}},
\bauthor{\bsnm{Arsenteva}, \binits{P.}},
\bauthor{\bsnm{Mumenthaler}, \binits{S.M.}},
\bauthor{\bsnm{Schjesvold}, \binits{F.}},
\bauthor{\bsnm{Taskén}, \binits{K.}},
\bauthor{\bsnm{Enserink}, \binits{J.M.}},
\bauthor{\bsnm{Leder}, \binits{K.}},
\bauthor{\bsnm{Frigessi}, \binits{A.}},
\bauthor{\bsnm{Foo}, \binits{J.}}:
\batitle{Phenotypic deconvolution in heterogeneous cancer cell populations
  using drug-screening data}.
\bjtitle{Cell Reports Methods}
\bvolume{3}(\bissue{3}),
\bfpage{100417}
(\byear{2023}).
\doiurl{10.1016/j.crmeth.2023.100417}
\end{barticle}
\endbibitem

%%% 87
\bibitem{costa2004high}
\begin{barticle}
\bauthor{\bsnm{Costa}, \binits{I.}},
\bauthor{\bsnm{Moral}, \binits{R.}},
\bauthor{\bsnm{Solanas}, \binits{M.}},
\bauthor{\bsnm{Escrich}, \binits{E.}}:
\batitle{High-fat corn oil diet promotes the development of high histologic
  grade rat {DMBA}-induced mammary adenocarcinomas, while high olive oil diet
  does not}.
\bjtitle{Breast cancer research and treatment}
\bvolume{86}(\bissue{3}),
\bfpage{225}--\blpage{235}
(\byear{2004})
\end{barticle}
\endbibitem

%%% 88
\bibitem{solanas2009differential}
\begin{botherref}
\oauthor{\bsnm{Solanas}, \binits{M.}},
\oauthor{\bsnm{Moral}, \binits{R.}},
\oauthor{\bsnm{Garcia}, \binits{G.}},
\oauthor{\bsnm{Grau}, \binits{L.}},
\oauthor{\bsnm{Vela}, \binits{E.}},
\oauthor{\bsnm{Escrich}, \binits{R.}},
\oauthor{\bsnm{Costa}, \binits{I.}},
\oauthor{\bsnm{Escrich}, \binits{E.}}:
Differential expression of h19 and vitamin d3 upregulated protein 1 as a
  mechanism of the modulatory effects of high virgin olive oil and high corn
  oil diets on experimental mammary tumours.
European Journal of Cancer Prevention,
153--161
(2009)
\end{botherref}
\endbibitem

%%% 89
\bibitem{bahrami2019neoletexe}
\begin{barticle}
\bauthor{\bsnm{Bahrami}, \binits{N.}},
\bauthor{\bsnm{Sauer}, \binits{T.}},
\bauthor{\bsnm{Engebretsen}, \binits{S.}},
\bauthor{\bsnm{Aljabri}, \binits{B.}},
\bauthor{\bsnm{Bemanian}, \binits{V.}},
\bauthor{\bsnm{Lindstr{\o}m}, \binits{J.}},
\bauthor{\bsnm{L{\"u}ders}, \binits{T.}},
\bauthor{\bsnm{Kristensen}, \binits{V.}},
\bauthor{\bsnm{Lorentzen}, \binits{A.}},
\bauthor{\bsnm{Loeng}, \binits{M.}}, \betal:
\batitle{The {NEOLETEXE} trial: a neoadjuvant cross-over study exploring the
  lack of cross resistance between aromatase inhibitors}.
\bjtitle{Future Oncology}
\bvolume{15}(\bissue{32}),
\bfpage{3675}--\blpage{3682}
(\byear{2019})
\end{barticle}
\endbibitem

%%% 90
\bibitem{jo2009hypertrophy}
\begin{barticle}
\bauthor{\bsnm{Jo}, \binits{J.}},
\bauthor{\bsnm{Gavrilova}, \binits{O.}},
\bauthor{\bsnm{Pack}, \binits{S.}},
\bauthor{\bsnm{Jou}, \binits{W.}},
\bauthor{\bsnm{Mullen}, \binits{S.}},
\bauthor{\bsnm{Sumner}, \binits{A.E.}},
\bauthor{\bsnm{Cushman}, \binits{S.W.}},
\bauthor{\bsnm{Periwal}, \binits{V.}}:
\batitle{Hypertrophy and/or hyperplasia: dynamics of adipose tissue growth}.
\bjtitle{PLoS computational biology}
\bvolume{5}(\bissue{3}),
\bfpage{1000324}
(\byear{2009})
\end{barticle}
\endbibitem

%%% 91
\bibitem{bozec2016mechanism}
\begin{botherref}
\oauthor{\bsnm{Bozec}, \binits{A.}},
\oauthor{\bsnm{Hannemann}, \binits{N.}}:
Mechanism of regulation of adipocyte numbers in adult organisms through
  differentiation and apoptosis homeostasis.
JoVE (Journal of Visualized Experiments)
(112),
53822
(2016)
\end{botherref}
\endbibitem

%%% 92
\bibitem{marino2008methodology}
\begin{barticle}
\bauthor{\bsnm{Marino}, \binits{S.}},
\bauthor{\bsnm{Hogue}, \binits{I.B.}},
\bauthor{\bsnm{Ray}, \binits{C.J.}},
\bauthor{\bsnm{Kirschner}, \binits{D.E.}}:
\batitle{A methodology for performing global uncertainty and sensitivity
  analysis in systems biology}.
\bjtitle{Journal of theoretical biology}
\bvolume{254}(\bissue{1}),
\bfpage{178}--\blpage{196}
(\byear{2008})
\end{barticle}
\endbibitem

%%% 93
\bibitem{PRCCcode}
\begin{botherref}
\oauthor{\bsnm{Lab}, \binits{K.}}:
Our Approach to Uncertainty and Sensitivity Analysis (with R and MATLAB Codes
  for Use).
(2020 (accessed January 12, 2020)).
\url{http://malthus.micro.med.umich.edu/lab/usanalysis.html}
\end{botherref}
\endbibitem

\end{thebibliography}




\end{document}